\documentclass{amsart}
\usepackage{epsfig}
\usepackage{amsthm,amsfonts}
\usepackage{amssymb,graphicx,color}

\newtheorem{Teo}{Theorem}[section]
\newtheorem{Cor}[Teo]{Corollary}
\newtheorem{Prop}[Teo]{Proposition}
\newtheorem{Def}[Teo]{Definition}
\newtheorem{Lem}[Teo]{Lemma}
\newtheorem{Exam}[Teo]{Example}
\newtheorem{remark}[Teo]{Remark}

\newcommand{\C}{\mathbb C}
\newcommand{\R}{\mathbb R}

\newcommand{\Z}{\mathbb Z}
\newcommand{\F}{\mathbb F}

\newcommand{\E}{\mathcal{E}}
\newcommand{\N}{\mathcal{N}}
\newcommand{\T}{\mathcal{T}}
\newcommand{\I}{\mathcal{I}}
\newcommand{\W}{\mathcal W}

\begin{document}

\title[Lipschitz geometry and combinatorics of abnormal surface germs
 ]{Lipschitz geometry and combinatorics of abnormal surface germs
}

\author[]{Andrei Gabrielov$\dagger$}\thanks{$\dagger$ Research supported by the NSF grant DMS-1665115}
\address{Department of Mathematics, Purdue University,
	West Lafayette, IN 47907, USA}\email{gabrielov@purdue.edu}
\author[]{Emanoel Souza*}\thanks{*Research supported by grant Print/CAPES/UFC N 88887.312005/2018-00}
\address{Departamento de Matem\'atica, Universidade Federal do Cear\'a
(UFC), Campus do Pici, Bloco 914, Cep. 60455-760. Fortaleza-Ce,
Brasil}\email{emanoelfs.cdd@gmail.com}

\keywords{Lipschitz geometry, Surface singularities, Catalan numbers}

\subjclass{51F30, 14P10, 03C64, 05A17}

\begin{abstract}
We study outer Lipschitz geometry of real semialgebraic or, more general, definable in a polynomially bounded o-minimal structure over the reals, surface germs. In particular, any definable H\"older triangle is either Lipschitz normally embedded or contains some ``abnormal'' arcs. We show that abnormal arcs constitute finitely many ``abnormal zones'' in the space of all arcs, and investigate geometric and combinatorial properties of abnormal surface germs. We establish a strong relation between geometry and combinatorics of abnormal H\"older triangles.
\end{abstract}

\maketitle

\section{Introduction}
This paper explores Lipschitz geometry of germs of semialgebraic (or, more general, definable in a polynomially bounded o-minimal structure) real surfaces, with the goal towards effective bi-Lipschitz classification of definable surface singularities.

Lipschitz geometry of singularities attracted considerable attention for the last 50 years, as a natural approach to classification of singularities which is intermediate between their bi-regular (too fine) and topological (too coarse) equivalence.
In particular, the finiteness theorems of Mostowski \cite{Mostowski} and Parusinski \cite{Parusinski} suggest the possibility of
effective bi-Lipschitz classification of definable real surface germs.

In the seminal paper of Pham and Teissier \cite{Pham} on Lipschitz geometry of germs of complex plane algebraic curves, it was shown that two such germs are (outer metric) bi-Lipschitz equivalent exactly when they are ambient topologically equivalent, thus outer bi-Lipschitz equivalence class of such germs is completely determined by the essential Puiseux pairs of their irreducible branches, and by the orders of contact between the branches.

Later it became clear (see \cite{birbrair2000normal}) that any singular germ $X$ inherits two metrics from the ambient space: the inner metric where the distance between two points of $X$ is the length of the shortest path connecting them inside $X$, and the outer metric with the distance between two points of $X$ being just their distance in the ambient space.
This defines two classification problems, equivalence up to bi-Lipschitz homeomorphisms with respect to the inner and outer metrics,
the inner metric classification being more coarse than the outer metric one.

Any semialgebraic surface germ with a link homeomorphic to a line segment is inner bi-Lipschitz equivalent to the standard $\beta$-H\"older triangle $\{0\le x\le 1,\;0\le y\le x^\beta\}$.
Any semialgebraic surface germ with an isolated singularity and connected link is inner bi-Lipschitz equivalent to the germ of a $\beta$-horn -- surface of revolution of a $\beta$-cusp $\{0\le x\le 1,\;y=x^\beta\}$.
For the Lipschitz normally embedded singularities, the inner and outer metrics are equivalent, thus the two classifications are the same.
Kurdyka \cite{kurdyka1997distance} proved that any semialgebraic set can be decomposed into finitely many normally embedded semialgebraic sets.
Birbrair and Mostowski \cite{birbrair2000normal} used Kurdyka's construction to prove that any semialgebraic set is inner Lipschitz equivalent to a normally embedded semialgebraic set.

Classification of surface germs with respect to the outer metric is much more complicated.
A singular germ $X$ can be considered as the family $\{X_t\}$ of its links (intersections with the spheres of a small radius $t>0$).
 Thus Lipschitz geometry of $X$ can be understood as the dynamics of $X_t$ as $t\to 0$.
 For this purpose we investigate the ``Valette link'' of $X$, the family of arcs in $X$ parameterized by the distance to the origin.
 The outer Lipschitz invariants of $X$ are described in terms of the tangency orders between those arcs.

The first step towards the outer metric classification of surface germs was made in \cite{birbrair2014lipschitz} for the
surfaces of a very special kind, each of them being the union of the real plane and a graph of a function defined on that plane.
In the classical singularity theory, this corresponds to classification of functions with respect to bi-Lipschitz $K$-equivalence.

The present paper is the next step towards outer metric classification of surface germs.
Using Kurdyka's ``pancake decomposition'' of a germ into normally embedded subsets, and the ``pizza decomposition'' from \cite{birbrair2014lipschitz} for the distance functions, we identify basic ``abnormal'' parts of a surface germ, called snakes, and investigate their geometric and combinatorial properties.

In Section \ref{Section: Basic Definitions} we review some standard (and some less standard) definitions and technical tools of Lipschitz geometry of surface germs. The standard metric of $\R^n$ induces two metrics on $X$: the outer and inner metrics.
The distance between two points $x$ and $y$ of $X$ in the outer metric is just the distance $|x-y|$ between them in $\R^n$, while the distance in the inner metric is the infimum of the lengths of definable paths connecting $x$ and $y$ inside $X$.
A surface $X$ is Lipschitz normally embedded if these two metrics are equivalent.
An arc $\gamma\subset X$ is the germ of a definable mapping $[0,\epsilon)\to X$ such that $|\gamma(t)|=t$.
The outer (resp., inner) tangency order of two arcs $\gamma$ and $\gamma'$ is the exponent of the distance between $\gamma(t)$ and $\gamma'(t)$ in the outer (resp., inner) metric. This equips the set of all arcs in $X$ (known as the Valette link $V(X)$ of $X$, see \cite{valette2007link}) with a non-archimedean metric.
The simplest surface germ is a $\beta$-H\"older triangle, which is inner bi-Lipschitz equivalent to the germ of the set $\{0\le x\le 1,\;0\le y\le x^\beta\}\subset\R^2$.
A $\beta$-H\"older triangle $T$ has two boundary arcs, corresponding to $y=0$ and $y=x^\beta$.
All other arcs in $T$ are interior arcs.
An arc $\gamma\subset X$ is Lipschitz non-singular if it is topologically non-singular and there is a normally embedded H\"older triangle $T\subset X$ such that $\gamma$ is an interior arc of $T$.
There are finitely many Lipschitz singular arcs in any surface $X$.
A H\"older triangle is non-singular if all its interior arcs are Lipschitz non-singular.
An arc $\gamma\subset X$ is generic if its inner tangency order with any singular arc of $X$ is equal to the minimal tangency order of any two arcs in $X$.

In Subsection \ref{subsection:pancake decomposition}, we describe Kurdyka's ``pancake decomposition'' of a surface germ (see Definition \ref{Def: pancake decomposition} and Remark \ref{Rem: existence of pancake decomp}) into normally embedded H\"older triangles (``pancakes'').

Proposition \ref{Prop:two triangles} in Subsection \ref{Subsection: two triangles} states that,
 for any two normally embedded $\beta$-H\"older triangles $T$ and $T'$ such that, for some $\alpha>\beta$, the boundary arcs of $T$ have tangency orders at least $\alpha$ with the boundary arcs of $T'$ and all interior arcs of $T$ have tangency order at least $\alpha$ with $T'$, there is a bi-Lipschitz homeomorphism $h:T\to T'$ such that the tangency order between $\gamma$ and $h(\gamma)$ is at least $\alpha$ for any arc $\gamma$ of $T$.

In Section \ref{Section: Lipschitz functions on a NE Holder triangle} we present the ``pizza decomposition'' from \cite{birbrair2014lipschitz} in a suitable form.
Together with pancake decomposition, it is our main technical tool for the study of Lipschitz geometry of surface germs.

Abnormal surfaces, the main object of this paper, are introduced in Section \ref{Section: Snakes}.
An arc $\gamma\subset X$ is abnormal if there are two Lipschitz normally embedded H\"older triangles $T$ and $T'$ such that $\gamma$ is their common boundary arc and $T\cup T'$ is not normally embedded. Otherwise $\gamma$ is a normal arc.
A surface $X$ is abnormal if all its generic arcs are abnormal. Note that the set of abnormal arcs in $X$ is outer Lipschitz invariant.
Abnormal surfaces are important building blocks of general surface germs.
In particular, we study abnormal non-singular $\beta$-H\"older triangles, which we call $\beta$-snakes (see Fig.~\ref{fig:three snakes}).
Snakes are ``weakly normally embedded'' (see Definition \ref{Def: weak LNE}): if $X$ is a $\beta$-snake then any H\"older triangle $T\subset X$ that is not normally embedded has the same exponent $\beta$ (see Proposition \ref{Prop:weak LNE}).
This fundamental property of snakes allows one to clarify the outer Lipschitz geometry of a surface germ by separating exponents associated with its different parts.
Another peculiar property of snakes is non-uniqueness of their minimal pancake decompositions (see Remark \ref{REM: different number of pancakes in minimal pancake decompositions} and Fig.~\ref{fig:Two pancake decomposition of a snake}). Furthermore, there is a canonical (outer Lipschitz invariant) decomposition of the Valette link of a $\beta$-snake into finitely many normally embedded $\beta$-zones, segments and nodal zones (see Corollary \ref{Cor: V(X) is a disjoint union of segments and nodal zones} and Proposition \ref{Prop:bubble inside snake}).

In section \ref{Section: Main Theorem} we explain the role played by snakes in Lipschitz geometry of general surface germs.
Theorem \ref{Teo: Main HT decomposition} states that each abnormal arc of a surface germ $X$ belongs to one of the finitely many snakes and ``non-snake bubbles'' (see Fig.~\ref{fig:bubble snake and non-snake bubbles}) contained in $X$.

In Section \ref{Section: Names of snakes}, we introduce snake names, combinatorial invariants associated with snakes, and investigate their non-trivial combinatorics. In particular, we show that any snake name can be reduced to a binary one, and derive recurrence relations for the numbers of distinct binary snake names of different lengths.

In subsection \ref{Subsection: weak equivalence} we present a strong relationship between geometry and combinatorics of snakes.
We define ``weakly outer bi-Lipschitz maps'' (see Definition \ref{Def:weak equivalence}) between surface germs,
and give combinatorial description of weak outer Lipschitz equivalence of snakes in terms of their snake names and
some extra combinatorial data.

We thank Lev Birbrair, Alexandre Fernandes and Rodrigo Mendes, whose insights into Lipschitz geometry of singularities, made available to us through private communications, contributed to some of the technical tools used in this paper.

\section{Lipschitz geometry of surface germs: basic definitions and results}\label{Section: Basic Definitions}

All sets, functions and maps in this paper are assumed to be definable in a polynomially bounded o-minimal structure over $\mathbb{R}$ with the field of exponents $\mathbb{F}$, for example, real semialgebraic or subanalytic. Unless the contrary is explicitly stated, we consider germs at the origin of all sets and maps.

\begin{Def}\label{DEF: inner, outer and normally embedded}
	\normalfont Given a germ at the origin of a set  $X\subset \mathbb{R}^{n}$ we can define two metrics on $X$, the \textit{outer metric} $d(x,y)=|x-y|$ and the \textit{inner metric} $d_{i}(x,y)=\inf\{l(\alpha)\}$, where $l(\alpha)$ is the length of a rectifiable path $\alpha$ from $x$ to $y$ in $X$. Note that such a path $\alpha$ always exist, since $X$ is definable. A set $X \subset \R^{n}$ is \textit{Lipschitz normally embedded} if the outer and inner metrics are equivalent (i.e., there is a positive constant $C$ such that $d_i(x,y) \le Cd(x,y)$ for all $x,y\in X$).
\end{Def}

\begin{remark}
	\normalfont The inner metric is not always definable, but one can consider an equivalent definable metric (see \cite{kurdyka1997distance}), for example, the \textit{pancake metric} (see \cite{birbrair2000normal} and \cite[Lemma 2.5]{birbrair2018arc}).
\end{remark}

\begin{Exam}\label{cusp}
	\normalfont The set $X=\{(x,y) \in \R^2 \mid x^2 = y^3\}$ has two branches $X_\pm=\{y\ge 0,\;x=\pm y^{3/2}\}$. For two points $p_+=(y^{3/2},y)\in X_+$ and $p_-=(-y^{3/2},y)\in X_-$, we have $d(p_+,p_-)=2y^{3/2}$ and $d_i(p_+,p_-)\ge y$, thus $d_i(p_+,p_-)/d(p_+,p_-)\to\infty$ as $y\searrow 0$. In particular, $X$ is not normally embedded.
\end{Exam}

\subsection{H\"older triangles}\label{Subsec: Holder triangles}

\begin{Def}\label{Def: arc}
	\normalfont	An \textit{arc} in $\mathbb{R}^{n}$ is a germ at the origin of a mapping $\gamma \colon [0,\epsilon) \longrightarrow \mathbb{R}^{n}$ such that $\gamma(0) = 0$. Unless otherwise specified, we suppose that arcs are parameterized by the distance to the origin, i.e., $|\gamma(t)|=t$. We usually identify an arc $\gamma$ with its image in $\mathbb{R}^{n}$. For a germ at the origin of a set $X$, the set of all arcs $\gamma \subset X$ is denoted by $V(X)$ (known as the Valette link of $X$, see \cite{valette2007link}).
\end{Def}

\begin{Def}\label{DEF: order of tangency}
	\normalfont The \textit{tangency order} of two arcs $\gamma_{1}$ and $\gamma_{2}$ in $V(X)$ (notation $tord(\gamma_{1},\gamma_{2})$) is the exponent $q$ where $|\gamma_{1}(t) - \gamma_{2}(t)| = ct^{q} + o(t^{q})$ with $c\neq 0$. By definition, $tord(\gamma,\gamma)=\infty$. For an arc $\gamma$ and a set of arcs $Z \subset V(X)$, the tangency order of $\gamma$ and $Z$ (notation $tord(\gamma, Z)$), is the supremum of $tord(\gamma, \lambda)$ over all arcs $\lambda \in Z$. The tangency order of two sets of arcs $Z$ and $Z'$ (notation $tord(Z,Z')$) is the supremum of $tord(\gamma, Z')$ over all arcs $\gamma \in Z$. Similarly, we define the tangency orders in the inner metric, denoted by $itord(\gamma_{1},\gamma_{2}),\; itord(\gamma, Z)$ and $itord(Z,Z')$.
\end{Def}

\begin{Def}\label{DEF: standard Holder triangle}
	\normalfont For $\beta \in \mathbb{F}$, $\beta \ge 1$, the \textit{standard} $\beta$-\textit{H\"older triangle} $T_\beta\subset\R^2$ is the germ at the origin of the set
	\begin{equation}\label{Formula:Standard Holder triangle}
	T_\beta = \{(x,y)\in \R^2 : 0\le x\le 1, \; 0\le y \le x^\beta\}.
	\end{equation}
	The curves $\{x\ge 0,\; y=0\}$ and $\{x\ge 0,\; y=x^\beta\}$ are the \textit{boundary arcs} of $T_\beta$.
\end{Def}

\begin{Def}\label{DEF: Holder triangle}
	\normalfont A germ at the origin of a set $T \subset \mathbb{R}^{n}$ that is bi-Lipschitz equivalent with respect to the inner metric to the standard $\beta$-H\"older triangle $T_\beta$ is called a $\beta$-\textit{H\"older triangle} (see \cite{birbrair1999local}).
The number $\beta \in \mathbb{F}$ is called the \textit{exponent} of $T$ (notation $\beta=\mu(T)$). The arcs $\gamma_{1}$ and $\gamma_{2}$ of $T$ mapped to the boundary arcs of $T_\beta$ by the homeomorphism are the \textit{boundary arcs} of $T$ (notation $T=T(\gamma_{1},\gamma_{2})$). All other arcs of $T$ are \textit{interior arcs}. The set of interior arcs of $T$ is denoted by $I(T)$.
\end{Def}

\begin{remark}\label{Rem: NE HT condition}
	\normalfont It follows from the Arc Selection Lemma that a H\"older triangle $T$ is normally embedded if, and only if, $tord(\gamma,\gamma')=itord(\gamma,\gamma')$ for any two arcs $\gamma$ and $\gamma'$ of $T$  (see \cite[Theorem 2.2]{birbrair2018arc}).
\end{remark}

\begin{Def}\label{DEF: Lipschitz non-singular arc}
	\normalfont Let $X$ be a surface (a two-dimensional set). An arc $\gamma \subset X$ is \textit{Lipschitz non-singular} if there exists a normally embedded H\"older triangle $T \subset X$ such that $\gamma$ is an interior arc of  $T$ and $\gamma \not\subset \overline{X\setminus T}$. Otherwise, $\gamma$ is \textit{Lipschitz singular}. It follows from pancake decomposition (see Definition \ref{Def: pancake decomposition} and Remark \ref{Rem: existence of pancake decomp}) that a surface $X$ contains finitely many Lipschitz singular arcs. The union of all Lipschitz singular arcs in $X$ is denoted by $Lsing(X)$.
\end{Def}

\begin{Def}\label{DEF: non-singular Holder triangle}
	\normalfont A H\"older triangle $T$ is \textit{non-singular} if all interior arcs of $T$ are Lipschitz non-singular.
\end{Def}

\begin{Exam}\label{Exam: Lipschitz singular arc}
	\normalfont Let $\alpha,\beta \in \F$ with $1\le \beta <\alpha$. Let $\gamma_1,\gamma_2,\lambda\subset \R^3$ be arcs (not parameterized by the distance to the origin) such that $\gamma_1(t)=(t,t^\beta,0),\;\gamma_2(t)=(t,t^\beta,t^\alpha)$ and $\lambda(t)=(t,0,0)$. Consider the H\"older triangles $T_1=T(\gamma_1,\lambda)=\{(x,y,z) : x\ge 0,\;0 \le y \le x^\beta,\;z=0\}$ and $T_2=T(\lambda,\gamma_2)=\{(x,y,z) : x \ge 0,\; 0 \le y \le x^\beta,\; z = x^{\alpha-\beta}y\}$.
Let $T=T_1\cup T_2$. Note that $T_1$ and $T_2$ are normally embedded $\beta$-H\"older triangles but $T$ is not normally embedded, since $tord(\gamma_1,\gamma_2)=\alpha > \beta = itord(\gamma_1,\gamma_2)$. Thus every interior arc $\gamma\ne\lambda$ of $T$ is Lipschitz non-singular. Let us show that $\lambda$ is a Lipschitz singular arc.
	
Consider the arcs $\gamma'_1(t)=(t,t^p,0)\subset T_1$ and $\gamma'_2(t)=(t,t^p,t^{\alpha-\beta+p})\subset T_2$, where $p>\beta,\;p\in\F$.
We have $tord(\gamma'_1,\lambda)=tord(\lambda,\gamma'_2)=p$ and $tord(\gamma'_1,\gamma'_2)=\alpha - \beta + p > p=itord(\gamma'_1,\gamma'_2)$. Thus H\"older triangles $T'_1=T(\gamma'_1,\lambda)$ and $T'_2=T(\lambda,\gamma'_2)$ are normally embedded but the H\"older triangle $T_p=T'_1\cup T'_2$ is not. If $T'\subset T$ is any H\"older triangle such that $\lambda\in I(T')$ then, for large enough $p$, the H\"older triangle $T_p$ is contained in $T'$. Therefore, $T'$ is not normally embedded, thus $\lambda$ is a Lipschitz singular arc of $T$. Note also that any point of $\lambda$ other than the origin has a normally embedded neighborhood in $T$.
\end{Exam}

\begin{Def}\label{Def: generic arc of a surface}
	\normalfont Let $X$ be a surface germ with connected link. The \textit{exponent} $\mu(X)$ of $X$ is defined as $\mu(X)=\min\, itord(\gamma,\gamma')$, where the minimum is taken over all arcs $\gamma,\,\gamma'$ of $X$.
A surface $X$ with exponent $\beta$ is called a $\beta$-surface. An arc $\gamma \subset X\setminus Lsing(X)$ is \textit{generic} if $itord(\gamma,\gamma') = \mu(X)$ for all arcs $\gamma'\subset Lsing(X)$. The set of generic arcs of $X$ is denoted by $G(X)$.
\end{Def}

\begin{remark}\label{Rem: generic arcs of a non-singular HT}
	\normalfont If $X=T(\gamma_{1},\gamma_{2})$ is a non-singular $\beta$-H\"older triangle then an arc $\gamma\subset X$ is \textit{generic} if, and only if, $itord(\gamma_{1},\gamma) = itord(\gamma,\gamma_{2}) = \beta$.
\end{remark}

\begin{Lem}\label{Lem:beta-bubble two NE pieces}
	Let $\gamma$ be an arc of a $\beta$-H\"older triangle $T=T(\gamma_1,\gamma_2)$ such that $tord(\gamma_1,\gamma)=itord(\gamma_1,\gamma)$ and $tord(\gamma,\gamma_2)=itord(\gamma,\gamma_2)$. If $tord(\gamma_1,\gamma_2)>\beta$ then $$itord(\gamma_1,\gamma)=itord(\gamma,\gamma_2)=\beta.$$
\end{Lem}
\begin{proof}
	Let $\beta_{1}=tord(\gamma_{1},\gamma)=itord(\gamma_{1},\gamma)$ and $\beta_{2}=tord(\gamma,\gamma_{2})=itord(\gamma,\gamma_{2})$.
	Then $\beta=\min(itord(\gamma_{1},\gamma),itord(\gamma,\gamma_{2}))=\min(\beta_{1},\beta_{2})$. If $\beta_{1}\ne \beta_{2}$ then, by the non-archimedean property, we have $$tord(\gamma_{1},\gamma_{2})  =\min(tord(\gamma_{1},\gamma),tord(\gamma,\gamma_{2}))=\min(\beta_{1},\beta_{2})=\beta,$$ a contradiction.	
\end{proof}

\subsection{Pancake decomposition}\label{subsection:pancake decomposition}

\begin{Def}\label{Def: pancake decomposition}
	\normalfont Let $X\subset \mathbb{R}^{n}$ be the germ at the origin of a closed set. A \textit{pancake decomposition} of $X$ is a finite collection of closed normally embedded subsets $X_{k}$ of $X$ with connected links, called \textit{pancakes}, such that $X=\bigcup X_{k}$ and $$\textrm{dim}(X_{j}\cap X_{k}) < \textrm{min}(\textrm{dim}(X_{j}),\textrm{dim}(X_{k}))\quad\text{for all}\; j,k.$$
\end{Def}

\begin{remark}\label{Rem: existence of pancake decomp}
	\normalfont The term ``pancake'' was introduced in \cite{birbrair2000normal}, but this notion first appeared (with a different name) in \cite{kurdyka1992subanalytic} and \cite{kurdyka1997distance}, where the existence of such decomposition was established.
\end{remark}

\begin{remark}\label{Rem:pancake of holder triangle is holder triangle}
	\normalfont If $X$ is a H\"older triangle then each pancake $X_k$ is also a H\"older triangle.
\end{remark}

\begin{Def}\label{Def: minimal pancak decomp}
	\normalfont A pancake decomposition $\{X_{k}\}$ of a set $X$ is \textit{minimal} if the union of any two adjacent pancakes $X_{j}$ and $X_{k}$ (such that $X_{j}\cap X_{k}\ne \{0\}$) is not normally embedded.
\end{Def}

\begin{remark}
	\normalfont When the union of two adjacent pancakes is normally embedded, they can be replaced by their union, reducing the number of pancakes. Thus, a minimal pancake decomposition always exists.
\end{remark}

\subsection{Bi-Lipschitz homeomorphisms between pancakes}\label{Subsection: two triangles}

\begin{Prop}\label{Prop:two triangles}
Let $T=T(\gamma_1,\gamma_2)$ and $T'=T(\gamma'_1,\gamma'_2)$ be normally embedded $\beta$-H\"older triangles such that $tord(\gamma_1,\gamma'_1)\ge\alpha,\;tord(\gamma_2,\gamma'_2)\ge\alpha$, and $tord(\gamma,T')\ge\alpha$
for all arcs $\gamma\subset T$, for some $\alpha>\beta$.
Then there is a bi-Lipschitz homeomorphism $h: T\to T'$ such that $h(\gamma_1)=\gamma'_1,\;h(\gamma_2)=\gamma'_2$,
and $tord(h(\gamma),\gamma)\ge\alpha$ for any arc $\gamma\subset T$.
\end{Prop}

\begin{proof}
According to Theorem 4.5 from \cite{birbrair2020lipschitz}, we may assume, embedding $T\cup T'$ into $\R^n$ for some $n\ge 5$,
that $T'=T_\beta$ is a standard $\beta$-H\"older triangle (\ref{Formula:Standard Holder triangle}) in the $xy$-plane $\R^2\subset\R^n$,
$\gamma'_1$ belongs to the positive $x$-axis and $\gamma'_2$ to the graph $y=x^\beta$.
Let $\pi:\R^n\to\R^2$ be orthogonal projection, and let $\rho:\R^n\to\R^{n-2}$ be orthogonal projection to the orthogonal complement of $\R^2$ in $\R^n$.
Orientation of $\R^2$ defines orientation of $T'$ such that $\gamma'_1$ is oriented in the positive $x$-axis direction,
and the link of $T'$ is oriented from $\gamma'_1$ to $\gamma'_2$. We assume that $T$ is oriented consistently, so that
its link is oriented from $\gamma_1$ to $\gamma_2$.

Let $S$ be the set of those points of $T$ where $\pi|_T$ is not a smooth, one-to-one, orientation-preserving map.
We claim that $S$ is a union of finitely many $\beta_j$-H\"older triangles $T_j$, where $\beta_j\ge\alpha$ for all $j$.

Let $V\subset\R^2$ be the union of the set of critical values of $\pi|_T$ and the arcs
$\pi(\gamma_1),\;\pi(\gamma_2),\;\gamma'_1$ and $\gamma'_2$. The set $W=\pi^{-1}(V)\cap T$  consists of
finitely many isolated arcs and, possibly, some ``vertical'' H\"older triangles mapped by $\pi$ to arcs in $\R^2$.
Removing from $W$ interiors of the vertical triangles, we obtain the set $U\subset T$ consisting of finitely many arcs,
all of them having tangency order at least $\alpha$ with $\R^2$, since they have tangency order at least $\alpha$ with $T'$.
Let $T_j\subset T$ be $\beta_j$-H\"older triangles bounded by arcs from $U$ and containing no interior arcs from $U$.
If $T_j$ is a vertical triangle then $\beta_j\ge\alpha$.

For any non-vertical triangle $T_j$, if $\pi^{-1}(\pi(T_j))\cap T$ contains more than one non-vertical triangle then,
since $T$ is normally embedded and all arcs of $T$ have tangency order at least $\alpha$ with $\R^2$, we have $\beta_j\ge\alpha$.
If $\pi|_{T_j}$ is orientation reversing then there is
a $\beta_j$-H\"older triangle $T_k\subset\pi^{-1}(\pi(T_j))\cap T$ such that $\pi|_{T_k}$ is orientation preserving, thus $\beta_j\ge\alpha$ in that case, too.

Note also that each of the sets $\overline{T'\setminus\pi(T)}$ and
$\overline{\pi(T)\setminus T'}$ is either empty or consists of at most two H\"older triangles with exponents at least $\alpha$, since $tord(\gamma_1,\gamma'_1)\ge\alpha$ and $tord(\gamma_2,\gamma'_2)\ge\alpha$.

Let now $T_j\subset T$ be a $\beta$-H\"older triangle bounded by two arcs from $U$
and containing no interior arcs from $U$. Then $T'_j=\pi(T_j)\subset T'$,
$\pi|_{T_j}$ is orientation preserving, and for each interior point $P\in T_j$ we have $\pi^{-1}(\pi(P))=\{P\}$.
For $(x,y)=\pi(P)\in T'_j$, let $f(x,y)=\rho(P)$ be a function $f=(f_1,\ldots,f_{n-2}):T'_j\to\R^{n-2}$.
For $c>0$, let $T'_{j,c}$ be the set of points in $T'_j$ where either $f$ is not differentiable or
$|\partial f_k/\partial y|\ge c$ for some $k$.
Since $tord(\gamma,\R^2)\ge\alpha$ for each arc $\gamma\subset T$, each set $T'_c$ is contained in the union of finitely many $\alpha$-H\"older triangles. Note that the mapping $\pi:{T_j}\to T'_j$ is bi-Lipschitz outside these triangles.

Adding the sets $T_{j,c}=\pi^{-1}(T'_{j,c})\cap T_j$, for some $c>0$ and each $\beta$-H\"older triangle $T_j$, to the set $S$, we can find a finite set of disjoint $\alpha$-H\"older triangles in $T$ such that projection of each of them to $\R^2$ is an $\alpha$-H\"older triangle either contained in $T'$ or intersecting $T'$ over an $\alpha$-H\"older triangle, and $\pi|_T$ is a bi-Lipschitz mapping from $T$ to $T'$ outside these triangles.

We can now define $h:T\to T'$ as any orientation preserving bi-Lipschitz homeomorphism from each of these $\alpha$-H\"older triangles to intersection of its projection with $T'$, and as $\pi$ in the complement to all these triangles.
\end{proof}

\begin{Exam}\label{ex:twotriang}
\normalfont The links of two normally embedded $\beta$-H\"older triangles $T=T(\gamma_1,\gamma_2)$ and $T'=T(\gamma'_1,\gamma'_2)$ are shown in Fig.~\ref{fig:twotriang}.
A H\"older triangle $T(q,r)\subset T$ is projected to $T'$ with orientation reversed. Since its exponent is $\beta_1\ge\alpha>\beta$, one can choose
the points $p,\,q,\,r,\,s$ so that $p'=\pi(p),\,s'=\pi(s)$ and a mapping $h:T'\to T$ such that $h(q')=q$ and $h(r')=r$ is an outer bi-Lipschitz homeomorphism.
\end{Exam}

\begin{figure}
		\centering
		\includegraphics[width=4.5in]{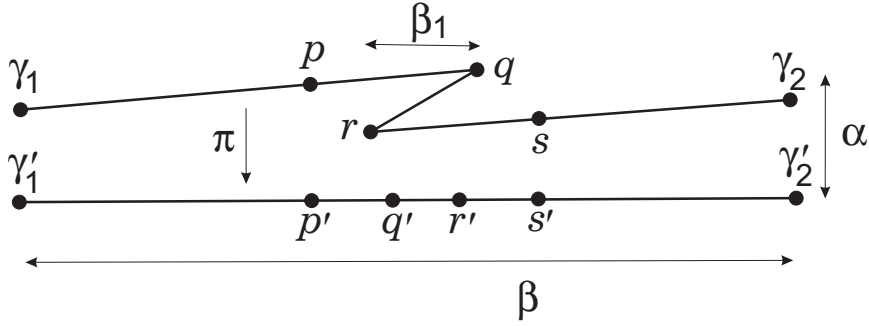}
		\caption{Links of two normally embedded $\beta$-H\"older triangles in Example \ref{ex:twotriang}}\label{fig:twotriang}
	\end{figure}

\subsection{Pizza decomposition}\label{Subsection: Pizza Decomp}

In this subsection we use the definitions and results of \cite{birbrair2014lipschitz}.

\begin{Def}\label{DEF: order of a function on an arc}
	\normalfont Let $f\not\equiv 0$ be a germ at the origin of a Lipschitz function defined on an arc $\gamma$. The \textit{order} of $f$ on $\gamma$, denoted by $ord_{\gamma}f$, is the value $q \in \mathbb{F}$ such that $f(\gamma(t)) = ct^{q} + o(t^{q})$ as $t \rightarrow 0$, where $c\neq 0$. If $f\equiv 0$ on $\gamma$, we set $ord_{\gamma}f = \infty$.
\end{Def}

\begin{Def}\label{DEF: interval Q of a pizza slice}
	\normalfont Let $T\subset \mathbb{R}^{n}$ be a H\"older triangle, and let $f\colon (T,0) \rightarrow (\mathbb{R},0)$ be a Lipschitz function. We define $$Q_{f}(T)=\bigcup_{\gamma\in V(T)}ord_{\gamma}f.$$
\end{Def}

\begin{remark}\label{Rem:Q_{f}(T) is a segment}
	\normalfont It was shown in \cite{birbrair2014lipschitz} that $Q_{f}(T)$ is a closed segment in $\mathbb{F}\cup \{\infty\}$.
\end{remark}

\begin{Def}\label{DEF: Elementary Holder triangle}
	\normalfont A H\"older triangle $T$ is \textit{elementary} with respect to a Lipschitz function $f$ if, for any two distinct arcs $\gamma$ and $\gamma'$ in $T$ such that $ord_{\gamma}f=ord_{\gamma'}f=q$, the order of $f$ is $q$ on any arc in the H\"older triangle $T(\gamma,\gamma')\subset T$.
\end{Def}

\begin{Def}\label{Def:width function}
	\normalfont Let $T\subset \mathbb{R}^{n}$ be a H\"older triangle and $f\colon (T,0) \rightarrow (\mathbb{R},0)$ a Lipschitz function. For each arc $\gamma \subset T$, the \textit{width} $\mu_{T}(\gamma,f)$ of $\gamma$ with respect to $f$ is the infimum of the exponents of H\"older triangles $T'\subset T$ containing $\gamma$ such that $Q_{f}(T')$ is a point. For $q\in Q_{f}(T)$ let $\mu_{T,f}(q)$ be the set of exponents $\mu_{T}(\gamma,f)$, where $\gamma$ is any arc in $T$ such that $ord_{\gamma}f=q$. It was shown in \cite{birbrair2014lipschitz} that the set $\mu_{T,f}(q)$ is finite. This defines a multivalued \textit{width function} $\mu_{T,f}\colon Q_{f}(T)\rightarrow \mathbb{F}\cup \{\infty\}$. When $f$ is fixed, we write $\mu_{T}(\gamma)$ instead of $\mu_{T}(\gamma,f)$ and $\mu_{T}$ instead of $\mu_{T,f}$. If $T$ is an elementary H\"older triangle with respect to $f$ then the function $\mu_{T,f}$ is single valued.
\end{Def}

\begin{Def}\label{DEF: pizza slice}
	\normalfont Let $T$ be a H\"older triangle and $f\colon (T,0)\rightarrow (\mathbb{R},0)$ a Lipschitz function. We say that $T$ is a \textit{pizza slice} associated with $f$ if it is elementary with respect to $f$ and $\mu_{T,f}(q)=aq+b$ is an affine function.
\end{Def}

\begin{Lem}\label{Lem: width function on NE HT}
	Let $X=T(\gamma_{1},\gamma_{2})$ be a normally embedded H\"older triangle partitioned by an interior arc $\gamma$ into two H\"older triangles $X_1=T(\gamma_{1},\gamma)$ and $X_2=T(\gamma,\gamma_{2})$. Let $f\colon (X_1,0)\rightarrow (\R,0)$ be the function given by $f(x)=d(x,X_2)$. Then, for every arc $\theta \subset X_1$, we have $$ord_{\theta}f=\mu_{X_1}(\theta,f)=tord(\theta,\gamma).$$
\end{Lem}
\begin{proof}
	Since $X$ is normally embedded, we can assume that $X$ is a standard H\"older triangle in $\R^2$. Then, for every arc $\gamma' \subset X_1$, since $\gamma$ is the closest arc in $X_2$ to $\gamma'$, we have $ord_{\gamma'}f = tord(\gamma',\gamma)$. Moreover, given an arc $\theta \subset X_1$, we write $q_\theta=ord_{\theta}f$, and if $\theta'\in G(T(\theta,\gamma))$ then $ord_{\gamma'}f=tord(\gamma',\gamma)=tord(\theta,\gamma)=q_\theta$ for every arc $\gamma'\subset T(\theta,\theta')$. Thus, $\mu_{X_1}(q_{\theta})\le \mu(T(\theta,\theta'))=tord(\theta,\gamma) =q_{\theta}$. However, if $\mu_{X_1}(q_{\theta})< q_{\theta}$ then there is an arc $\gamma'\subset X_1$ such that $tord(\theta,\gamma')<q_\theta$ and consequently, $tord(\theta,\gamma)\ne q_\theta$.
\end{proof}

\begin{Prop}\label{Prop:width function properties elementary triangle}
	\emph{(See \cite{birbrair2014lipschitz})} Let $T$ be a $\beta$-H\"older triangle, $f$ a Lipschitz function on $T$ and $Q=Q_{f}(T)$. If $T$ is a pizza slice associated with $f$ then
	\begin{enumerate}
		\item $\mu_{T}$ is constant only when $Q$ is a point;
		\item $\beta\le\mu_{T}(q)\le \max(q,\beta)$ for all $q \in Q$;
		\item $\mu(ord_{\gamma}f)=\beta$ for all $\gamma \in G(T)$;
		\item If $Q$ is not a point, let $\mu_0=\max_{q\in Q}\mu_T(q)$, and let $\gamma_0$ be the boundary arc of $T$ such that $\mu_T(\gamma_0)=\mu_0$. Then $\mu_T(\gamma)=itord(\gamma_0,\gamma)$
		for all arcs $\gamma\subset T$ such that $itord(\gamma_0,\gamma)\le\mu_0$.
	\end{enumerate}
\end{Prop}

\begin{Def}\label{Def:Pizza decomp}
		\normalfont A decomposition $\{T_i\}$ of a H\"older triangle $X$ into $\beta_{i}$-H\"older triangles $T_{i}=T(\lambda_{i-1},\lambda_{i})$ such that $T_{i-1}\cap T_{i}=\lambda_{i}$ is a \textit{pizza decomposition} of $X$ (or just a \textit{pizza} on $X$) associated with $f$ if each $T_{i}$ is a pizza slice associated with $f$. We write $Q_{i}=Q_{f}(T_{i})$, $\mu_{i}=\mu_{T_{i},f}$ and $q_i=ord_{\lambda_i}f$. 
\end{Def}

\begin{remark}\label{REM: existence of pizza slice}
	\normalfont The existence of a pizza associated with a function $f$ was proved in \cite{birbrair2014lipschitz} for a function defined in $(\R^{2},0)$. The same arguments prove the existence of a pizza associated with a function defined on a H\"older triangle as in Definition \ref{Def:Pizza decomp}. The results mentioned in this subsection remain true when $f$ is a Lipschitz function on a H\"older triangle $T$ with respect to the inner metric, although in this paper we need them only for Lipschitz functions with respect to the outer metric.
\end{remark}

\begin{Def}\label{Def:minimal pizza}
	\normalfont A pizza $\{T_i\}_{i=1}^{p}$ associated with a function $f$ is \textit{minimal} if, for any $i\in \{2,\ldots, p\}$, $T_{i-1}\cup T_i$ is not a pizza slice associated with $f$.
\end{Def}

\begin{Exam}
	\normalfont Consider $T_1$ and $T_2$ as in Example \ref{Exam: Lipschitz singular arc}. Let $f\colon (T_1,0)\rightarrow (\R,0)$ be the function given by $f(x,y,z)=x^{\alpha - \beta}y$. Note that $T_2$ is the graph of $f$. For each arc $\gamma\subset T_1$, we have $\gamma(t)=(t,ct^p + o(t^p),0)$, where $c>0$ and $p\ge \beta$. Hence, $f(\gamma(t)) = ct^{\alpha - \beta + p} + o(t^{\alpha - \beta + p})$ and, consequently, $ord_{\gamma}f=\alpha - \beta + p$. Moreover, if $\gamma'(t)=c't^{p'} + o(t^{p'})$ is another arc in $T_1$ then $ord_{\gamma}f=ord_{\gamma'}f$ if and only if $p=p'$. Thus, $T_1$ is elementary with respect to $f$, and a minimal pizza decomposition of $T_1$ associated with $f$ consists of the single pizza slice $T_1$ with $Q_1=[\alpha,\infty)$. Since $ord_{\lambda}f=\infty$, we have $\mu(ord_{\lambda}f)=\infty = \max_{q\in Q_1} \mu(q)$. Proposition \ref{Prop:width function properties elementary triangle} implies that $\mu(ord_{\gamma}f) = itord(\gamma,\lambda)$ for every arc $\gamma \subset T_1$. For $\gamma(t)=(t,ct^p + o(t^p),0)$ we obtain $itord(\gamma,\lambda) = tord(\gamma,\lambda) = p$. Since $q=ord_{\gamma}f=\alpha - \beta + p$, we have $p=q + \beta - \alpha$, thus $\mu(q)=q + \beta - \alpha$ for every $q\in Q_1$.
\end{Exam}

\begin{Def}\label{DEF: multipizza}
	\normalfont Consider the set of germs of Lipschitz functions $f_{l}\colon (X,0) \rightarrow (\mathbb{R},0),\;l=1,\ldots , m$, defined on a H\"older triangle $X$. A \textit{multipizza} on $X$ associated with $\{f_{1},\ldots, f_{m}\}$ is a decomposition $\{T_{i}\}$ of $X$ into $\beta_{i}$-H\"older triangles which is a pizza on $X$ associated with $f_{l}$ for each $l$.
\end{Def}

\begin{remark}\label{REM: existence of a multipizza}
	\normalfont The existence of a multipizza follows from the existence of a pizza associated with a single Lipschitz function $f$, since a refinement of a pizza associated with any function $f$ is also a pizza associated with $f$.
\end{remark}

\subsection{Zones}\label{Subsec: Zones}
In this subsection, $(X,0)\subset (\R^n,0)$ is a surface germ.

\begin{Def}\label{Def: zone}
	\normalfont A nonempty set of arcs $Z \subset V(X)$ is a \textit{zone} if, for any two distinct arcs $\gamma_{1}$ and $\gamma_{2}$ in $Z$, there exists a non-singular H\"older triangle $T=T(\gamma_{1},\gamma_{2}) \subset X$ such that $V(T) \subset Z$. If $Z = \{\gamma\}$ then $Z$ is a \textit{singular zone}.
\end{Def}

\begin{Def}\label{Def: maximal zone in}
	\normalfont Let $B \subset V(X)$ be a nonempty set. A zone $Z\subset B$ is \textit{maximal in} $B$ if, for any non-singular H\"older triangle $T$ such that $V(T) \subset B$, one has either $Z\cap V(T)=\emptyset$ or $V(T) \subset Z$.
\end{Def}

\begin{remark}\label{rem:maxzone}
\normalfont It follows from the definition of a zone that, for any family $\{Z_i\}$ of zones containing an arc $\gamma$, their union is also a zone. In particular, if a  zone $Z\subset B$ is maximal in $B$ then, for any arc $\gamma\in Z$, it is the union of all zones in $B$ containing $\gamma$.
\end{remark}

\begin{remark}
	\normalfont A zone could be understood as an analog of a connected subset of $V(X)$, and a maximal zone in a set $B$ is an analog of a connected component of $B$.
\end{remark}

\begin{Def}\label{Def:order of zone}
	\normalfont The \textit{order} $\mu(Z)$ of a zone $Z$ is the infimum of $tord(\gamma,\gamma')$ over all arcs $\gamma$ and $\gamma'$ in $Z$. If $Z$ is a singular zone then $\mu(Z) = \infty$. A zone $Z$ of order $\beta$ is called a $\beta$-zone.
\end{Def}

\begin{remark}
	\normalfont The tangency order can be replaced by the inner tangency order in Definition \ref{Def:order of zone}. Note that, for any arc $\gamma \in Z$, $\inf_{\gamma'\in Z}tord(\gamma,\gamma')=\inf_{\gamma'\in Z}itord(\gamma,\gamma')=\mu(Z)$.
\end{remark}

\begin{Def}\label{Def: LNE zone}
	\normalfont A zone $Z$ is normally embedded if, for any two arcs $\gamma$ and $\gamma'$ in $Z$, there exists an normally embedded H\"older triangle $T=T(\gamma,\gamma')$ such that $V(T)\subset Z$.
\end{Def}

\begin{Def}\label{Def: open closed perfect zone}
	\normalfont A $\beta$-zone $Z$ is \textit{closed} if there is a $\beta$-H\"older triangle $T$ such that $V(T)\subset Z$.
Otherwise, $Z$ is \textit{open}.
If $Z$ is a closed $\beta$-zone then an arc $\gamma \in Z$ is \textit{generic} if there exists a $\beta$-H\"older triangle $T$ such that $V(T) \subset Z$ and $\gamma$ is a generic arc of $T$.
If $Z$ is an open $\beta$-zone then an arc $\gamma \in Z$ is \textit{generic} if, for any $\alpha>\beta$, there exists an $\alpha$-H\"older triangle $T$ such that $V(T) \subset Z$ and $\gamma$ is a generic arc of $T$.
The set of generic arcs of $Z$ is denoted by $G(Z)$. By definition, if $Z$ is a singular zone, then its only arc is generic.
 A zone $Z$ is \textit{perfect} if $Z=G(Z)$.
\end{Def}

\begin{Def}\label{Def:complete and open complete zone}
	\normalfont A closed $\beta$-zone $Z \subset V(X)$ is $\beta$-\textit{complete} if, for any $\gamma\in Z$, $$Z=\{\gamma'\in V(X) : itord(\gamma,\gamma')\ge \beta\}.$$ An open $\beta$-zone $Z \subset V(X)$ is $\beta$-\textit{complete} if, for any $\gamma\in Z$, $$Z=\{\gamma'\in V(X) : itord(\gamma,\gamma')> \beta\}.$$
\end{Def}

\begin{remark}\label{Rem:open complete zones}
	\normalfont Let $Z$ and $Z'$ be open $\beta$-complete zones. Then, either $Z\cap Z'=\emptyset$ or $Z=Z'$. Moreover, $Z\cap Z'=\emptyset$ implies $itord(Z,Z')\le\beta$. The same holds when $Z$ and $Z'$ are closed $\beta$-complete zones, except $Z\cap Z'=\emptyset$ implies $itord(Z,Z')<\beta$.
\end{remark}

\begin{Exam}\label{Exam: closed, complete and perfect zones}
	\normalfont If $T$ is a non-singular $\beta$-H\"older triangle then the set $V(T)$ of all arcs in $T$, the set $I(T)$ of interior arcs of $T$, and the set $G(T)$ of generic arcs of $T$ are closed $\beta$-zones, but only $V(T)$ is $\beta$-complete, and only $G(T)$ is closed perfect. The set $V(T)\setminus G(T)$ consists of two open non-perfect $\beta$-complete zones. For any arc $\gamma\in G(T)$, the set of arcs $\gamma'\in V(T)$ such that $tord(\gamma,\gamma')>\beta$ is an open perfect $\beta$-zone.
\end{Exam}

\begin{Def}\label{Def:adjacent zones}
	\normalfont Two zones $Z$ and $Z'$ in $V(X)$ are \textit{adjacent} if $Z\cap Z'=\emptyset$ and there exist arcs $\gamma \subset Z$ and $\gamma' \subset Z'$ such that $V(T(\gamma,\gamma')) \subset Z\cup Z'$.
\end{Def}

\begin{Lem}\label{Lem: union of adjacent zones is a zone}
	Let $X$ be a H\"older triangle, and let $Z$ and $Z'$ be two zones in $V(X)$ of orders $\beta$ and $\beta'$, respectively. If either $Z\cap Z'\ne\emptyset$ or $Z$ and $Z'$ are adjacent, then $Z\cup Z'$ is a zone of order $\min(\beta,\beta')$.
\end{Lem}
\begin{proof}
	One can easily check that in both cases $Z\cup Z'$ is a zone.

If there is an arc $\lambda\in Z\cap Z'$ then, for any arcs $\gamma\in Z$ and $\gamma'\in Z'$, we have $itord(\gamma,\gamma')\ge\min(itord(\gamma,\lambda),itord(\lambda,\gamma'))\ge\min(\beta,\beta')$.

If $Z$ and $Z'$ are adjacent, let $T=T(\lambda,\lambda')$ be a H\"older triangle such that $\lambda\in Z,\;\lambda'\in Z'$ and
$V(T)\subset Z\cup Z'$. If $\mu(T)<\min(\beta,\beta')$, let us choose an arc $\lambda''\in G(T)$.
If $\lambda''\in Z$ (resp., $\lambda''\in Z'$) then $itord(\lambda,\lambda'')<\beta$ (resp., $itord(\lambda'',\lambda')<\beta'$),
a contradiction. Thus $\mu(T)\ge\min(\beta,\beta')$ and for any arcs $\gamma\in Z$ and $\gamma'\in Z'$ we have $itord(\gamma,\gamma')\ge\min(itord(\gamma,\lambda),itord(\lambda',\gamma'),\mu(T))\ge\min(\beta,\beta')$,
so $\mu(Z\cup Z')=\min(\beta,\beta')$ in both cases.
\end{proof}

\begin{Lem}\label{Lem: beta zone must have beta intersection part in HT decomp}
	Let $\{X_i\}$ be a finite decomposition of a H\"older triangle $X$ into $\beta_i$-H\"older triangles. If $Z\subset V(X)$ is a $\beta$-zone then $Z_i=Z\cap V(X_i)$ is a $\beta$-zone for some $i$.
\end{Lem}
\begin{proof}
	Since $Z=\bigcup_i Z_i$, it follows from Lemma \ref{Lem: union of adjacent zones is a zone} that $\mu(Z)=\min_{i}\mu(Z_i)$.
If $\mu(Z_i)>\beta$ for all $i$ then, by the non-archimedean property, $\mu(Z)>\beta$, a contradiction.
\end{proof}

\begin{Lem}\label{Lem: there are no adjacent perfect zones}
	Let $X$ be a H\"older triangle. If $Z$ and $Z'$ are closed perfect $\beta$-zones in $V(X)$, then they are not adjacent.
\end{Lem}
\begin{proof}
	Suppose, by contradiction, that $Z$ and $Z'$ are adjacent. Definition \ref{Def:adjacent zones} implies that there is a H\"older triangle $T=T(\gamma,\gamma')$ such that $\gamma \in Z,\;\gamma'\in Z'$ and $V(T)\subset Z\cup Z'$.
Since $Z$ and $Z'$ are adjacent $\beta$-zones, $\mu(T)\ge\mu(Z\cup Z')=\beta$ by Lemma \ref{Lem: union of adjacent zones is a zone}.
Let $h:T_\beta\to T$ be an inner bi-Lipschitz homeomorphism, where $T_\beta$ is a standard $\beta$-H\"older triangle (\ref{Formula:Standard Holder triangle}),
such that $h(\{x\ge 0,\;y=0\})=\gamma$. Let $c_0=\sup\left(c\in[0,1]: h(\{x\ge 0,\;y=c x^\beta\})\in Z\right)$.
If $\gamma_0=h(\{x\ge 0,\;y=c_0 x^\beta\})\in Z$ then any arc $\gamma_1\subset T\setminus T(\gamma,\gamma_0)$ such that $itord(\gamma_0,\gamma_1)=\beta$
does not belong to $Z$, in contradiction with $Z$ being a closed perfect $\beta$-zone.
Similarly, if $\gamma_0\in Z'$ there is a contradiction with $Z'$ being a closed perfect $\beta$-zone.
\end{proof}

\begin{Def}\label{DEF: normal and abnormal arcs and zones}
	\normalfont A Lipschitz non-singular arc $\gamma$ of a surface germ $X$ is \textit{abnormal} if there are two normally embedded H\"older triangles $T\subset X$ and $T'\subset X$ such that $T\cap T' = \gamma$ and $T\cup T'$ is not normally embedded.
Otherwise $\gamma$ is \textit{normal}. A zone is \textit{abnormal} (resp., \textit{normal}) if all of its arcs are abnormal (resp., normal). The sets of abnormal and normal arcs of $X$ are denoted $Abn(X)$ and $Nor(X)$, respectively.
\end{Def}

\begin{remark}\label{rem:normal and abnormal arcs}
	\normalfont It follows from Definition \ref{DEF: normal and abnormal arcs and zones} that the property of an arc to be abnormal (resp., normal)
is outer Lipschitz invariant: if $h:X\to X'$
is an outer bi-Lipschitz homeomorphism then $h(\gamma)\subset X'$ is an abnormal (resp., normal) arc for any abnormal (resp., normal) arc $\gamma\subset X$.
\end{remark}

\begin{Def}\label{Def: abnormal surface}
	\normalfont A surface germ $X$ is called \textit{abnormal} if $Abn(X)=G(X)$, the set of generic arcs of $X$.
\end{Def}

\begin{remark}\label{Rem:triangles of abnormal arc}
	\normalfont Given an abnormal arc $\gamma\subset X$, we can choose normally embedded triangles $T=T(\lambda,\gamma)\subset X$ and $T'=T(\gamma,\lambda')\subset X$ so that $T\cap T'=\gamma$ and $tord(\lambda,\lambda')>itord(\lambda,\lambda')$. It follows from Lemma \ref{Lem:beta-bubble two NE pieces} that $tord(\lambda,\gamma)=tord(\gamma,\lambda')=itord(\lambda,\lambda')$.
\end{remark}

\begin{Def}\label{Def: maximal abnormal and normal zones}
	\normalfont Given an abnormal (resp., normal) arc $\gamma \subset X$ the \textit{maximal abnormal zone} (resp., \textit{maximal normal zone}) in $V(X)$ containing $\gamma$ is the union of all abnormal (resp., normal) zones in $V(X)$ containing $\gamma$. Alternatively, the maximal abnormal (resp., normal) zone containing $\gamma$ is the union of all zones in $Abn(X)$ (resp., $Nor(X)$) containing $\gamma$.
\end{Def}

\begin{remark}\label{max zones are unique}
\normalfont Since the property of an arc to be abnormal (resp., normal) is outer Lipschitz invariant (see Remark \ref{rem:normal and abnormal arcs}),
maximal abnormal (resp., normal) zones in $V(X)$ are also outer Lipschitz invariant: if $h:X\to X'$
is an outer bi-Lipschitz homeomorphism then $h(Z)\subset V(X')$ is a maximal abnormal (resp., normal) zone for any maximal abnormal (resp., normal) zone $Z\subset V(X)$. Here $h:V(X)\to V(X')$ is the natural action of $h$ on arcs in $X$. Classification of maximal abnormal zones in $V(X)$ will be given in Section \ref{Section: Main Theorem} below.
\end{remark}



\section{Lipschitz functions on a normally embedded $\beta$-H\"older triangle}\label{Section: Lipschitz functions on a NE Holder triangle}

\begin{Def}\label{Def:B_beta and H_beta}
	\normalfont Let $(T,0)\subset (\R^{n},0)$ be a normally embedded $\beta$-H\"older triangle, and $f\colon (T,0)\rightarrow (\R,0)$ a Lipschitz function such that $ord_{\gamma}f\ge \beta$ for all $\gamma \in V(T)$. We define the following sets of arcs:
	$$B_{\beta}=B_{\beta}(f)=\{\gamma \in G(T) : ord_{\gamma}f=\beta\}$$ and
	$$H_{\beta}=H_{\beta}(f)=\{\gamma \in G(T) : ord_{\gamma}f>\beta\}.$$
\end{Def}

In this section we study properties of these two sets. In particular, we are going to prove that each of them is a finite union of $\beta$-zones. The following two statements follow immediately from $f$ being Lipschitz.

\begin{Lem}\label{Lem:itord>beta in H_{beta} and B_{beta} implies same ordf}
	Let $T$ and $f$ be as in Definition \ref{Def:B_beta and H_beta}, and let $\gamma \in B_{\beta}$ and $\gamma' \in H_\beta$. Then $tord(\gamma,\gamma')=\beta$.
\end{Lem}

\begin{Lem}\label{Lem:order of boundary arcs of adjacent pizza slices}
	\normalfont Let $T$ and $f$ be as in Definition \ref{Def:B_beta and H_beta}, and let $T'=T(\gamma_{1},\gamma_{2})\subset T$. If $ord_{\gamma_{1}}f=\beta$ and $ord_{\gamma_{2}}f>\beta$ then $\mu(T')=\beta$ and $\mu_{T'}(\gamma_{1},f)=\beta$.
\end{Lem}

\begin{Lem}\label{Lem:minimal pizza}
	Let $T$ and $f$ be as in Definition \ref{Def:B_beta and H_beta}, and let $\{T_i=T(\lambda_{i-1},\lambda_{i})\}_{i=1}^{p}$ be a minimal pizza on $T$ associated with $f$. If $p>1$ then each $T_i$ has at least one boundary arc $\lambda$ such that $ord_{\lambda}f>\beta$.
\end{Lem}
\begin{proof}
	Note that, for each $i<p$, if $ord_{\lambda_{i-1}}f=ord_{\lambda_i}f=\beta$ then $ord_{\lambda_{i+1}}f>\beta$. Indeed, if $ord_{\lambda_{i-1}}f=ord_{\lambda_i}f=ord_{\lambda_{i+1}}f=\beta$ then $Q_i=Q_{i+1}=\{\beta\}$ and $T_i\cup T_{i+1}$ is a pizza slice, a contradiction with $\{T_i\}$ being minimal. Similarly, for each $i>1$, if $ord_{\lambda_{i-1}} f=ord_{\lambda_{i}}  f=\beta$ then $ord_{\lambda_{i-2}} f>\beta$.
	
	Suppose, by contradiction, that there exists $T_i$ such that $ord_{\lambda_{i-1}}f=ord_{\lambda_i}f=\beta$. Since $p>1$, either $i<p$ or $i>1$. If $i<p$ then $ord_{\lambda_{i+1}} f>\beta$ and, by Lemma 3.2, $T_i\cup T_{i+1}$ is a pizza slice, in contradiction with $\{T_i\}$ being minimal. Similarly, if $i>1$ then $T_{i-1}\cup T_i$ is a pizza slice, again a contradiction.

\end{proof}
\begin{Lem}\label{Lem:pizza HT and arcs in B_{beta}}
	Let $T$ and $f$ be as in Definition \ref{Def:B_beta and H_beta}, and let $\{T_{i}=T(\lambda_{i-1},\lambda_{i})\}_{i=1}^{p}$ be a minimal pizza associated with $f$. Then:
	\begin{enumerate}
		\item If $B_{\beta}\cap V(T_{i})\ne \emptyset$ then $\beta_{i}=\beta$.
		\item If $\lambda_{i} \in B_{\beta}$ then there exists a $\beta$-H\"older triangle $T'\subset T_i\cup T_{i+1}$, with $V(T') \subset B_\beta$, such that $\lambda_{i}$ is a generic arc of $T'$.
	\end{enumerate}	
\end{Lem}
\begin{proof}
	$(1)$ Consider $\gamma \in B_{\beta}\cap V(T_{i})$.  If $T_i=T$ the statement is obvious, since $T$ has exponent $\beta$. Suppose that $T_i\ne T$. Since $T_i$ is pizza slice (in particular, $T_i$ is elementary with respect to $f$) and $ord_{\gamma}f=\beta$, either $ord_{\lambda_{i-1}}f=\beta$ or $ord_{\lambda_{i}}f=\beta$. Then, by Lemmas \ref{Lem:order of boundary arcs of adjacent pizza slices} and \ref{Lem:minimal pizza}, $\beta_i=\beta$.
	
	$(2)$ As $\lambda_i \in B_\beta \subset G(T)$, it is not one of the boundary arcs of $T$. In particular, $0<i<p$. Item $(1)$ of this Lemma implies that $\beta_{i}=\beta_{i+1}=\beta$. Thus, $G(T_i\cup T_{i+1}) \subset B_\beta$ and one can define $T'=T(\gamma',\gamma'')$ where $\gamma'\in G(T_i)$ and $\gamma''\in G(T_{i+1})$.
\end{proof}

\begin{Prop}\label{Prop:arcs in B_{beta} are generic}
	Let $T$ and $f$ be as in Definition \ref{Def:B_beta and H_beta}, and let $\{T_{i}\}_{i=1}^{p}$ be a minimal pizza on $T$ associated with $f$. Let $B_0=G(T_1)$, $B_p=G(T_p)$ and, for $0<i<p$, $B_i=G(T_i\cup T_{i+1})$. Then
	\begin{enumerate}
		\item If $ord_{\lambda_{i}}f=\beta$ then $B_i$ is a closed perfect $\beta$-zone maximal in $B_\beta$.
		\item If $p>1$ then the set $B_\beta$ is disjoint union of all closed perfect $\beta$-zones $B_i$ such that $ord_{\lambda_{i}}f=\beta$.
	\end{enumerate}
\end{Prop}
\begin{proof}
	$(1)$ When $p=1$ and $B_\beta \ne \emptyset$ then $B_0=B_p=B_\beta = G(T)$ and the result is trivially true. Thus, assume that $p>1$.
	Consider $0\le i <p$ such that $ord_{\lambda_{i}}f=\beta$. Lemma \ref{Lem:pizza HT and arcs in B_{beta}} implies that $\beta_{i+1}=\beta$. If $i=0$ then $ord_{\lambda_1}f>\beta$, by Lemma \ref{Lem:minimal pizza}. Proposition \ref{Prop:width function properties elementary triangle} implies that $B_0=G(T_1)$ is a closed perfect $\beta$-zone in $B_\beta$. Furthermore, also by Proposition \ref{Prop:width function properties elementary triangle}, $B_0$ is maximal in $B_\beta$, since for every arc $\gamma \in V(T_1)\cap G(T)$, $ord_{\gamma}f=\beta$ if and only if $tord(\gamma,\lambda_1)=\beta$. Thus, when $i=0$, $B_0$ is a closed perfect $\beta$-zone maximal in $B_\beta$. Similarly, if $ord_{\lambda_p}f=\beta$ then $ord_{\lambda_{p-1}}f>\beta$, $\beta_p=\beta$ and $B_p=G(T_p)$ is a closed perfect $\beta$-zone maximal in $B_\beta$. Finally, suppose that $0<i<p$. Then, Lemma \ref{Lem:pizza HT and arcs in B_{beta}} implies that $\beta_i=\beta_{i+1}=\beta$ and Lemma \ref{Lem:minimal pizza} implies that $ord_{\lambda_{i-1}}f>\beta$ and $ord_{\lambda_{i+1}}f>\beta$. Therefore, by Proposition \ref{Prop:width function properties elementary triangle}, $B_i=G(T_i\cup T_{i+1})$ is a closed perfect $\beta$-zone maximal in $B_\beta$.
	
	$(2)$ Consider $\I=\{i_0<i_1<\cdots <i_m\}=\{l\in \Z : ord_{\lambda_l}f=\beta \}$. Then, by item $(1)$ of this Proposition, each $B_{i_j}$ is a closed perfect $\beta$-zone maximal in $B_\beta$. Moreover, by Lemma \ref{Lem:minimal pizza}, unless $p=1$, the set $\I$ does not contain consecutive integers and consequently, $B_{i_0},\ldots,B_{i_m}$ are disjoint, since there are arcs in $H_\beta$ in between each two such zones. Hence, $B_{i_0},\ldots,B_{i_m}$ are closed perfect $\beta$-zones maximal in $B_\beta$ such that $$\bigcup_{l=0}^m B_{i_l} \subset B_\beta.$$
	
	Finally, given an arc $\gamma\in B_\beta$, there exists $1\le i\le p$ such that $\gamma \in T_i$. Thus, by Lemma \ref{Lem:pizza HT and arcs in B_{beta}}, $\beta_i=\beta_{i+1}=\beta$ and either $B_{i-1}$ or $B_i$ is a closed perfect $\beta$-zone maximal in $B_\beta$ containing $\gamma$, since we have either $ord_{\lambda_{i-1}}f=\beta$ or $ord_{\lambda_{i}}f=\beta$. So, $$B_\beta = \bigcup_{l=0}^m B_{i_l}.$$
\end{proof}

\begin{Prop}\label{Prop:maximal zones in H_{beta}}
	Let $T$ and $f$ be as in Definition \ref{Def:B_beta and H_beta}. Let $\{T_{i}=T(\lambda_{i-1},\lambda_i)\}_{i=1}^p$ be a minimal pizza associated with $f$. Then
	\begin{enumerate}
		\item  For each $i\in \{1,\ldots,p-1\}$ such that $\lambda_i\in G(T)$ and $ord_{\lambda_i} f>\beta$,
		
		$H_i=\{\gamma \in G(T) : tord(\gamma,\lambda_{i})>\beta\}$ is an open $\beta$-complete zone in $H_\beta$.
		\item  For each $i\in \{1,\ldots,p\}$, if $\beta_{i}=\beta$ and $ord_{\lambda_l}f>\beta$ for $l=i-1,i$ then $H'_i=G(T_i)$ is a closed perfect $\beta$-zone in $H_\beta$.
		\item Each maximal zone $Z\subset H_\beta$ is the union of some zones as in items $(1)$ and $(2)$.
		\item The set $H_{\beta}$ is a finite union of maximal $\beta$-zones.
	\end{enumerate}
\end{Prop}
\begin{proof}
	$(1)$ This is an immediate consequence of Lemma \ref{Lem:itord>beta in H_{beta} and B_{beta} implies same ordf}.
	
	$(2)$ This follows from Proposition \ref{Prop:width function properties elementary triangle}.
	
	$(3)$ We will explicitly define all the closed perfect $\beta$-zones $Z_1,\ldots,Z_m$ maximal in $H_\beta$. To define such zones, consider the sequence $\{0=i_0<\cdots < i_m=p\}$ such that $$\{i_0,\ldots, i_m\}=\{l\in \Z  : ord_{\lambda_l}f=\beta\}\cup \{0,p\}.$$ Note that we do not necessarily have $ord_{\lambda_{i_j}}f=\beta$ for $j=0,m$.
	
	For each $j\in \{1,\ldots,m\}$ we define $T'_j=T(\lambda_{i_{j-1}},\lambda_{i_j})$. Note that each $T'_j$ is a $\beta$-H\"older triangle. We further define the set of indices $I_j=\{l\in \Z : \lambda_l \in G(T'_j), \; ord_{\lambda_l}f>\beta\}$ and the integer numbers $a_j=\min I_j$ and $b_j=\max I_j$. Note that if $ord_{\lambda_{i_{j-1}}}f=ord_{\lambda_{i_{j}}}f=\beta$ then, by Lemma \ref{Lem:minimal pizza}, $i_{j-1}$ and $i_j$ are not consecutive integers and $I_j$ is nonempty.

	First, assume that $m>1$ and define the sets of arcs $Z_j\subset V(T'_j)$ as follows (see Fig.~\ref{fig:Hbeta maximal zones}). $$Z_j=H_{a_j}\cup V(T(\lambda_{a_j},\lambda_{b_j}))\cup H_{b_j},\; \text{for each}\;
	1<j<m,$$
	
	$$
	Z_1 = \left\{
	\begin{array}{cll}
	H_{a_1}\cup V(T(\lambda_{a_1},\lambda_{b_1}))\cup H_{b_1}, & \hbox{if} & ord_{\lambda_{0}}f=\beta\\
	H'_{a_1}\cup H_{a_1}\cup V(T(\lambda_{a_1},\lambda_{b_1}))\cup H_{b_1}, & \hbox{if} &  ord_{\lambda_{0}}f>\beta \; \text{and} \; I_1\ne \emptyset \\
	\emptyset, & \hbox{if} & I_1= \emptyset
	\end{array}
	\right. $$
	and
	$$
	Z_m = \left\{
	\begin{array}{cll}
	H_{a_m}\cup V(T(\lambda_{a_m},\lambda_{b_m}))\cup H_{b_m}, & \hbox{if} & ord_{\lambda_{p}}f=\beta\\
	H_{a_m}\cup V(T(\lambda_{a_m},\lambda_{b_m}))\cup H_{b_m}\cup H'_{b_m + 1}, & \hbox{if} &  ord_{\lambda_{p}}f>\beta \; \text{and} \; I_m\ne \emptyset \\
	\emptyset, & \hbox{if} & I_m= \emptyset
	\end{array}
	\right.. $$
	
	In any of the cases above, if $a_j=b_j$ we set $T(\lambda_{a_j},\lambda_{b_j})=\lambda_{a_j}$.
	
	Now we are going to prove that, for each $1\le j \le m$, if $Z_j\ne \emptyset$ then it is a $\beta$-zone maximal in $H_\beta$. We consider three cases: $1<j<m$, $j=1$ and $j=m$.
	
	\textit{Case $1<j<m$.} In this case we have $ord_{\lambda_{i_{j-1}}}f=ord_{\lambda_{i_{j}}}f=\beta$. Thus,  $I_j$ is nonempty. So, the numbers $a_j$ and $b_j$ exist and $Z_j$ is also nonempty. Finally, note that if $a_j\ne b_j$ then $H_{a_j}\cap V(T(\lambda_{a_j},\lambda_{b_j}))\ne \emptyset$ and $V(T(\lambda_{a_j},\lambda_{b_j}))\cap H_{b_j}\ne \emptyset$ (see Fig.~\ref{fig:Hbeta maximal zones}a), and if $a_j=b_j$ then $Z_j=H_{a_j}=H_{b_j}$. In any case $Z_j$ is a zone, since the union of a sequence of finitely many zones, such that the intersection of any two consecutive such zones is nonempty, is a zone. Moreover, Proposition \ref{Prop:width function properties elementary triangle} and Lemma \ref{Lem:itord>beta in H_{beta} and B_{beta} implies same ordf} imply that $Z_j$ is maximal in $H_\beta$ since from the definition of $a_j$ and $b_j$, if $V(T'')\cap Z_j\ne \emptyset$ for a H\"older triangle $T''$ with $V(T'')\subset H_\beta$, the boundary arcs of $T''$ must both belong to $Z_j$.
	
	\textit{Case $j=1$.} We have three options: $ord_{\lambda_{0}}f=\beta$, $ord_{\lambda_{0}}f>\beta$ and $I_1\ne \emptyset$, and $I_1 = \emptyset$.
	
	If $ord_{\lambda_{0}}f=\beta$ then, using the same arguments as in case 1, we obtain that $Z_1$ is a maximal $\beta$-zone in $H_\beta$.
	
	Suppose that $ord_{\lambda_{0}}f>\beta$ and $I_1\ne \emptyset$. Since $a_1$ and $b_1$ exist, note that, $H'_{a_1}$ and $H_{a_1}\cup V(T(\lambda_{a_1},\lambda_{b_1}))\cup H_{b_1}$ are adjacent zones (see Fig.~\ref{fig:Hbeta maximal zones}b and Fig.~\ref{fig:Hbeta maximal zones}c). Then, $Z_1$ is a zone. Moreover, by the definitions of $a_1$ and $b_1$, Proposition \ref{Prop:width function properties elementary triangle} and Lemma \ref{Lem:itord>beta in H_{beta} and B_{beta} implies same ordf} imply that every arc in $G(T)\cap H_\beta$ must belong to $Z_1$. So, again $Z_1$ is a maximal $\beta$-zone in $H_\beta$.
	
	If $I_1 = \emptyset$ then, by Proposition \ref{Prop:width function properties elementary triangle}, $H_\beta \cap G(T'_1)=\emptyset$.
	
	\textit{Case $j=m$.} This case is very similar to the case $j=1$ and its proof is omitted.
	
	Second, if $m=1$ we have four options: $ord_{\lambda_0}f>\beta$ and $ord_{\lambda_p}f>\beta$, $ord_{\lambda_0}f=ord_{\lambda_p}f=\beta$, $ord_{\lambda_0}f>\beta$ and $ord_{\lambda_p}f=\beta$, and $ord_{\lambda_0}f=\beta$ and $ord_{\lambda_p}f>\beta$.
	
	If $ord_{\lambda_0}f>\beta$ and $ord_{\lambda_p}f>\beta$ then $H_\beta=G(T)$.
	
	If $ord_{\lambda_0}f=ord_{\lambda_p}f=\beta$ then, similarly as shown above in case $1<j<m$, $Z=H_{a_1}\cup V(T(\lambda_{a_1},\lambda_{b_1}))\cup H_{b_1}$ is a closed perfect $\beta$-zone maximal in $H_\beta$.
	
	If $ord_{\lambda_0}f>\beta$ and $ord_{\lambda_p}f=\beta$ then either $H_\beta=\emptyset$ if $I_1=\emptyset$ or $H'_{a_1}\cup H_{a_1}\cup V(T(\lambda_{a_1},\lambda_{b_1}))\cup H_{b_1}$ is a closed perfect $\beta$-zone maximal in $H_\beta$ otherwise.
	
	If $ord_{\lambda_0}f=\beta$ and $ord_{\lambda_p}f>\beta$ then either $H_\beta=\emptyset$ if $I_1=\emptyset$ or $H_{a_1}\cup V(T(\lambda_{a_1},\lambda_{b_1}))\cup H_{b_1}\cup H'_{b_1 + 1}$ is a closed perfect $\beta$-zone maximal in $H_\beta$ otherwise.
	
	Finally, since $Z_1, \ldots , Z_m$ are disjoint zones maximal in $H_\beta$, any maximal zone in $H_\beta$ coincide with one of those. 
	
	$(4)$ By item $(3)$ of this Proposition, $H_\beta=\bigcup_{j=1}^mZ_j$.

	\begin{figure}
		\centering
		\includegraphics[width=4.5in]{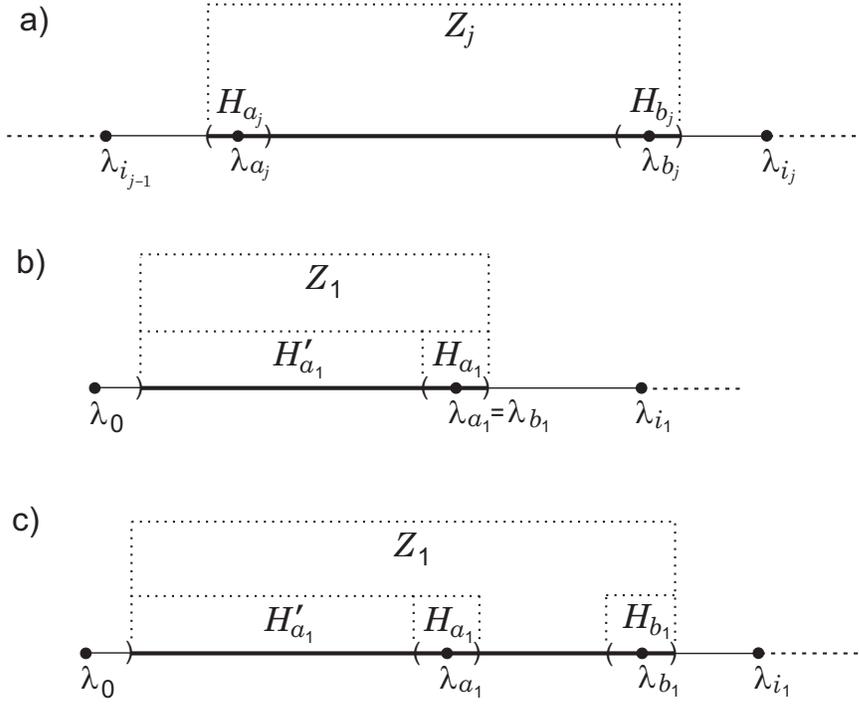}
		\caption{Several cases in the proof of Proposition 3.7:\newline a) $1<j<m$; b) $j=1,\; a_1=b_1$; c) $j=1,\; a_1<b_1$.
			Zones $Z_j$, $H_i$ and $H'_i$ are indicated by dotted lines. The ``open intervals'' containing $\lambda_i$
			represent open $\beta$-complete zones.}\label{fig:Hbeta maximal zones}
	\end{figure}
\end{proof}

\section{Snakes}\label{Section: Snakes}

In this section we define snakes, one of the main objects of this paper. A $\beta$-snake is an abnormal surface germ which is a $\beta$-H\"older triangle. We define a canonical partition of the Valette link of a $\beta$-snake into segments and nodal zones. All segments of a $\beta$-snake are closed perfect $\beta$-zones, and all its nodal zones are open $\beta$-complete. A node of a $\beta$-snake is defined as the union of its nodal zones having tangency order higher than $\beta$. We consider relations between pancake decompositions of a snake and its segments and nodes.

\subsection{Snakes and their pancake decomposition}\label{Subsec: Snakes and their pancakes}

\begin{Def}\label{Def:snake}
	\normalfont A non-singular $\beta$-H\"older triangle $T$ is called a $\beta$\textit{-snake} if $T$ is an abnormal surface (see Definition \ref{Def: abnormal surface}).
\end{Def}

\begin{remark}\label{Rem: abnormal arcs of a snake}
	\normalfont It follows from Definition \ref{Def:snake} and Remark \ref{Rem: generic arcs of a non-singular HT} that each normal arc in $T$ has inner tangency order higher than $\beta$ with one of its boundary arcs, and each abnormal arc in $T$ has inner tangency order $\beta$ with both boundary arcs.
\end{remark}

\begin{Exam}\label{Exam: explanation about snakes links' pictures}
	\normalfont
	A snake with the link as in Fig.~\ref{fig:three snakes}a is a bubble snake (see Definition \ref{Def:bubble} below). A snake with the link as in Fig.~\ref{fig:three snakes}b is a binary snake, while a snake with the link as in Fig.~\ref{fig:three snakes}c is not (see Definition \ref{Def: binary snake} below). We use planar pictures to represent the links of snakes. Points in the picture correspond to arcs in a snake with the given link. Although the Euclidean distance in the link's picture does not accurately translate the tangency order of arcs in the snake with the given link, we will often use it so that points with smaller Euclidean distance in the picture correspond to arcs in the snake with higher tangency order. For example, points inside the shaded disks correspond to arcs with the tangency order higher than $\beta$.
\end{Exam}

\begin{remark}\label{REM: different number of pancakes in minimal pancake decompositions}
	\normalfont Note that minimal generic pancake decompositions of a snake may have different number of pancakes. For example, one of the two minimal pancake decompositions of the snake Fig.~\ref{fig:three snakes}b shown in Fig.~\ref{fig:Two pancake decomposition of a snake} has two pancakes while the other one has three.
\end{remark}

\begin{remark}\label{REM: circular snake}
	\normalfont One can also define a \textit{circular snake} as a surface with connected link such that any arc in it is abnormal (in particular, each of its arcs is Lipschitz non-singular).
The simplest circular snake is a normally embedded surface germ with the link homeomorphic to a circle.
A link of a circular snake that is not normally embedded is shown in Fig.~\ref{fig:circ-snakes}.
Circular snakes are not discussed in this paper, although they appear in Example
\ref{Exam: curve in C2}, and in Theorem \ref{Teo: Main HT decomposition}.
\end{remark}

\begin{figure}
	\centering
	\includegraphics[width=4.5in]{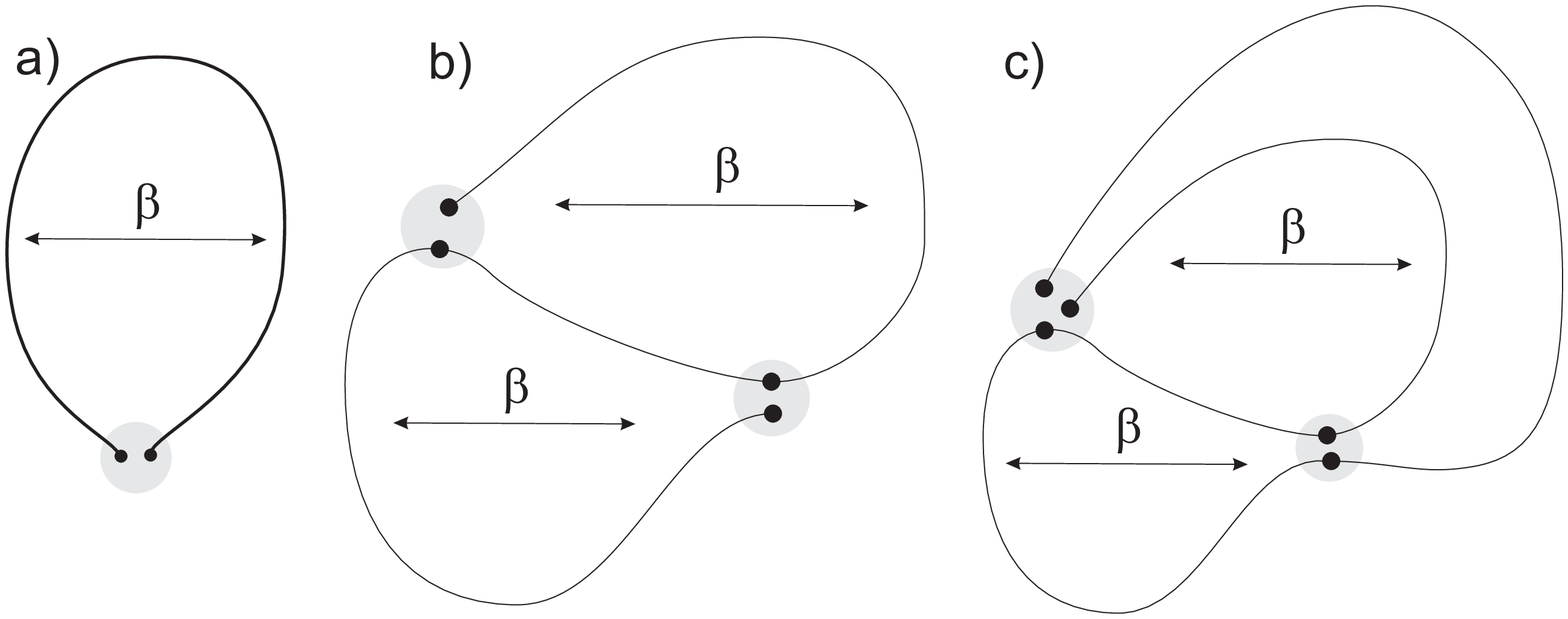}
	\caption{Three links of $\beta$-snakes: a) a bubble snake; b) a binary snake; c) a non-binary snake. Shaded disks represent arcs with the tangency order higher than $\beta$.}\label{fig:three snakes}
\end{figure}

\begin{figure}
	\centering
	\includegraphics[width=4.5in]{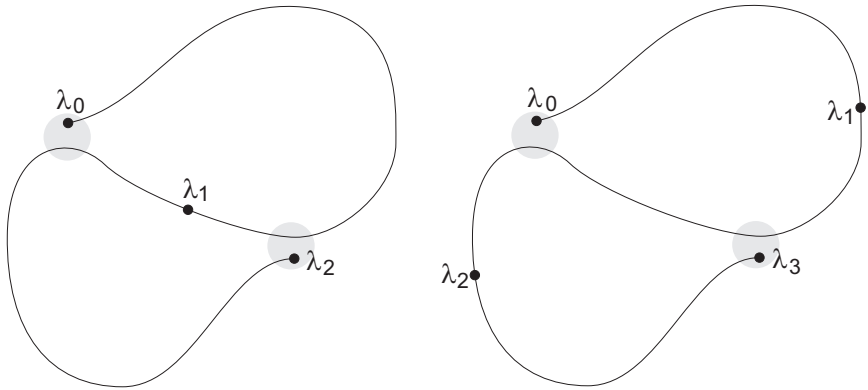}
	\caption{Two minimal pancake decompositions of the snake in Fig.~\ref{fig:three snakes}b. Black dots indicate the boundary arcs of pancakes.}\label{fig:Two pancake decomposition of a snake}
\end{figure}

\begin{figure}
	\centering
	\includegraphics[width=3in]{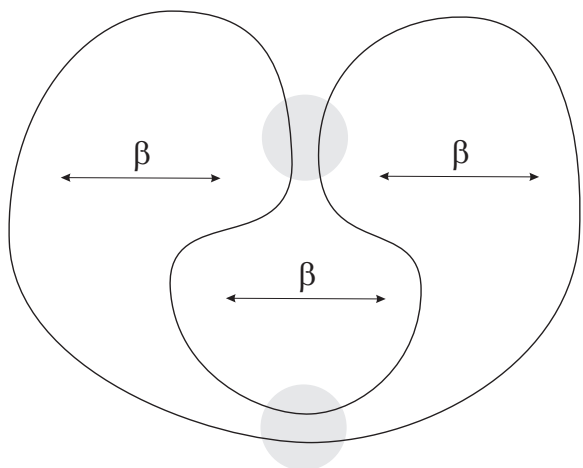}
	\caption{An example of a circular $\beta$-snake. Shaded disks represent arcs with the tangency order higher than $\beta$.}\label{fig:circ-snakes}
\end{figure}

\begin{figure}
	\centering
	\includegraphics[width=3in]{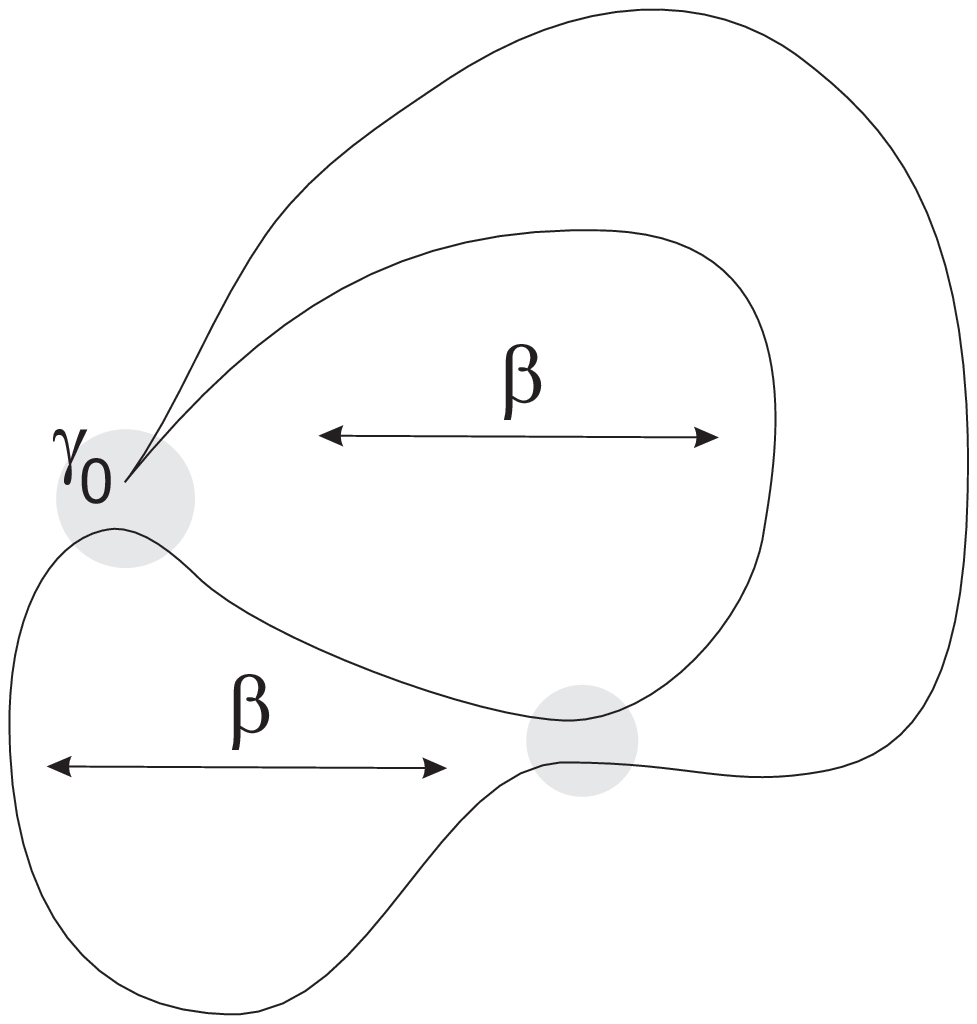}
	\caption{The link of a cusp snake with a singular arc $\gamma_0$. Shaded disks represent arcs with tangency order higher than $\beta$.}\label{fig:cusp snake}
\end{figure}

\begin{Exam}\label{Exam: cusp snake}	
	\normalfont A circular link (bi-Lipschitz homeomorphic to a circle with respect to the inner metric) of an abnormal $\beta$-surface $X$ with $Lsing(X)=\gamma_0$ is shown in Fig.~\ref{fig:cusp snake}. Note that $X$ is not a snake, since it is not even a H\"older triangle. Despite its circular link, $X$ is not a circular snake as well, since it contains the singular arc $\gamma_0$. One can obtain a snake $T\subset X$ as follows. Consider arcs $\gamma_1\ne\gamma_0$ and $\gamma_2\ne\gamma_0$ in $X$ such that $itord(\gamma_0,\gamma_1)=itord(\gamma_0,\gamma_2)=\alpha>\beta$ and $T=T(\gamma_{1},\gamma_{2})\subset (X\setminus \gamma_0)\cup \{0\}$ is a $\beta$-H\"older triangle. Then $T$ is a $\beta$-snake with link as shown in Fig.~\ref{fig:three snakes}c. This surface $X$ may be considered as a snake with both boundary arcs equal the singular arc $\gamma_0$. We call such a surface a cusp snake.
\end{Exam}

\begin{Lem}\label{Lem: maximal abnormal zone alpha le beta}
	Let $X$ be a surface germ and $A\subset V(X)$ a maximal abnormal $\beta$-zone. Let $\gamma \in A$, and let $T=T(\lambda,\gamma)\subset X$ and $T'=T(\gamma,\lambda')\subset X$ be normally embedded $\alpha$-H\"older triangles such that $T\cap T'=\gamma$ and $tord(\lambda,\lambda')>itord(\lambda,\lambda')$. Then $\alpha\le \beta$.
\end{Lem}
\begin{proof}
	Suppose, by contradiction, that $\alpha >\beta$.
	Since $A$ has order $\beta$, we can assume, without loss of generality, that there is a non-singular $\beta$-H\"older triangle $T_0=T(\gamma_0,\gamma)\subset X$ such that $T_0\cap T'=\gamma$ and $V(T_0)\cap A$ is a $\beta$-zone.
	In particular, if $\theta \subset T_0$ is any arc such that $itord(\theta,\gamma)>\beta$, then $\theta \in A$. Since $itord(\lambda,\gamma)=\alpha >\beta$, we have $\lambda \in A$. Let $\theta_0=\lambda$ and $\theta'_0=\lambda'$. Then, there are normally embedded $\alpha_1$-H\"older triangles $T_1=T(\theta_1,\theta_0)\subset X$ and $T'_1=T(\theta_0,\theta'_1)\subset X$  such that $T(\gamma_0,\theta_0)\cap T'_1=T_1\cap T'_1 = \theta_0$ and $tord(\theta_1,\theta'_1)>itord(\theta_1,\theta'_1)$.
	Since $T\cup T'$ is not normally embedded, we have $\theta'_1\subset T\cup T'$.
	
	Note that $\alpha_1 >\beta$. Indeed, if $\theta'_1 \subset T$ then, by non-archimedean property, $\alpha_1=itord(\theta_0,\theta'_1)\ge itord(\theta_0,\gamma)=\alpha >\beta$. If $\theta'_1 \subset T'$ we use that $\alpha=itord(\gamma,\theta'_0)$ and apply non-archimedean property again. Thus, $\theta_1 \in A$ and there are normally embedded $\alpha_2$-H\"older triangles $T_2=T(\theta_2,\theta_1)\subset X$ and $T'_2=T(\theta_1,\theta'_2)\subset X$ such that $T(\gamma_0,\theta_1)\cap T'_2=T_2\cap T'_2 = \theta_1$ and $tord(\theta_2,\theta'_2)>itord(\theta_2,\theta'_2)$.
	Since $T_1\cup T'_1$ is not normally embedded, we have $\theta'_2\subset T_1\cup T'_1$. Similarly, we prove that $\alpha_2\ge \alpha_1 > \beta$ and obtain that $\theta_2 \in A$.
	
	Continuing this procedure for $i>2$, we see that $\alpha_i\ge \cdots \ge \alpha_1>\beta$, thus $\theta_i\in A$ and there are normally embedded $\alpha_{i+1}$-H\"older triangles $T_{i+1}=T(\theta_{i+1},\theta_i)$ and $T'_{i+1}=T(\theta_i,\theta'_{i+1})$ such that $T_{i+1}\cap T'_{i+1} = \theta_i$, $tord(\theta_{i+1},\theta'_{i+1})>itord(\theta_{i+1},\theta'_{i+1})$ (see Fig.~\ref{fig: maximal abnormal zone alpha le beta}).
	
	Given a minimal pancake decomposition of $X$, for each $i\ge 0$, $\theta'_i$ and $\theta_{i+2}$ belong to different pancakes, since $T(\theta_i,\theta'_i)\subset T(\theta_{i+2},\theta'_i)$ is not normally embedded. Since there are only finitely many pancakes in a minimal pancake decomposition, our procedure must stop after finitely many steps, in contradiction with $\alpha>\beta$.
\end{proof}
\begin{figure}
	\centering
	\includegraphics[width=4in]{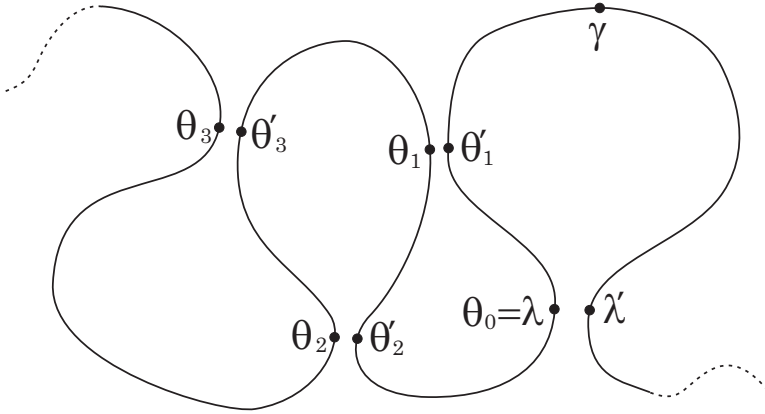}
	\caption{Construction in the proof of Lemma \ref{Lem: maximal abnormal zone alpha le beta}. Each pair $\theta_i$, $\theta'_i$ has tangency order higher than its inner tangency order.}\label{fig: maximal abnormal zone alpha le beta}
\end{figure}

\begin{Lem}\label{Lem:minimal pancake decomp of a snake}
	Let $X$ be a $\beta$-snake, and let $\{X_k\}_{k=1}^p$ be a minimal pancake decomposition of $X$.
Then each $X_k$ is a $\beta$-H\"older triangle.
\end{Lem}
\begin{proof}
	We may assume that $X=T(\lambda_{0},\lambda_{p})$ and $X_{k}=T(\lambda_{k-1},\lambda_{k})$ for each $k$, thus $X_{k}\cap X_{k+1}=\lambda_{k}$.
Let $\mu(X_{k}) = tord(\lambda_{k-1},\lambda_{k}) = \beta_{k}\ge\beta$.
	
	We prove first that $\beta_1=\beta$. If $\beta_{1}>\beta$ then, by the definition of a $\beta$-snake, $\lambda_{1}$ is a normal arc. However, since $X_{1}$ and $X_{2}$ are two normally embedded H\"older triangles such that $X_{1}\cap X_{2}=\lambda_{1}$ and, by the definition of a minimal pancake decomposition, $X_{1}\cup X_{2}$ is not normally embedded, $\lambda_{1}$ is also abnormal, which is a contradiction.
By a similar argument we can prove that $\beta_{p}=\beta$.
	
	Since $\beta_1=\beta_p=\beta$, we have $\lambda_j \in G(X)$ for $1<j<p$ by the definition of a $\beta$-snake.
Since $\{X_{k}\}$ is a minimal pancake decomposition, $X_j\cup X_{j+1}$ is not normally embedded,
thus there are normally embedded $\alpha$-H\"older triangles $T=T(\lambda,\lambda_j)\subset X_j$ and $T'=T(\lambda_j,\lambda')\subset X_{j+1}$ such that
$T\cap T'=\lambda_j$ and $tord(\lambda,\lambda')>itord(\lambda,\lambda')=\alpha$.
By Lemma \ref{Lem: maximal abnormal zone alpha le beta}, we have $\alpha \le \beta$.
Since $X$ is a $\beta$-snake and $T\subset X_j$, we have $\alpha\le \beta\le\beta_j\le\alpha$. Thus $\beta_j=\beta$ also for $1<j<p$.
\end{proof}

\begin{Def}\label{Def: weak LNE}
\normalfont A $\beta$-H\"older triangle $X$ is \textit{weakly normally embedded} if, for any two arcs $\gamma$ and $\gamma'$ in $X$ such that $tord(\gamma, \gamma') > itord(\gamma, \gamma')$, we have $itord(\gamma, \gamma')=\beta$.
\end{Def}

\begin{Prop}\label{Prop:weak LNE}
	Let $X$ be a $\beta$-snake. Then $X$ is weakly normally embedded.
\end{Prop}
\begin{proof}
	Let $\gamma$ and $\gamma'$ be two arcs in $X$ such that $tord(\gamma, \gamma') > itord(\gamma, \gamma')$.
Consider a minimal pancake decomposition $\{X_{k}\}_{k=1}^{p}$ of $X$.
Since each pancake is normally embedded, $\gamma$ and $\gamma'$ do not belong to the same  pancake. If $\gamma$ and $\gamma'$ are not in adjacent pancakes, then Lemma \ref{Lem:minimal pancake decomp of a snake} implies that  $itord(\gamma, \gamma')=\beta$.
Let us assume that $\gamma \subset X_{j-1}$ and $\gamma' \subset X_{j}$ for some $j>1$. Consider $T=T(\gamma,\lambda_{j-1})\subset X_{j-1}$ and $T'=T(\lambda_{j-1},\gamma')\subset X_j$.
Note that both $T$ and $T'$ are normally embedded, since each of them is contained in a pancake. Let $\alpha = itord(\gamma, \gamma')$.
Since $tord(\gamma, \gamma')>\alpha$, Lemma \ref{Lem:beta-bubble two NE pieces} implies that $\mu(T)=\mu(T')=\alpha$.
Since $\lambda_{j-1}\in G(X)$ by Lemma \ref{Lem:minimal pancake decomp of a snake},
Lemma \ref{Lem: maximal abnormal zone alpha le beta} implies that $\alpha \le \beta$.
Since $X$ is a $\beta$-snake and $T(\gamma,\gamma')\subset X$, we have also $\alpha\ge\beta$. Thus $\alpha=\beta$.
\end{proof}

\subsection{Segments and nodes}\label{Subsec: Segments and nodes}

\begin{Def}\label{Def:horn-neighbourhood}
	\normalfont Let $X$ be a surface and $\gamma \subset X$ an arc. For $a>0$ and $1\le\alpha\in\F$, the $(a,\alpha)$-\textit{horn neighborhood} of $\gamma$ in $X$ is defined as follows: $$HX_{a,\alpha}(\gamma) = \bigcup_{0\le t \ll1} X\cap S(0,t)\cap \overline{B}(\gamma(t),at^{\alpha}),$$
	where $S(0,t)=\{x \in \mathbb{R}^{n} : |x|=t\}$ and $\overline{B}(y,R) = \{x \in \mathbb{R}^{n} : |x - y|\le R\}$.
\end{Def}

\begin{remark}\label{Rmk:horn-neighborhood}
	\normalfont When there is no confusion about the surface $X$ being considered, one writes $H_{a,\alpha}(\gamma)$ instead of $HX_{a,\alpha}(\gamma)$.
\end{remark}

\begin{Def}\label{Def of multiplicity}
	\normalfont If $X$ is a $\beta$-snake and $\gamma$ an arc in $X$, the \textit{multiplicity} of $\gamma$, denoted by $m_{X}(\gamma)$ (or just $m(\gamma)$), is defined as the number of connected components of $HX_{a,\beta}(\gamma)\setminus \{0\}$ for $a>0$ small enough.
\end{Def}

\begin{remark}\normalfont
	Since $X$ is definable in an o-minimal structure, the family of sets $\{HX_{a,\beta}(\gamma)\}_{a>0}$ in Definition \ref{Def of multiplicity} is also definable. In particular, the number of connected component of this set is constant for small $a>0$.
\end{remark}

\begin{Lem}\label{Lem:exist arc in horn connected component}
	Let $X$ be a surface, $\gamma \subset X$ an arc and $Y\subset X$ a closed set. If, for $a>0$ sufficiently small, $Y\cap HX_{a,\alpha}(\gamma)\ne\{0\}$, then there is an arc $\gamma'\subset Y$ such that $tord(\gamma,\gamma')>\alpha$.
\end{Lem}
\begin{proof}
	Let $Y_t=S(0,t)\cap Y$ and $M_t=\{x \in Y_t : d(\gamma(t),Y_t)=d(\gamma(t),x)\}$. Each set $M_t$ is definable, and so is $M=\bigcup_{0\le t} M_t$. By the Arc Selection Lemma there exists an arc $\gamma' \subset M \subset Y$.
	
	If for each arc $\gamma' \subset M$ we have $tord(\gamma,\gamma')=\alpha$ then, for $a>0$ sufficiently small, $\gamma' \not\subset Y\cap H_{a,\alpha}(\gamma)$, a contradiction with $Y\cap H_{a,\alpha}(\gamma)\ne\{0\}$.
\end{proof}

\begin{Prop}\label{Prop:connected comp of horn neigbourhood}
	Let $X$ be a surface, $T\subset X$ a normally embedded $\beta$-H\"older triangle and $\gamma \subset X$ an arc.
Then, for $1\le \alpha\in\F$ and $a>0$ sufficiently small, the link of $T\cap HX_{a,\alpha}(\gamma)$ is connected.
\end{Prop}
\begin{proof}
	Let $H=HX_{a,\alpha}(\gamma)$. If $\alpha<\beta$ and $T\cap H\ne \{0\}$ for $a>0$ sufficiently small, then there is an arc $\gamma' \subset T$ such that $tord(\gamma',\gamma'')>\alpha$, by Lemma \ref{Lem:exist arc in horn connected component}.
	This implies that $T\subset H$, thus the link of $T\cap H=T$ is connected.
	
	Suppose that $\alpha\ge\beta$ and, for $a>0$ sufficiently small, the link of $T\cap H$ is not connected.
	Let $C$ and $C'$ be two distinct connected components of $(T\cap H)\setminus \{0\}$. By Lemma \ref{Lem:exist arc in horn connected component}, for small enough $a$, there exist arcs $\gamma' \subset C$ and $\gamma'' \subset C'$ such that $$tord(\gamma',\gamma'')\ge \min(tord(\gamma,\gamma'),tord(\gamma,\gamma''))>\alpha.$$
	
	Consider $T'=T(\gamma',\gamma'') \subset T$. As $\gamma'$ and $\gamma''$ are in different connected components, there exists an arc $\lambda \subset (T'\setminus H)\cup\{0\}$. Thus, $itord(\gamma',\gamma'')\le\alpha$, a contradiction with $T$ being normally embedded.
\end{proof}

\begin{Cor}\label{Cor:intersection of horn neighbourhood with Holder triangle}
	Let $X$ be a surface, $T\subset X$ a normally embedded H\"older triangle and $\gamma\subset X$ an arc.
Then, for $1\le\alpha\in\F$ and $a>0$ sufficiently small, either $T\cap HX_{a,\alpha}(\gamma)= \{0\}$ or it is a H\"older triangle.
\end{Cor}

\begin{Def}\label{DEF: constant zone}
	\normalfont Let $X$ be a $\beta$-snake and $Z \subset V(X)$ a zone. We say that $Z$ is a \textit{constant zone} of multiplicity $q$ (notation $m(Z)=q$) if all arcs in $Z$ have the same multiplicity $q$.
\end{Def}

\begin{Def}\label{Def:segment-nodal arcs-zones}
	\normalfont Let $X$ be a $\beta$-snake and $\gamma \subset X$ an arc. We say that $\gamma$ is a \textit{segment arc} if there exists a $\beta$-H\"older triangle $T \subset X$ such that $\gamma$ is a generic arc of $T$ and $V(T)$ is a constant zone. Otherwise $\gamma$ is a \textit{nodal arc}. We denote the set of segment arcs and the set of nodal arcs in $X$ by $\mathbf{S}(X)$ and $\mathbf{N}(X)$, respectively. A \textit{segment} of $X$ is a maximal zone in $\mathbf{S}(X)$. A \textit{nodal zone} of $X$ is a maximal zone in $\mathbf{N}(X)$. We write $Seg_{\gamma}$ or $Nod_{\gamma}$ for a segment or a nodal zone containing an arc $\gamma$.
\end{Def}

\begin{Prop}\label{Prop:segment is a perfect zone}
	If $X$ is a $\beta$-snake then each segment of $X$ is a closed perfect $\beta$-zone.
\end{Prop}
\begin{proof}
	Given an arc $\gamma$ in a segment $S$ of $X$, by Definition \ref{Def:segment-nodal arcs-zones}, there exists a $\beta$-H\"older triangle $T=T(\gamma_1,\gamma_2)$ such that $\gamma$ is a generic arc of $T$ and $V(T)\subset S$ is a constant zone. Let $\gamma'_{1}$ and $\gamma'_{2}$ be generic arcs of $T(\gamma_{1},\gamma)$ and $T(\gamma,\gamma_{2})$, respectively. It follows that $T'=T(\gamma'_{1},\gamma'_{2})$ is a $\beta$-H\"older triangle such that $\gamma$ is a generic arc of $T'$ and $V(T')\subset S$.
\end{proof}

\begin{remark}\label{Rem:nodal zone is open perfect}\normalfont
Similar arguments show that each nodal zone of a $\beta$-snake $X$ which does not contain any boundary arcs of $X$ is an open perfect $\beta$-zone.
\end{remark}

\begin{Lem}\label{Lem:multipizza and generic arcs of multipizza triangle}
	Let $X$ be a $\beta$-snake and $\{X_{k}\}_{k=1}^{p}$ a pancake decomposition of $X$. Let $T=X_{j}$ be one of the pancakes and consider the set of germs of Lipschitz functions $f_{l}\colon (T,0) \rightarrow (\mathbb{R},0)$ given by $f_{l}(x)=d(x,X_{l})$. If $\{T_{i}\}$ is a multipizza on $T$ associated with $\{f_{1},\ldots, f_{p}\}$ then, for each $i$, the following holds:
	\begin{enumerate}
		\item $\mu_{il}(ord_{\gamma}f_{l})=\beta_{i}$ for all $l$ and all $\gamma \in G(T_{i})$, thus $G(T_{i})$ is a constant zone.
		\item $V(T_{i})$ intersects at most one segment of $X$.
		\item If $V(T_i)$ is contained in a segment then it is a constant zone.
	\end{enumerate}	
\end{Lem}
\begin{proof}
	$(1)$. This is an immediate consequence of Definition \ref{Def:Pizza decomp} and Proposition \ref{Prop:width function properties elementary triangle}.
	
	$(2)$. If $\beta_{i}>\beta$ and $V(T_{i})$ intersects a segment $S$, then $V(T_{i})\subset S$, since $S$ is a closed perfect $\beta$-zone, by Proposition \ref{Prop:segment is a perfect zone}.
	
	Let $\beta_{i}=\beta$. Suppose that $V(T_{i})$ intersects distinct segments $S$ and $S'$. As each segment is a closed perfect $\beta$-zone, we can choose arcs $\lambda \in S$ and $\lambda' \in S'$ so that $\lambda, \lambda' \in G(T_{i})$. Let $T'=T(\lambda,\lambda')$. By item $(1)$ of this Lemma, all arcs in $G(T')$ have the same multiplicity. It follows from Definition \ref{Def:segment-nodal arcs-zones} that each arc in $T'$ is a segment arc. Thus, $S$ and $S'$ belong to the same segment, a contradiction.
	
	$(3)$ This a consequence of Definition \ref{Def:segment-nodal arcs-zones} and item $(1)$ of this Lemma. 
\end{proof}

\begin{Prop}\label{Prop:segment has finitely many zones}
	Let $X$ be a $\beta$-snake. Then
	\begin{enumerate}
		\item There are no adjacent segments in $X$.
		\item $X$ has finitely many segments.
	\end{enumerate}
	
\end{Prop}
\begin{proof}
	$(1)$ This is an immediate consequence of Proposition \ref{Prop:segment is a perfect zone} and Lemma \ref{Lem: there are no adjacent perfect zones}.
	
	$(2)$ Let $\{X_{k}\}_{k=1}^{p}$ be a pancake decomposition of $X$. It is enough to show that, for each pancake $X_{j}$, $V(X_{j})$ intersects with finitely many segments. But this follows from Lemma \ref{Lem:multipizza and generic arcs of multipizza triangle}, since there are finitely many H\"older triangles in a multipizza.
\end{proof}

\begin{Lem}\label{Lem:itord > beta imples same multiplicity}
	Let $X$ be a $\beta$-snake. Then, any two arcs in $V(X)$ with inner tangency order higher than $\beta$ have the same multiplicity.
\end{Lem}
\begin{proof}
	Let $\{X_{k}\}_{k=1}^{p}$ be a pancake decomposition of $X$, $T=X_{j}$ one of the pancakes and $\{T_{i}\}$ a multipizza associated with $\{f_1,\ldots, f_p\}$ as in Lemma \ref{Lem:multipizza and generic arcs of multipizza triangle}. Consider arcs $\gamma$ and $\gamma'$ in $V(X)$ such that $itord(\gamma,\gamma')>\beta$ and $\gamma \in V(T)$. We can suppose that $\gamma,\gamma' \in V(T)$, otherwise we can just replace $\gamma'$ by the boundary arc of $T$ in $T(\gamma,\gamma')$.
	
	It is enough to show that for each $l$ we have $ord_{\gamma}f_{l}>\beta$ if and only if $ord_{\gamma'}f_{l}>\beta$. This follows from Lemma \ref{Lem:itord>beta in H_{beta} and B_{beta} implies same ordf}.
\end{proof}

\begin{Cor}\label{Cor:segments and nodal zones are constant zones}
	Let $X$ be a $\beta$-snake. Then, all segments and all nodal zones of $X$ are constant zones.
\end{Cor}
\begin{proof}
	Let $\{X_k\}$ be a minimal pancake decomposition of $X$, and $\{T_i\}$ a multipizza on $T=T_j$ associated with $\{f_{1},\ldots, f_{p}\}$ as in Lemma \ref{Lem:multipizza and generic arcs of multipizza triangle}.
	
	Let $S$ be a segment of $X$. Consider two arcs $\gamma,\gamma'\in S$. Replacing, if necessary, one of the arcs $\gamma$, $\gamma'$ by one of the boundary arcs of $T$,  we can assume that $\gamma,\gamma'\in V(T)$. Thus, if $\gamma \in T_i$ and $\gamma'\in T_{i+l}$, for some $l\ge 0$, it follows from Lemma \ref{Lem:multipizza and generic arcs of multipizza triangle} that $m(V(T_i))=m(V(T_{i+1}))=\cdots=m(V(T_{i+l}))$ and consequently, $m(\gamma)=m(\gamma')$.
	
	Let now $N$ be a nodal zone of $X$. Consider two arcs $\gamma,\gamma'\in N$ and assume, without loss of generality, that $T'=T(\gamma,\gamma')\subset T$. If $itord(\gamma,\gamma')=\beta$ then $G(T')\cap G(T_i)\ne \emptyset$ for some $i$ such that $\beta_i=\beta$. As $G(T')$ and $G(T_i)$ are closed perfect $\beta$-zones, $G(T')\cap G(T_i)$ is also a closed perfect $\beta$-zone. Lemma \ref{Lem:multipizza and generic arcs of multipizza triangle} implies that $G(T')\cap G(T_i)$ contains a segment arc, since $G(T_i)$ is a constant zone, a contradiction with $V(T')\subset N$. Thus, $itord(\gamma,\gamma')>\beta$ and $m(\gamma)=m(\gamma')$ by Lemma \ref{Lem:itord > beta imples same multiplicity}.
\end{proof}

\begin{remark}\label{Rem:non-closed zones are constant zones}
	\normalfont If $X$ is a $\beta$-snake then any open $\beta$-zone $Z$ in $V(X)$, and any zone $Z'$ of order $\beta'>\beta$, is a constant zone.
\end{remark}

\begin{Prop}\label{Prop:nodal zones are finite}
	Let $X$ be a $\beta$-snake. Then
	\begin{enumerate}
		\item  For any nodal arc $\gamma$ we have $Nod_{\gamma}=\{\gamma'\in V(X) : itord(\gamma,\gamma')>\beta\}$. In particular, a nodal zone is an open $\beta$-complete zone.
		\item There are no adjacent nodal zones.
		\item There are finitely many nodal zones in $V(X)$.
	\end{enumerate}
\end{Prop}
\begin{proof}
	$(1)$ Let $\gamma \in V(X)$ be a nodal arc. Given $\gamma'\in V(X)$, if $itord(\gamma,\gamma')=\beta$ then $\gamma'\notin Nod_{\gamma}$. Indeed, if $\gamma'\in Nod_{\gamma}$ and $itord(\gamma,\gamma')=\beta$ then, since $V(T(\gamma,\gamma'))\subset Nod_{\gamma}$ and $Nod_{\gamma}$ is a constant zone, by Corollary \ref{Cor:segments and nodal zones are constant zones}, every arc in $G(T(\gamma,\gamma'))$ is a segment arc, a contradiction with $Nod_{\gamma}$ being a zone. Thus, a nodal zone is completely determined by any one of its arcs, i.e., $Nod_{\gamma}=\{\gamma'\in V(X) : itord(\gamma,\gamma')>\beta\}$. Therefore, any nodal zone is an open $\beta$-complete zone.
	
	$(2)$ This is an immediate consequence of $(1)$ and Remark \ref{Rem:open complete zones}.

	$(3)$ It is a consequence of Proposition \ref{Prop:segment has finitely many zones} and item $(2)$ of this Proposition.
\end{proof}

\begin{Cor}\label{Cor: V(X) is a disjoint union of segments and nodal zones}
	If $X$ is a snake then $V(X)$ is a disjoint union of finitely many segments and nodal zones.
\end{Cor}

\begin{Def}\label{DEF: boundary and interior nodal zones}
	\normalfont Let $X=T(\gamma_{1},\gamma_{2})$ be a $\beta$-snake. By Definition \ref{Def:segment-nodal arcs-zones}, the boundary arcs $\gamma_{1}$ and $\gamma_{2}$ of $X$ are nodal arcs. The nodal zones $Nod_{\gamma_{1}}$ and $Nod_{\gamma_{2}}$ are called the \textit{boundary nodal zones}. All other nodal zones are called \textit{interior nodal zones}.
\end{Def}

\begin{Prop}\label{Prop:segment is the generic arcs from HD of nodal adjacent zones}
	Let $X$ be a $\beta$-snake. Then, each interior nodal zone in $X$ has exactly two adjacent segments, and each segment in $X$ is adjacent to exactly two nodal zones. Moreover, if $N$ and $N'$ are the nodal zones adjacent to a segment $S$, then for any arcs $\gamma \subset N$ and $\gamma' \subset N'$, we have $S = G(T(\gamma,\gamma'))$.
\end{Prop}
\begin{proof}
	Propositions \ref{Prop:nodal zones are finite} and \ref{Prop:segment has finitely many zones} imply that each nodal zone in $V(X)$ could only be adjacent to a segment $S$, and vice versa.
	
	Finally, let $N$ and $N'$ be the two nodal zones adjacent to $S$ and let $\gamma \in N$ and $\gamma'\in N'$. Since each arc in $T(\gamma,\gamma')$ which has tangency order higher than $\beta$ with one of the boundary arcs is a nodal arc, by Proposition \ref{Prop:nodal zones are finite}, each segment arc in $T(\gamma,\gamma')$ must be in $G(T(\gamma,\gamma'))$, and vice versa.
\end{proof}

\begin{Def}\label{Def: node}
	\normalfont Let $X$ be a $\beta$-snake. A \textit{node} $\N$ in $X$ is a union of nodal zones in $X$ such that for any nodal zones $N,N'$ with $N\subset \N$ then $N'\subset \N$ if and only if $tord(N,N')>\beta$. Given a node $\N=\bigcup_{i=1}^m N_i$, where $N_i$ are the nodal zones in $\N$, the set $Spec(\N)=\{q_{ij}=tord(N_i,N_j) : i\ne j\}$ is called the \textit{spectrum} of $\N$.
\end{Def}

\subsection{Clusters and cluster partitions}\label{Subsection: clusters}

\begin{Def}\label{Def: clusters}
\normalfont Let $\N$ and $\N'$ be nodes of a $\beta$-snake $X$, and let $\mathcal{S}(\N,\N')$ be the (possibly empty) set of all segments of $X$ having adjacent nodal zones in the nodes $\N$ and $\N'$ (see Proposition \ref{Prop:segment is the generic arcs from HD of nodal adjacent zones}). Two segments $S$ and $S'$ in $\mathcal{S}(\N,\N')$ belong to the same \textit{cluster}
if $tord(S,S')>\beta$. This defines a \textit{cluster partition} of $\mathcal{S}(\N,\N')$.
The size of each cluster $C$ of this partition is equal to the multiplicity of each segment $S\in C$ (see Definition \ref{DEF: constant zone}).
\end{Def}

\begin{remark}\label{Rmk:clusters}
\normalfont Proposition \ref{Prop: spiral snake and its segments} below implies that all segments of a spiral snake $X$ belong to the same cluster. If $X$ is not a spiral snake, then Proposition \ref{Prop: HT in consecutive segments} below implies that any two segments of $X$
adjacent to the same nodal zone do not belong to the same cluster.
\end{remark}

\begin{Exam}\label{Exam: curve in C2}
	\normalfont Given relatively prime natural numbers $p$ and $q$, where $1<p<q$, the germ at the origin of the complex curve $X=\{y^p=x^q\}\subset\C^2$, considered as a real surface in $\R^4$, is an example of a circular 1-snake with a single segment and no nodes.
	Removing from $X$ the H\"older triangle $T=\{(x,y)\in \C^2 : 0\le arg(x)\le \pi/q,\;0\le arg(y)\le \pi/p \}$, and taking the closure, one obtains a 1-snake $X'$ with $p$ segments of multiplicity $p$ and $p-1$ segments of multiplicity $p-1$. Each of the two nodes $\N$ and $\N'$ of $X'$ has multiplicity $p$, and its spectrum consists of a single exponent $q/p$.
The set $\mathcal{S}(\N,\N')$ is partitioned into two clusters of sizes $p$ and $p-1$.
\end{Exam}

\begin{figure}
	\centering
	\includegraphics[width=4.5in]{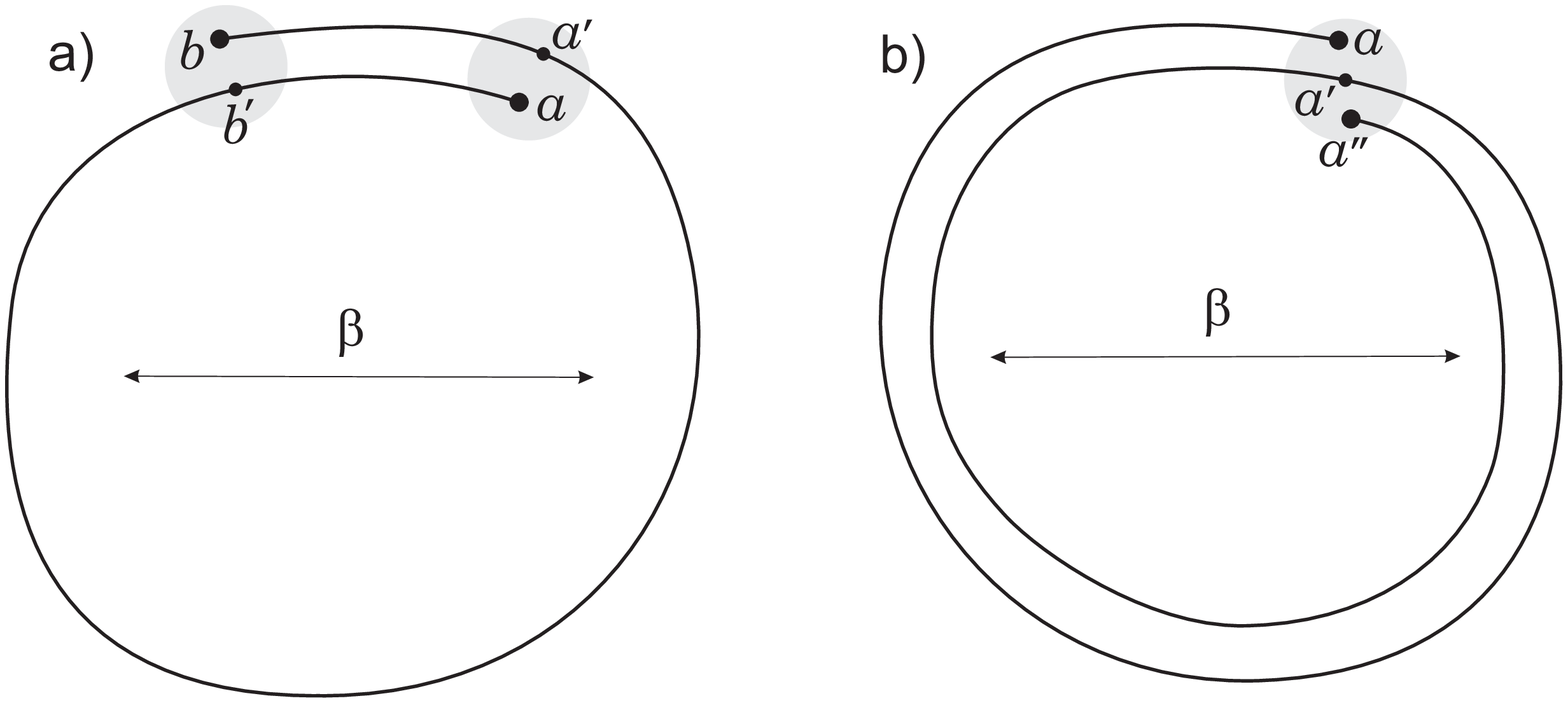}
	\caption{Links of snakes with segments of multiplicity two. a) has two nodes; b) a spiral snake and its single node. Shaded disks represent arcs with tangency order higher than $\beta$.}\label{fig:segment multiplicity}
\end{figure}

\begin{Exam}\label{Exam: multiplicity of segments in a complex curve}
	\normalfont Fig.~\ref{fig:segment multiplicity}a represents the link of a $\beta$-snake $X$ with three segments, $S=G(T(a',b'))$, $S'=G(T(b',a))$ and $S''=G(T(b,a'))$, such that $m(S)=1$ and $m(S')=m(S'')=2$, and two nodes $\N=Nod_{a}\cup Nod_{a'}$ and $\N'=Nod_{b}\cup Nod_{b'}$. If $\beta=1$, $tord(\gamma,T(b,a'))=3/2$ for all arcs $\gamma \subset T(b,a')$ and  $tord(\gamma',T(b',a))=3/2$ for all arcs $\gamma' \subset T(b',a)$ then the link of $X$ is outer metric equivalent to the link of the snake $X'$ in Example \ref{Exam: curve in C2} with $p=2$ and $q=3$.
\end{Exam}

\begin{Exam}\label{Exam: two segments one cluster}
\normalfont Fig.~\ref{fig:segment multiplicity}b represents the link of a $\beta$-snake $X'$ containing two segments, $S=G(T(a,a'))$ and $S'=G(T(a',a''))$ such that $m(S)=m(S')=2$, and a single node $\N=Nod_{a}\cup Nod_{a'}\cup Nod_{a''}$.
All three segments of $X'$ belong to a single cluster in $\mathcal{S}(\N,\N)$.
\end{Exam}

\subsection{Segments and nodal zones with respect to a pancake}\label{Subsection: seg and nodal zones with respect}

\begin{Def}\label{Def: distance functions to pancakes}
	\normalfont Let $X$ be a $\beta$-snake, and $\{X_{k}\}_{k=1}^{p}$ a pancake decomposition of $X$. If $\mu(X_j)=\beta$ we define the functions $f_1,\ldots, f_p$, where $f_{l}\colon (X_{j},0)\rightarrow (\R,0)$ is given by $f_{l}(x)=d(x,X_{l})$. For each $l$ we define $m_{l}\colon V(X_{j})\rightarrow \{0,1\}$ as follows: $m_{l}(\gamma)=1$ if and only if $ord_{\gamma}f_{l}>\beta$ and $m_{l}(\gamma)=0$ otherwise. In particular, $m_{j}\equiv 1$.
\end{Def}

\begin{remark}\label{multiplicity formula}
	\normalfont Consider $m_1,\ldots,m_p$ as in Definition \ref{Def: distance functions to pancakes}. For each $\gamma\in G(X_{j})$ we have $m(\gamma)=\sum_{l=1}^{p}m_{l}(\gamma)$.
\end{remark}

\begin{Def}\label{DEF: constant zones with resp to a pancake}
	\normalfont Consider $m_1,\ldots,m_p$ as in Definition \ref{Def: distance functions to pancakes}. A zone $Z \subset V(X_{j})$ is \textit{constant with respect to} $X_{l}$ if $m_{l}|_{Z}$ is constant.
\end{Def}

\begin{Def}\label{DEF: segments and nodal zones with resp to a pancake}
	\normalfont Let $m_1,\ldots,m_p$ be as in Definition \ref{Def: distance functions to pancakes}. Consider an arc $\gamma\in G(X_{j})$. For each $l$ we say that $\gamma$ is a \textit{segment arc with respect to} $X_{l}$ if there exists a $\beta$-H\"older triangle $T$ such that $\gamma$ is a generic arc of $T$ and $V(T)$ is constant with respect to $X_{l}$. Otherwise $\gamma$ is a \textit{nodal arc with respect to} $X_{l}$. The set of segment arcs in $G(X_j)$ with respect to $X_l$ and the set of nodal arcs in $G(X_{j})$ with respect to $X_l$ are denoted by $\mathbf{S}_l(X_{j})$ and $\mathbf{N}_l(X_{j})$, respectively. Furthermore, a \textit{segment with respect to} $X_{l}$ is a zone $S_{l,j}$ maximal in $\mathbf{S}_l(X_{j})$, and a \textit{nodal zone with respect to} $X_{l}$ is a zone $N_{l,j}$ maximal in $\mathbf{N}_l(X_{j})$. We write $Seg_{\gamma}^{l,j}$ or $Nod_{\gamma}^{l,j}$ for a segment or a nodal zone with respect to $X_{l}$ in $G(X_j)$ containing an arc $\gamma$.
\end{Def}

\begin{remark}\label{Rem: relative segments and nodal zones and perfect and open complete zones}
	\normalfont Let $f_1,\ldots,f_p$ be as in Definition \ref{Def: distance functions to pancakes}. Propositions \ref{Prop:segment is a perfect zone}, \ref{Prop:segment has finitely many zones} and \ref{Prop:nodal zones are finite} remain valid for segments and nodal zones in $G(X_j)$ with respect to $X_{l}$.
	
	In particular, taking $f=f_{l}$ and $T=X_{j}$, segments in $G(X_j)$ with respect to $X_{l}$ are in one-to-one correspondence with the closed perfect zones maximal in $B_{\beta}(f_l)$ and $H_{\beta}(f_l)$. Similarly, the nodal zones with respect to $X_{l}$ are in one-to-one correspondence with the open $\beta$-complete zones in $H_{\beta}(f_l)$ (see Propositions \ref{Prop:arcs in B_{beta} are generic} and \ref{Prop:maximal zones in H_{beta}}).
\end{remark}

\begin{Lem}\label{Lem:relative nodal zones and relative segments}
	Let $X$ be a $\beta$-snake, and let $f_1,\ldots,f_p$ be as in Definition \ref{Def: distance functions to pancakes}.\newline
\emph{(1)} Let $l,l' \in \{1,\ldots, p\}\setminus\{j\}$ with $l\ne l'$. Then, any two nodal zones $N_{l,j}$ and $N_{l',j}$ either coincide or are disjoint. If $N_{l,j}\cap N_{l',j}=\emptyset$ then $itord(N_{l,j},N_{l',j})=\beta$.\newline
\emph{(2)} If $\gamma \in G(X_j)$ is a segment arc of $X$ then $\gamma$ is a segment arc with respect to $X_{l}$.
\end{Lem}
\begin{proof}
	$(1)$ This is an immediate consequence of Remark \ref{Rem: relative segments and nodal zones and perfect and open complete zones} and Remark \ref{Rem:open complete zones}.
	
	$(2)$ Suppose that $\gamma \in G(X_j)$ belong to a segment $S$ of $X$, and there exists $l\ne j$ such that $\gamma$ is a nodal arc with respect to $X_{l}$. Remark \ref{Rem: relative segments and nodal zones and perfect and open complete zones} implies that $$Nod_{\gamma}^{l,j}=\{\gamma' \in V(X_{j}) : tord(\gamma,\gamma')>\beta\} \subset H_{\beta}(f_l).$$
	
	There is a closed perfect $\beta$-zone $B_{l}\subset B_{\beta}(f_{l})$ adjacent to $Nod_{\gamma}^{l,j}$. Let $\lambda_{1} \in S\cap B_{l}$. Then, $m_{l}(\gamma)=1$ and $m_{l}(\lambda_{1})=0$. As $\gamma,\lambda_{1}\in S$, it follows that $m(\gamma)=m(\lambda_{1})$. Thus, Remark \ref{multiplicity formula} implies that there is $l_1\in \{1,\ldots, p\}\setminus\{l,j\}$ such that $m_{l_{1}}(\lambda_{1})=1$ and $m_{l_1}(\gamma)=0$.
	
	As $\gamma \in B_\beta(f_{l_1})$, by Proposition \ref{Prop:arcs in B_{beta} are generic}, there is a closed perfect $\beta$-zone $B_{l_1}$ maximal in $B_\beta(f_{l_1})$ containing $\gamma$. Thus, there is $\lambda_{2} \in (B_{l}\cap B_{l_1})\cap V(T(\lambda_{1},\gamma))$. In particular $\lambda_2 \in S$.
	
	As $\gamma,\lambda_{2} \in S$ it follows that $m(\gamma)=m(\lambda_{2})$. Then, as $m_{l}(\lambda_{2})=m_{l_1}(\lambda_{2})=0$, by Remark \ref{multiplicity formula}, there is $l_2\in \{1,\ldots, p\}\setminus\{l,l_1,j\}$ such that $m_{l_{2}}(\lambda_{2})=1$ and $m_{l_{2}}(\gamma)=m_{l_{2}}(\lambda_{1})=0$.
	
	Similarly, as $\gamma \in B_\beta(f_{l_2})$, there are a closed perfect $\beta$-zone $B_{l_2}$ maximal in $B_\beta(f_{l_2})$ containing $\gamma$ and an arc $\lambda_{3} \in (B_{l}\cap B_{l_1}\cap B_{l_2})\cap V(T(\lambda_{1},\gamma))$. In particular $\lambda_3 \in S$.
	
	Continuing with this process, after at most $p-1$ steps, we get a contradiction.
\end{proof}

\begin{Cor}\label{Cor:relative multiplicity pancake and segment}
	Let $X$ be a $\beta$-snake, $\{X_{k}\}_{k=1}^{p}$ a pancake decomposition of $X$, and $S\subset V(X)$ a segment.
If $\gamma,\lambda \in S\cap G(X_{j})$ then $m_{l}(\gamma)=m_{l}(\lambda)$ for all $l$.
\end{Cor}
\begin{proof}\
	Given arcs $\gamma,\lambda \in S\cap G(X_{j})$, by Lemma \ref{Lem:relative nodal zones and relative segments}, $\gamma \in Seg^{l,j}_\lambda$ for all $l$. As a segment in $G(X_j)$ with respect to $X_l$ is a constant zone, it follows that $m_{l}(\gamma)=m_{l}(\lambda)$.
\end{proof}

\begin{Prop}\label{Prop: each node contains at least two nodal zones}
	Let $X$ be a $\beta$-snake. Then
	\begin{enumerate}
		\item If $tord(\gamma,\gamma')>itord(\gamma,\gamma')$ for some $\gamma,\gamma' \in V(X)$, and $\gamma$ is a nodal arc, then $\gamma'$ is also a nodal arc.
		\item Each node of $X$ has at least two nodal zones.
	\end{enumerate}
\end{Prop}
\begin{proof}
	$(1)$ Consider $\gamma,\gamma' \in V(X)$ such that $tord(\gamma,\gamma')>itord(\gamma,\gamma')$ and $\gamma$ is a nodal arc. Clearly, $\gamma'\notin Nod_{\gamma}$, since $itord(\gamma,\gamma')=\beta$. Suppose, by contradiction, that $\gamma'$ is a segment arc, say $\gamma'\in S$ where $S=Seg_{\gamma'}$. Let $\{X_k\}$ be a minimal pancake decomposition of $X$. Assume that $\gamma \in X_j$ and $\gamma'\in  X_{j'}$. Since each pancake is normally embedded, $j\ne j'$. As $\gamma'$ is a segment arc we can assume that $\gamma'\in G(X_{j'})$. Consider arcs $\theta'_1 \in S\cap G(T(\lambda_{j'-1},\gamma'))$ and $\theta'_2 \in S\cap G(T(\gamma',\lambda_{j'}))$. Since $\theta'_1,\theta'_2,\gamma'\in S$ and $m_j(\gamma')=1$, by Corollary \ref{Cor:relative multiplicity pancake and segment}, $m_j(\theta'_1)=m_j(\theta'_2)=1$. Thus, there exist arcs $\theta_1,\theta_2 \in V(X_j)$ such that $tord(\theta_i,\theta'_i)>\beta$, for $i=1,2$, $T=T(\theta_1,\theta_2)$ is a $\beta$-H\"older triangle and $\gamma$ is a generic arc of $T$. Then, $V(T)\subset H_\beta(f_{j'})$, what implies that $\gamma$ is segment arc with respect to $X_{j'}$, a contradiction with Remark \ref{Rem: relative segments and nodal zones and perfect and open complete zones}, since if a nodal arc belongs to a zone contained in $H_\beta(f_{j'})$, this zone should be an open $\beta$-complete zone.
	
	$(2)$ Let $\N$ be a node of $X$, and $N$ a nodal zone of $\N$. By Remark \ref{Rem: relative segments and nodal zones and perfect and open complete zones}, given $\gamma \in N$ there exists $\gamma'\in V(X)\setminus N$ such that $tord(\gamma,\gamma')>itord(\gamma,\gamma')$, since $N\subset H_\beta(f_l)$ for some $l$. By item $(1)$ of this Proposition, $\gamma'$ is a nodal arc and $Nod_{\gamma'}\ne N$ is a nodal zone of $\N$.
\end{proof}

\subsection{Bubbles, bubble snakes and spiral snakes}\label{Subsec: bubbles, bubble snakes and spiral snakes}

\begin{Def}\label{Def:bubble}
	\normalfont A $\beta$-\textit{bubble} is a non-singular $\beta$-H\"older triangle $X=T(\gamma_{1},\gamma_{2})$ such that there exists an interior arc $\theta$ of $X$ with both $X_{1}=T(\gamma_{1},\theta)$ and $X_{2}=T(\theta,\gamma_{2})$ normally embedded and $tord(\gamma_{1},\gamma_{2})>itord(\gamma_{1},\gamma_{2})$. If $X$ is a snake then it is called a $\beta$-\textit{bubble snake}.
\end{Def}

\begin{remark}
	\normalfont It follows from Lemma \ref{Lem:beta-bubble two NE pieces} that if $X$ is a $\beta$-bubble then $X_{1}$ and $X_{2}$ are $\beta$-H\"older triangles.
\end{remark}

\begin{Def}\label{DEF: spiral snake}
	\normalfont A \textit{spiral} $\beta$-snake $X$ is a $\beta$-snake with a single node and two or more segments (see Fig.~\ref{fig:segment multiplicity}b).
\end{Def}

\begin{Exam}\label{Exam: spiral snake}
	\normalfont Instead of removing the H\"older triangle $T$ from a complex curve as in Example \ref{Exam: curve in C2}, remove an $\alpha$-H\"older triangle $T'$ with $\alpha>1$ contained in $X$. Then $X''=\overline{X\setminus T'}$ is a spiral snake with $p$ segments.
\end{Exam}

\begin{remark}\label{REM: number of segments of spiral and bubble snakes}
	\normalfont Any snake with a single node and $p$ segments is either a bubble snake if $p=1$ or a spiral snake if $p>1$.
\end{remark}

\begin{Prop}\label{Prop: spiral snake and its segments}
	Let $X$ be a spiral $\beta$-snake. Then, for each segment arc $\gamma$ in $X$ and for each segment $S\ne Seg_\gamma$ of $X$, $tord(\gamma,S)>\beta$.
\end{Prop}
\begin{proof}
	
	\begin{figure}
		\centering
		\includegraphics[width=3in]{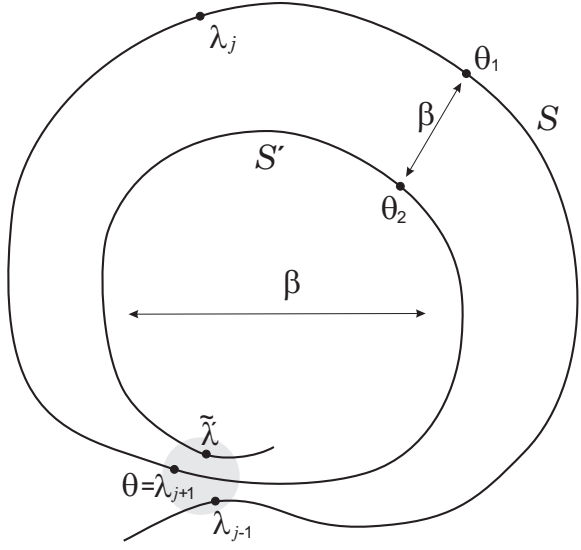}
		\caption{Contradictory case in the proof of Proposition \ref{Prop: spiral snake and its segments}. The shaded disk represents the single node of $X$.}\label{fig:spiral and its segments}
	\end{figure}

	First, we are going to prove that if $X$ is a spiral $\beta$-snake and $S$, $S'$ are consecutive segments of $X$, then $tord(\gamma,S')>\beta$ for each $\gamma \in S$. Let $N$ be the nodal zone adjacent to both $S$ and $S'$, and $\tilde{N}$, $\tilde{N}'$ the other nodal zones adjacent to $S$ and $S'$, respectively. Consider arcs $\lambda \in N$, $\tilde{\lambda} \in \tilde{N}$,  $\tilde{\lambda}' \in \tilde{N}'$, and the $\beta$-H\"older triangles $T=T(\tilde{\lambda},\lambda)$ and $T'=T(\lambda,\tilde{\lambda}')$. Proposition \ref{Prop:segment is the generic arcs from HD of nodal adjacent zones} implies that $S=G(T)$ and $S'=G(T')$.
	Consider the germ of the function $f\colon (T,0)\rightarrow (\R,0)$ given by $f(x)=d(x,T')$. Let $\{T_i=T(\lambda_{i-1},\lambda_i)\}_{i=1}^{p}$ be a minimal pizza on $T$ associated with $f$. It is enough to show that $ord_{\lambda_i}f>\beta$ for each $i=0,\ldots, p$.
	
	Suppose, by contradiction, that there is $j\in \{0,\ldots,p\}$ such that $ord_{\lambda_j}f=\beta$. Since a spiral snake has a single node $\N$, both $\lambda_0$ and $\lambda_p$ belong to $\N$. Thus, $ord_{\lambda_0}f>\beta$, $ord_{\lambda_p}f>\beta$ and $0<j<p$. Then $p>1$ and Lemma \ref{Lem:pizza HT and arcs in B_{beta}} implies that $ord_{\lambda_{j-1}}f>\beta$ and $ord_{\lambda_{j+1}}f>\beta$. We claim that both $\lambda_{j-1}$ and $\lambda_{j+1}$ do not belong to $S$ and consequently, since $X$ is a spiral snake, are nodal arcs (see Fig.~\ref{fig:spiral and its segments}). Assume that $\lambda_{j+1} \in S$ (if $\lambda_{j-1}\in S$ we obtain a similar contradiction). Let $\{X_k\}$ be a pancake decomposition of $X$ such that $\lambda_{j+1}\in G(X_k)$, $X_k\subset T$ and $\mu(X_k)=\beta$. As $ord_{\lambda_{j+1}}f>\beta$, $T$ is not normally embedded and $\{X_k\}$ is a pancake decomposition, there exists a pancake $X_l$, $X_l\ne X_k$, such that $X_l\cap T'\ne \emptyset$ and $tord(\lambda_{i+1},X_l\cap T')>\beta$. Since $\lambda_{j+1} \in S$, Lemma \ref{Lem:relative nodal zones and relative segments} implies that $\lambda_{j+1}$ is a segment arc in $G(X_k)$ with respect to $X_l$. Thus, since $ord_{\lambda_{j+1}}f_l= ord_{\lambda_{j+1}}f>\beta$, Remark \ref{Rem: relative segments and nodal zones and perfect and open complete zones} implies that $\lambda_{j+1}$ is contained in a closed perfect $\beta$-zone $H$ maximal in $H_\beta(f_l)$. Hence, $H\cap G(T_{j+1})\ne \emptyset$, a contradiction with $G(T_{j+1})\subset B_\beta(f_l)\subset B_\beta(f)$, by Proposition \ref{Prop:width function properties elementary triangle}.
	
	Then, for every $\lambda' \in S$, $ord_{\lambda'}f=\beta$ and consequently, $tord(S,S')=\beta$, a contradiction with the arc $\theta=\lambda_{j+1}$, in an interior nodal zone, being abnormal. To show this, suppose that $\theta$ is normal and consider arcs $\theta_1\in V(T(\gamma_1,\theta))$ and $\theta_2 \in V(T(\theta,\gamma_2))$, where $X=T(\gamma_{1},\gamma_{2})$, such that $T(\theta_1,\theta)$ and $T(\theta,\theta_2)$ are normally embedded $\beta$-H\"older triangles which intersection is $\theta$ and $tord(\theta_1,\theta_2)>itord(\theta_1,\theta_2)$. Since $T(\theta_1,\theta)$ and $T(\theta,\theta_2)$ are normally embedded and $X$ has a single node, since it is a spiral snake, both $\theta_1$ and $\theta_2$ are in $S\cup S'$, say $\theta_1\in S$ and $\theta_2\in S'$. However, $tord(S,S')=\beta$ thus $tord(\theta_1,\theta_2)=\beta$, a contradiction with $tord(\theta_1,\theta_2)>\beta =itord(\theta_1,\theta_2)$.
	
	Finally, given two non-necessarily consecutive segments $S$ and $S'=Seg_\gamma$ as in the Proposition \ref{Prop: spiral snake and its segments}, the result follows from the non-archimedean property.
\end{proof}

\begin{figure}
	\centering
	\includegraphics[width=4.5in]{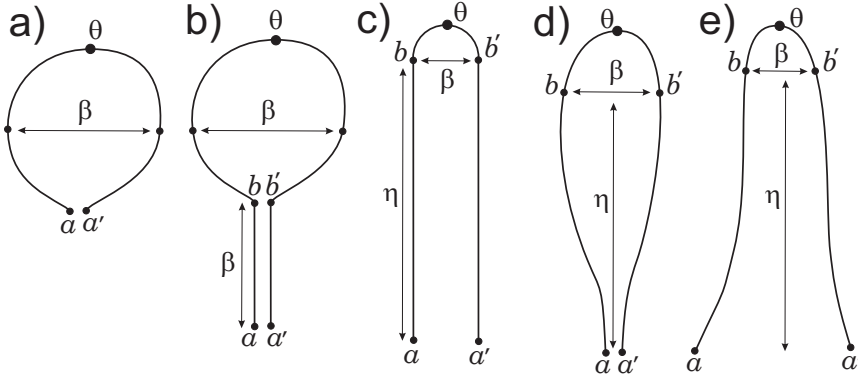}
	\caption{Links of a bubble snake and non-snake bubbles.}\label{fig:bubble snake and non-snake bubbles}
\end{figure}

\begin{Exam}\label{Exam: cases of the main theorem}
	\normalfont
	
	Fig.~\ref{fig:bubble snake and non-snake bubbles}a shows the link of a $\beta$-bubble snake $X_a$, with $tord(a,a')>\beta$.
	
	Fig.~\ref{fig:bubble snake and non-snake bubbles}b shows the link of a $\beta$-bubble $X_b$ with a ``neck'' consisting
	of two normally embedded $\beta$-H\"older triangles $T=T(a,b)$ and $T'=T(a',b')$ such that
	$tord(\gamma,T')>\beta$ for all arcs $\gamma\in V(T)$ and $tord(\gamma,T)>\beta$ for all arcs $\gamma\in V(T')$.
	Since all arcs in $T$ and $T'$ are normal, $X_b$ is not a snake, although it does contain a $\beta$-bubble snake $X_a$.
	
	Figs.~\ref{fig:bubble snake and non-snake bubbles}c, \ref{fig:bubble snake and non-snake bubbles}d and \ref{fig:bubble snake and non-snake bubbles}e show the links of non-snake $\eta$-bubbles $X_c$, $X_d$ and $X_e$, respectively,
	with $tord(a,a')>\eta$.  The set of abnormal arcs in each of them is a closed perfect $\beta$-complete abnormal zone $Z$.
	In each of these three figures $T(b,b')$ is a normally embedded $\beta$-H\"older triangle, while $T=T(a,b)$ and $T'=T(a',b')$ are normally embedded $\eta$-H\"older triangles where $\eta<\beta$.
	For each arc $\gamma\in V(T)\setminus Z$ we have $tord(\gamma,T')=\beta$ in $X_c$, $tord(\gamma,T')>\beta$ in $X_d$,  and $tord(\gamma,T')<\beta$ in $X_e$. 	
\end{Exam}

\begin{Exam}\label{Exam: snake inside bubble}
	\normalfont Fig.~\ref{fig:nonsnakebubble1} shows the link of a non-snake $\beta$-bubble containing a non-bubble $\beta$-snake with the same link as in Fig.~\ref{fig:three snakes}b.
	\begin{figure}
		\centering
		\includegraphics[width=3in]{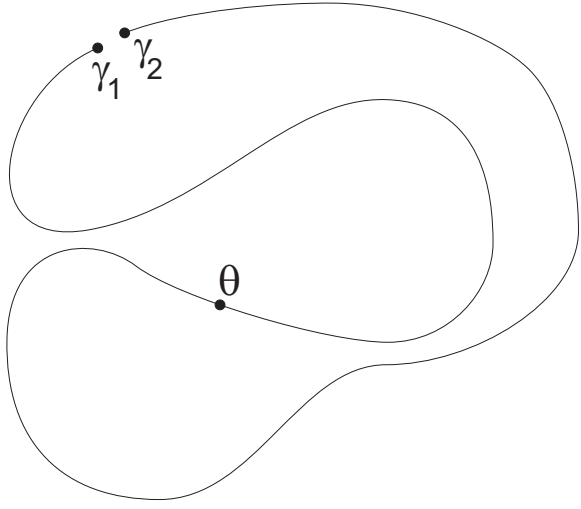}
		\caption{Link of a non-snake bubble containing a non-bubble snake.}\label{fig:nonsnakebubble1}
	\end{figure}
\end{Exam}

\begin{Prop}\label{Prop:bubble snake and its generic arcs}
	Let $X$ be a $\beta$-bubble snake as in Definition \ref{Def:bubble}. If $\gamma\in G(X_1)$ then $tord(\gamma,X_2)=\beta$.
\end{Prop}
\begin{proof}
	Suppose, by contradiction, that $tord(\gamma,X_2)>\beta$. Then there is an arc $\gamma'\in G(X_2)$ such that $tord(\gamma,\gamma')>\beta$.
	Choose $b>0$ (by Corollary \ref{Cor:intersection of horn neighbourhood with Holder triangle} such a real number exists) so that $T=X_1\cap H_{b,\beta}(\gamma_1)$ and $T'=X_2\cap H_{b,\beta}(\gamma_1)$ are H\"older triangles (in particular, $T$ and $T$' are connected), $\gamma\not\subset T$, $\gamma'\not\subset T'$. Next, choose $\lambda\in G(T)$ so that $T(\gamma_1,\lambda)\subset H_{b/2,\beta}(\gamma_1)$ (see Fig.~\ref{fig:bubble snake}). Then any arc $\lambda'\subset X_2$ such that $tord(\lambda', T(\gamma_1,\lambda))>\beta$ must belong to $T'$.
	Note that $T$ has exponent $\beta$. Thus, since $X$ is a snake and $\lambda\in G(T)\subset G(X)$, the arc $\lambda$ is abnormal: there exist normally embedded triangles $\tilde T\subset T(\gamma_1,\lambda)$ and $\tilde T'\subset T(\lambda,\gamma_2)$ such that $\tilde T\cup\tilde T'$ is not normally embedded. Since $X_1$ is normally embedded, $\theta\subset \tilde T'$, thus $\gamma\subset \tilde T'$ and both $\tilde T$ and $\tilde T'$ are $\beta$-H\"older triangles. Since $\tilde T\cup\tilde T'$ is not normally embedded, there exists an arc $\lambda'\subset\tilde T'\cap X_2$ such that $tord(\lambda', T(\gamma_1,\lambda))>\beta$. Then $\lambda'\subset T'$, which implies that $\gamma'\subset\tilde T'$, in contradiction to $\tilde T'$ being normally embedded, as $itord(\gamma,\gamma')=\beta$ and $tord(\gamma,\gamma')>\beta$.
	
	\begin{figure}
		\centering
		\includegraphics[width=3in]{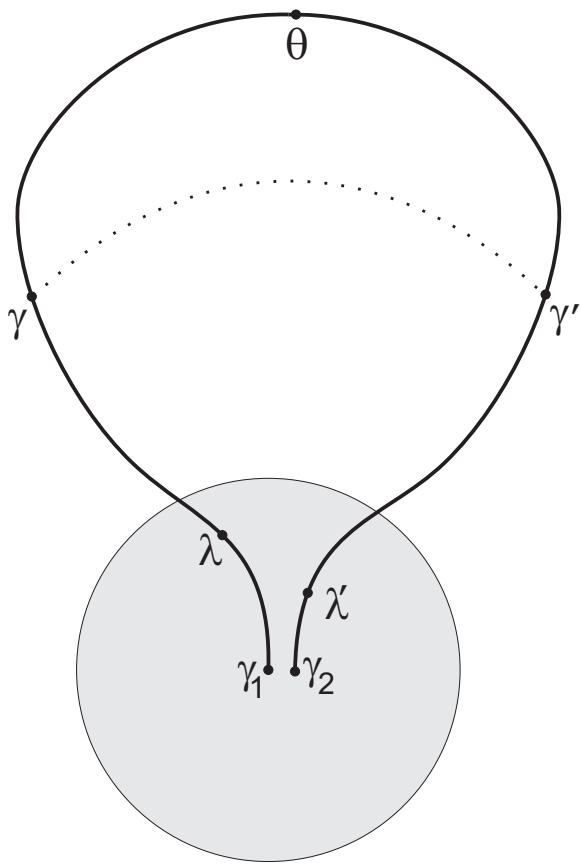}
		\caption{The shaded disk represents a $\beta$-horn neighborhood of $\gamma_{1}$ in the proof of Proposition \ref{Prop:bubble snake and its generic arcs}.}\label{fig:bubble snake}
	\end{figure}
\end{proof}

\begin{Prop}\label{Prop:bubble, segments, generic arcs and pancake decomp}
	If $X$ is a $\beta$-bubble snake as in Definition \ref{Def:bubble} then
	\begin{enumerate}
		\item $V(X)$ consists of a single segment $S$ of multiplicity $1$ and a single node $\N$ with two boundary nodal zones.
		\item For any generic arc $\gamma$ of $X$, both $T(\gamma_{1},\gamma)$ and $T(\gamma,\gamma_{2})$ are normally embedded.
		\item Any minimal pancake decomposition of $X$ has exactly two pancakes.
	\end{enumerate}
\end{Prop}
\begin{proof}
	$(1)$ Let $f:(X_1,0)\rightarrow (\R,0)$ be the germ of the Lipschitz function given by $f(x)=d(x,X_2)$. Note that if $\gamma \in G(X_{1})$ then $ord_{\gamma}f>\beta$ if and only if $tord(\gamma,X_2)>\beta$. Thus, by Proposition \ref{Prop:bubble snake and its generic arcs}, the result follows.
	
	$(2)$ Let $\gamma$ be a generic arc of $X$. Let $\widetilde{T_{1}}=T(\gamma_{1},\gamma)$ and $\widetilde{T_{2}}=T(\gamma,\gamma_{2})$. From the definition of a bubble, there exists a generic arc $\theta$ of $X$ such that $X_{1}=T(\gamma_{1},\theta)$ and $X_{2}=T(\theta,\gamma_{2})$ are normally embedded.
	
	Suppose that $\widetilde{T_{2}} \subset X_{2}$. We are going to prove that $\widetilde{T_{1}}$ is normally embedded. The case when $\widetilde{T_{1}} \subset X_{1}$ and $\widetilde{T_{2}}$ is not normally embedded is similar.
	
	If $\widetilde{T_{1}}$ is not normally embedded then there exist arcs $\lambda \in V(X_{1})$ and $\lambda' \in V(\widetilde{T_{1}})\setminus V(X_1)$ such that $tord(\lambda,\lambda')>itord(\lambda,\lambda')$. Note that $\lambda'$ is generic and consequently abnormal. This implies that $\lambda$ is also abnormal.
	
	Thus, $\lambda$ and $\lambda'$ must be generic arcs of $X$. but this implies, by $(1)$, that $m(\lambda)=m(\lambda')=1$, a contradiction with $tord(\lambda,\lambda')>itord(\lambda,\lambda')$.
	
	$(3)$ This is an immediate consequence of item $(2)$ of this Proposition.
\end{proof}

\subsection{Pancake decomposition defined by segments and nodal zones}\label{Subsec: Pancake decomp defined by segments and nodal zones}

\begin{Prop}\label{Prop:bubble inside snake}
	Let $X=T(\gamma_{1},\gamma_{2})$ be a snake, $S$ a segment of $X$ and $N\ne N'$ two nodal zones of $X$ adjacent to $S$. Then
	\begin{enumerate}
		\item If $\gamma$ and $\gamma'$ are two arcs in $N$ then $T(\gamma,\gamma')$ is normally embedded.
		\item If $\gamma$ and $\gamma'$ are two arcs in $S$ then $T(\gamma,\gamma')$ is normally embedded.
		\item  If $\gamma\in S$ and $\gamma'\in N$ then $T(\gamma,\gamma')$ is normally embedded.
		\item If $\gamma\in N$ and $\gamma'\in N'$ then $T(\gamma,\gamma')$ is normally embedded, unless $X$ is either a bubble snake or a spiral snake.
	\end{enumerate}
	In particular, nodal zones and segments are normally embedded zones (see Definition \ref{Def: open closed perfect zone}).
\end{Prop}
\begin{proof}
	$(1)$ Let $\gamma$ and $\gamma'$ be two arcs in $N$. Note that, by Proposition \ref{Prop:nodal zones are finite}, $T(\gamma,\gamma')$ has exponent greater than $\beta$. Then, there are no arcs $\lambda$ and $\lambda'$ in $N$ such that $tord(\lambda,\lambda')>itord(\lambda,\lambda')$, otherwise, by Proposition \ref{Prop:weak LNE}, $itord(\lambda,\lambda')=\beta$, a contradiction with exponent of $T(\gamma,\gamma')$ greater than $\beta$.
	
	To prove the next two items it is enough to show that there are no arcs $\gamma \in S$ and $\gamma' \in S\cup N$ such that $tord(\gamma,\gamma')>itord(\gamma,\gamma')$. Let us assume that $\gamma' \subset T(\gamma_{1},\gamma)$.
	
	$(2)$ and $(3)$ Suppose, by contradiction, that there exist such arcs $\gamma$ and $\gamma'$. As $\gamma \in S$ it is abnormal and then there are arcs $\lambda \subset T(\gamma,\gamma')$ and $\lambda' \subset T(\gamma,\gamma_2)$ such that $T=T(\lambda,\gamma)$ and $T'=T(\gamma,\lambda')$ are normally embedded $\beta$-H\"older triangles such that $T\cap T'=\gamma$ and $tord(\lambda,\lambda')>itord(\lambda,\lambda')$ (see Fig.~\ref{fig:bubble inside snake}). Let $\{X_k\}$ be a minimal pancake decomposition of $X$. We can assume that $\lambda \in V(X_j)$ and $\lambda'\in V(X_{j+1})$, since none of these arcs is in a nodal boundary zone and consequently, if necessary, we could enlarge the pancake attaching a $\beta$-H\"older triangle to one of its boundaries.
	
	Since $T$ is normally embedded, $\lambda \in S$. So, we can assume that $\lambda \in G(X_j)$. We can further assume that $\gamma \in G(X_j)$, since $\gamma \in S$. Thus, since, by Corollary \ref{Cor:relative multiplicity pancake and segment}, $m_{j+1}(\gamma)=m_{j+1}(\lambda)=1$, there exists $\gamma''\in V(\lambda_j,\lambda')$ such that $tord(\gamma,\gamma'')>\beta=itord(\gamma,\gamma'')$, a contradiction with $T'$ being normally embedded.
	
	$(4)$ Note that as a spiral snake has a single node, the result is trivially false in this case. Thus, assume that $X$ is not a $\beta$-spiral snake. Suppose, by contradiction, that there exist arcs $\gamma \in N$ and $\gamma' \in N'$ such that $tord(\gamma,\gamma')>itord(\gamma,\gamma')$. If $X$ is not a bubble snake then we can assume that one of the arcs $\gamma$, $\gamma'$, say $\gamma$, is abnormal. As $\gamma$ is abnormal there are arcs $\lambda \subset T(\gamma,\gamma')$ and $\lambda' \subset T(\gamma,\gamma_2)$ such that $T=T(\lambda,\gamma)$ and $T'=T(\gamma,\lambda')$ are normally embedded $\beta$-H\"older triangles such that $T\cap T'=\gamma$ and $tord(\lambda,\lambda')>itord(\lambda,\lambda')$. Let $\{X_k\}$ be a minimal pancake decomposition of $X$. We can assume that $\lambda \in V(X_j)$ and $\lambda'\in V(X_{j+1})$.
	
	Since $T$ is normally embedded, we have $\lambda \in S$. Thus, we can assume that $\lambda \in G(X_j)$. As $m_{j+1}(\lambda)=1$, Lemma \ref{Lem:relative nodal zones and relative segments} and Proposition \ref{Prop:maximal zones in H_{beta}} imply that $N\subset H_{\beta}(f_{j+1})$ and $\lambda$ belong to a closed perfect $\beta$-zone maximal in $H_{\beta}(f_{j+1})$ (the segment with respect to $X_{j+1}$, $S_{\lambda}^{j+1}$) adjacent to $N$. Then, there exists $\gamma''\in V(\lambda_j,\lambda')$ such that $tord(\gamma,\gamma'')>\beta=itord(\gamma,\gamma'')$, a contradiction with $T'$ normally embedded.

\begin{figure}
	\centering
	\includegraphics[width=4in]{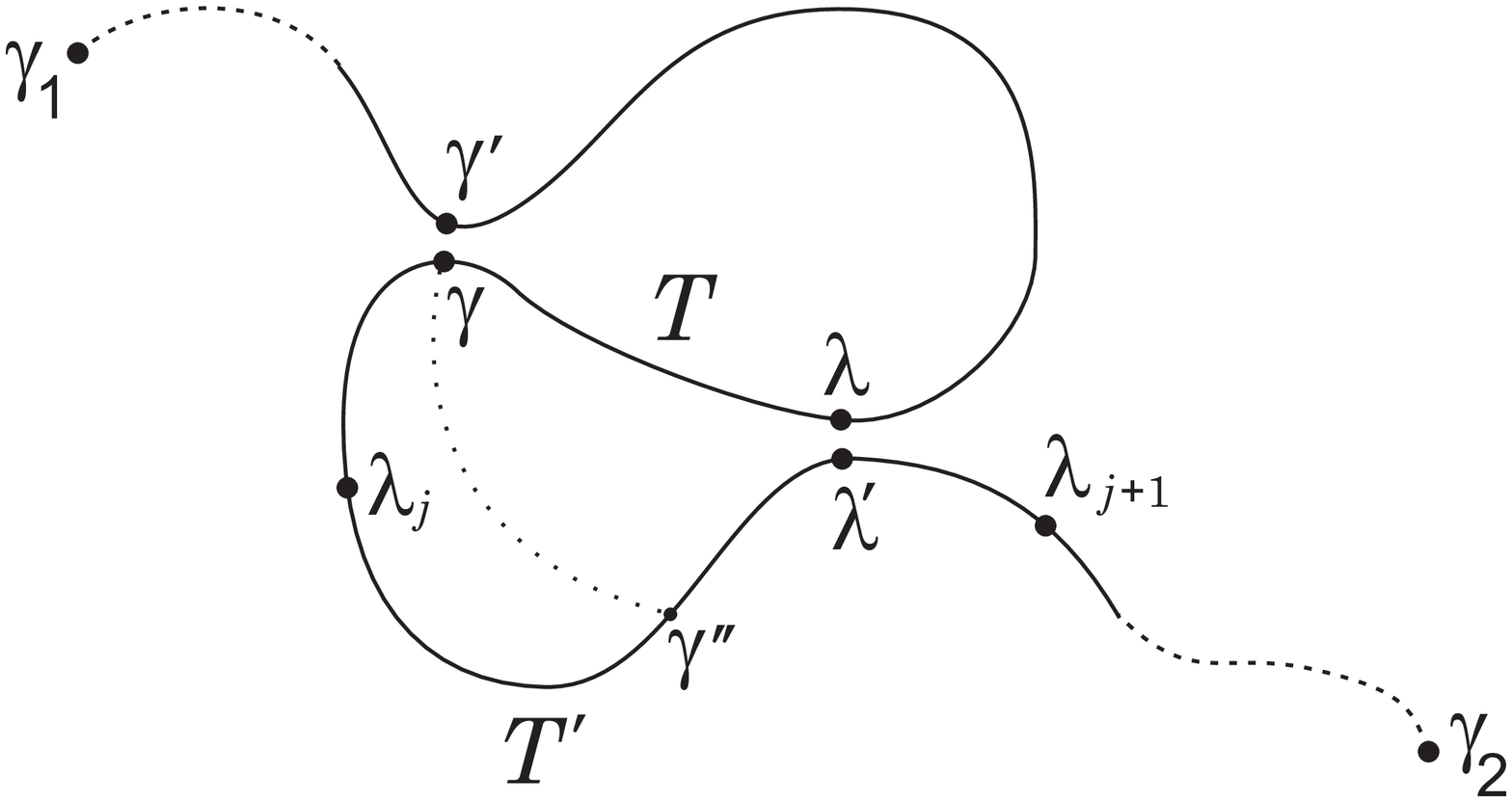}
	\caption{Position of arcs in the proof of Proposition \ref{Prop:bubble inside snake}}\label{fig:bubble inside snake}
	\end{figure}
\end{proof}

\begin{Prop}\label{Prop:pancake decom from nodal zones}
	The following decomposition of a snake $X$ other than the bubble and the spiral into H\"older triangles determines a pancake decomposition of $X$: the boundary arcs of the H\"older triangles in the decomposition are the two boundary arcs of $X$ together with one arc in each nodal zone. The segments of $X$ are in one-to-one correspondence with the sets of generic arcs of its pancakes.
\end{Prop}
\begin{proof}
	This is an immediate consequence of Proposition \ref{Prop:bubble inside snake}.
\end{proof}

\begin{remark}
	\normalfont In general, the pancake decomposition described in Proposition \ref{Prop:pancake decom from nodal zones} is not minimal. Moreover, if $\{X_k=T(\gamma_{k-1},\gamma_k)\}_{k=1}^p$ is the pancake decomposition defined in Proposition \ref{Prop:pancake decom from nodal zones} then $X$ has exactly $p$ segments $S_i = G(X_i)$, $i=1,\ldots, p$ and $p+1$ nodal zones $Nod_{\gamma_j}$, $j=0,\ldots, p$.
\end{remark}

\begin{Prop}\label{Prop: HT in consecutive segments}
	Let $X$ be a snake other than the spiral, $N$ a nodal zone of $X$ and $S\ne S'$ two segments of $X$ adjacent to $N$. If $\gamma\in S$ and $\gamma'\in S'$ then $T(\gamma,\gamma')$ is normally embedded.
\end{Prop}
\begin{proof}
	Let $\tilde{N}$ and $\tilde{N}'$ be the nodal zones, distinct from $N$, adjacent to $S$ and $S'$, respectively.
	Consider arcs $\tilde\gamma \in \tilde{N}$, $\tilde\gamma' \in \tilde{N}'$ and $\theta \in N$. As $X$ is not a spiral snake, by item $(4)$ of Proposition \ref{Prop:pancake decom from nodal zones}, we can assume that $T=T(\tilde\gamma,\theta)$ and $T'=T(\theta,\tilde\gamma')$ are pancakes of a minimal pancake decomposition. Proposition \ref{Prop:segment is the generic arcs from HD of nodal adjacent zones} implies that $S= G(T)$ and $S'=G(T')$.
	
	Then, if there were arcs $\gamma \in S$ and $\gamma' \in S'$ such that $tord(\gamma,\gamma')>itord(\gamma,\gamma')$, by Corollary \ref{Cor:relative multiplicity pancake and segment}, we would have that for each arc $\lambda\in S$ there should exist $\lambda\in S'$ such that $tord(\lambda,\lambda')>itord(\lambda,\lambda')$. This implies that $tord(\tilde{N},\tilde{N}')>\beta$, but the arcs in $S$ should be abnormal, a contradiction.
\end{proof}

\begin{Prop}\label{Prop: segments with contact implies adjacent nodal zones with contact}
	Let $S$ be a segment of a $\beta$-snake $X$ such that the nodal zones adjacent to $S$ belong to distinct nodes $\N$ and $\widetilde{\N}$. If $S'$ is another segment of $X$ such that $tord(S,S')>\beta$ then the nodal zones adjacent to $S'$ belong to the same nodes $\N$ and $\widetilde{\N}$.
\end{Prop}
\begin{proof}
	Let $N, \widetilde{N}$ and $N', \widetilde{N}'$ be the nodal zones adjacent to $S$ and $S'$, respectively. Assume that $N\subset \N$ and $\widetilde{N}\subset \widetilde{\N}$. Consider the arcs $\lambda\in N$, $\widetilde{\lambda} \in \widetilde{N}$, $\lambda'\in N'$, $\widetilde{\lambda}' \in \widetilde{N}'$ and the $\beta$-H\"older triangles $T=T(\lambda,\widetilde{\lambda})$ and $T'=T(\lambda',\widetilde{\lambda}')$. Proposition \ref{Prop:segment is the generic arcs from HD of nodal adjacent zones} implies that $S=G(T)$ and $S'=G(T')$. Moreover, Proposition \ref{Prop:pancake decom from nodal zones} implies that $T$ and $T'$ are pancakes from a pancake decomposition of $X$. Then, as $tord(S,S')>\beta$, Corollary \ref{Cor:relative multiplicity pancake and segment} implies that $tord(\gamma,S')>\beta$ for all arcs $\gamma \in S$.
	
	We now prove that either $N'\subset \N$ or $\widetilde{N}'\subset \N$. Suppose, by contradiction, that $tord(N,N')=tord(N,\widetilde{N}')=\beta$. Let $f\colon (T,0)\rightarrow (\R,0)$ be the function given by $f(x)=d(x,T')$ and let $\{T_i\}$ be a pizza on $T$ associated with $f$. As $tord(N,N')=tord(N,\widetilde{N}')=\beta$, Proposition \ref{Prop: each node contains at least two nodal zones} implies that $ord_{\lambda}f=\beta$. Then, Proposition \ref{Prop:width function properties elementary triangle} implies that there is an arc $\theta \in G(T)$ such that $ord_\theta f=\beta$, a contradiction with $tord(\theta, S')>\beta$, since $G(T)=S$.
	
	Finally, if, for example, $N'\subset \N$ then $\widetilde{N}'\subset \widetilde{\N}$. Indeed, $N'\subset \N$ implies $tord(\widetilde{N},N')=\beta$. If $tord(\widetilde{N},\widetilde{N}')=\beta$ then, similarly, $ord_{\widetilde{\lambda}}f=\beta$ and we obtain an arc $\theta \in G(T)$ such that $ord_\theta f=\beta$, a contradiction. Hence, $\widetilde{N}'\subset \widetilde{\N}$.
\end{proof}

\section{Main Theorem}\label{Section: Main Theorem}

In this section we investigate the role played by abnormal zones and snakes in the Lipschitz Geometry of surface germs. The main result of this section, Theorem \ref{Teo: Main HT decomposition}, was the original motivation for this paper. We use definitions and notations of the pizza decomposition from subsection \ref{Subsection: Pizza Decomp}. In particular, $\beta_i$, $Q_i$, $\mu_i$ and $q_i$ are as in Definition \ref{Def:Pizza decomp}.

\begin{Lem}\label{Lem: max abn zone fisrt slice property}
	Let $X=T(\gamma_1,\gamma_2)$ be a non-singular H\"older triangle partitioned by an interior arc $\gamma$ into two normally embedded H\"older triangles $X_1=T(\gamma_1,\gamma)$ and $X_2=T(\gamma,\gamma_2)$. Let $f\colon (X_1,0)\rightarrow (\R,0)$ be the function given by $f(x)=d(x,X_2)$, and let $\{T_i=T(\lambda_{i-1},\lambda_i)\}_{i=1}^p$ be a pizza on $X_1$ associated with $f$ such that $\lambda_{0}=\gamma$. Then, $\mu_1(q_{\theta})=itord(\theta,\gamma)$ for every arc $\theta \subset T_1$. Moreover, $\mu_1(q) = q$ for all $q\in Q_1$.
\end{Lem}
\begin{proof}
	Since the maximum of $\mu_1$ is $\mu_1(q_{\gamma})=\mu_1(\infty)=\infty$, by Proposition \ref{Prop:width function properties elementary triangle}, we have  $\mu_1(q_{\theta})=itord(\theta,\gamma)$ for every arc $\theta \subset T_1$.
	
	As $\gamma$ is Lipschitz non-singular, there is a normally embedded $\alpha$-H\"older triangle $X'=T(\tilde{\gamma_1},\tilde{\gamma_2}) \subset X$, with $\tilde{\gamma_1}\subset X_1$, such that $\gamma\in G(X')$. We are going to prove that, for each arc $\theta\subset T_1$ such that $itord(\theta,\gamma)>\alpha$, we have $\mu_1(q_{\theta})=q_{\theta}$. Indeed, given such an arc $\theta \subset T_1$, by the Arc Selection Lemma, there is an arc $\theta' \subset X_2$ such that $q_{\theta}=tord(\theta,\theta')$. We claim that $tord(\theta,\theta')=itord(\theta,\theta')$. Suppose, by contradiction, that $tord(\theta,\theta')>itord(\theta,\theta')$. As $X_1$ and $X_2$ are normally embedded, $tord(\theta,\gamma)=itord(\theta,\gamma)$ and $tord(\gamma,\theta')=itord(\gamma,\theta')$. Thus, Lemma \ref{Lem:beta-bubble two NE pieces} implies that $itord(\theta,\theta')=tord(\gamma,\theta')=tord(\theta,\gamma)>\alpha$ and consequently, since $\gamma \in G(X')$, $\theta'\subset X'$, a contradiction with $X'$ being normally embedded. Then, since $T(\theta,\gamma)\cup X_2$ is normally embedded, Lemma \ref{Lem: width function on NE HT} implies that $q_{\theta}=itord(\theta,\gamma)$. Finally,  $$q_{\theta}=tord(\theta,\theta')=itord(\theta,\theta')= itord(\theta,\gamma)=\mu_1(q_{\theta}).$$ Hence, since $\mu_1$ is linear, we have  $\mu_1(q) = q$ for all $q\in Q_1$.
\end{proof}

\begin{Lem}\label{Lem: T1 cup T' is NE and T1 cup T2 cup T' is not}
	Let $X$, $X_1$, $X_2$, $f$ and $\{T_i\}$ be as in Lemma \ref{Lem: max abn zone fisrt slice property}. Then,
	\begin{enumerate}
		\item $T_1\cup X_2$ is normally embedded.
		\item If $p>1$ and $\{T_i\}$ is a minimal pizza then $(T_1\cup T_2)\cup X_2$ is not normally embedded.
	\end{enumerate}
\end{Lem}
\begin{proof}
	$(1)$ If $T_1\cup X_2$ is not normally embedded then there are arcs $\theta\subset T_1$ and $\theta'\subset X_2$ such that $tord(\theta,\theta')>itord(\theta,\theta')$. Thus, $q_{\theta}\ge tord(\theta,\theta')>itord(\theta,\theta')$. However, $\mu_1(q_{\theta})=itord(\theta,\gamma)$ and, since $X_1$ and $X_2$ are normally embedded, by Lemma \ref{Lem:beta-bubble two NE pieces}, $itord(\theta,\gamma)=itord(\theta,\theta')$. Then, $\mu_1(q_{\theta})=itord(\theta,\theta') <q_{\theta}$, a contradiction with Lemma \ref{Lem: max abn zone fisrt slice property}.
	
	$(2)$ If $(T_1\cup T_2)\cup X_2$ is normally embedded, Lemmas \ref{Lem: width function on NE HT} and \ref{Lem: max abn zone fisrt slice property} imply that $T_1\cup T_2$ is a pizza slice, a contradiction with $\{T_i\}$ being a minimal pizza.
\end{proof}

\begin{Lem}\label{Lem: beta1 less beta2 and beta1 equals beta2}
	Let $X=T(\gamma_1,\gamma_2)$ be a non-singular H\"older triangle and $\gamma$ an interior arc of $X$. Let $T=T(\lambda,\gamma)$ and $T'=T(\gamma,\lambda')$ be normally embedded H\"older triangles in $X$ such that $T\cap T'=\gamma$ and $tord(\lambda,\lambda')>itord(\lambda,\lambda')$. Let $f\colon (T,0)\rightarrow (\R,0)$ be the function given by $f(x)=d(x,T')$, and let $\{T_i=T(\lambda_{i-1},\lambda_i)\}_{i=1}^p$ be a minimal pizza on $T$ associated with $f$ such that $\lambda_0=\gamma$. Then
	\begin{enumerate}
		\item If $\beta_2 <\beta_1$ then, for every $\sigma\in \F$ such that $\beta_2<\sigma<\beta_1$, there are arcs $\theta\subset T_2$ and $\theta'\subset T'$ such that $itord(\theta,\gamma)=\sigma$ and $tord(\theta,\theta')>itord(\theta,\theta')$.
		\item If $\beta_2=\beta_1$ then, for every arc $\theta\subset T_2$ such that $tord(\theta,\lambda_2)>\beta_2$, there is an arc $\theta'\subset T'$ such that $tord(\theta,\theta')>itord(\theta,\theta')$.
	\end{enumerate}
\end{Lem}
\begin{proof}
	For both items $(1)$ and $(2)$ we shall consider the following three cases:
		
	\textit{Case 1} - $Q_2$ is not a point and $\mu_{2}(q_2)=M$ is the maximum of $\mu_2$: Proposition \ref{Prop:width function properties elementary triangle} and Lemma \ref{Lem:itord>beta in H_{beta} and B_{beta} implies same ordf} imply that the minimum of $\mu_2$ is $\mu_2(q_1)=\beta_2=\mu_{2}(q_{\gamma'})$ for every arc $\gamma'\subset T_2$ such that $itord(\gamma',\lambda_2)=\beta_2$. Then, since $q_1=\beta_1$, by Lemma \ref{Lem: max abn zone fisrt slice property}, we have $\beta_1=q_{\gamma'}$, for every $\gamma'\subset T_2$ such that $itord(\gamma',\lambda_2)=\beta_2$.
		
	\textit{Case 2} - $Q_2$ is a point: Since by Lemma \ref{Lem: max abn zone fisrt slice property}, $q_1=\beta_1$, for every $\gamma'\subset T_2$ we have $\beta_1=q_{\gamma'}$.
	
	\textit{Case 3} - $Q_2$ is not a point and $\mu_{2}(q_1)=M$ is the maximum of $\mu_2$: Proposition \ref{Prop:width function properties elementary triangle} implies that $itord(\gamma',\lambda_1)=\mu_2(q_{\gamma'})$ for every arc $\gamma'\subset T_2$ such that $itord(\gamma',\lambda_1)\le M\le \beta_1$. Moreover, if $\gamma'\subset T_2$ and $itord(\gamma',\gamma)<\beta_1$ then $itord(\gamma',\lambda_1)=itord(\gamma',\gamma)=\mu_2(q_{\gamma'})$. Since $itord(\gamma',\gamma)=\mu_1(q_{\gamma'})=q_{\gamma'}$ for each arc $\gamma'\subset T_1$ and $\mu_2(q)\le q$ for each $q\in Q_2$, we have $itord(\gamma',\lambda_1)=\mu_2(q_{\gamma'})<q_{\gamma'}$ for every $\gamma'\subset T_2$ such that $itord(\gamma',\gamma)<\beta_1$. Otherwise $T_1\cup T_2$ would be a pizza slice, a contradiction with $\{T_i\}$ being a minimal pizza.
	
	$(1)$ Consider $\sigma \in \F$ such that $\beta_2<\sigma<\beta_1$ and an arc $\theta \subset T_2$ such that $\sigma=itord(\theta,\gamma)$. Let $\theta'\subset T'$ be an arc such that $q_{\theta} = tord(\theta,\theta')$.
	
	Suppose that $Q_2$ is as in case 1 or 2. Note that $\sigma>\beta_2$ implies $itord(\theta,\lambda_2)=\beta_2$. Thus, by cases 1 and 2 considered above, we have $q_{\theta}=\beta_1$. As $\sigma<\beta_1$, we have $$tord(\theta, \theta')=q_{\theta}=\beta_1>\sigma=itord(\theta,\gamma)\ge itord(\theta, \theta').$$
	
	If $Q_2$ is as in case 3 then $$tord(\theta, \theta')=q_{\theta}>itord(\theta,\lambda_1)=itord(\theta,\gamma)\ge itord(\theta,\theta').$$
	
	$(2)$ Suppose that $\beta_2=\beta_1$ and consider an arc $\theta\subset T_2$ such that $tord(\theta,\lambda_2)>\beta_2$. Let $\theta'\subset T'$ be an arc such that $q_{\theta} = tord(\theta,\theta')$. Lemma \ref{Lem: T1 cup T' is NE and T1 cup T2 cup T' is not} implies that $tord(T_2,T')>\beta_2$. Let $\tilde{\gamma}\subset T_2$ be an arc such that $tord(\tilde{\gamma},T')>\beta_2$. Note that $itord(\tilde{\gamma},\lambda_1) = \beta_2$, otherwise, by the non-archimedean property, we would have $T_1\cup T'$ not normally embedded, a contradiction with Lemma \ref{Lem: T1 cup T' is NE and T1 cup T2 cup T' is not}.
	
	If $Q_2$ is as in case 1 or 2 then $\beta_1=q_1=ord_{\tilde{\gamma}}f>\beta_2$, a contradiction. Then, it is enough to consider $Q_2$ as in case 1. Thus, we have $ord_{\tilde{\gamma}}f=q_{\theta}$ and consequently, $$tord(\theta, \theta')=q_{\theta}=q_{\tilde{\gamma}}>\beta_2=itord(\theta,\gamma)\ge itord(\theta,\theta').$$
\end{proof}

\begin{Lem}\label{Lem: maximal abnormal zone first and second slices}
	Let $A$, $T$, $T'$ be as in Lemma \ref{Lem: maximal abnormal zone alpha le beta} and let $\{T_i=T(\lambda_{i-1},\lambda_i)\}_{i=1}^p$ be a minimal pizza on $T$ associated with the function $f\colon (T,0)\rightarrow (\R,0)$ given by $f(x)=d(x,T')$, such that $\lambda_0=\gamma$. Then $p>1$, and one can choose the arcs $\lambda$ and $\lambda'$ in Lemma \ref{Lem: maximal abnormal zone alpha le beta} so that $p=2$ and $\lambda=\lambda_2$. Moreover, $\mu(T_2)\le \mu(T_1)=\beta$.
\end{Lem}
\begin{proof}
	Lemma \ref{Lem: T1 cup T' is NE and T1 cup T2 cup T' is not} implies that $p>1$, since otherwise we would have $T\cup T'$ normally embedded. Since $(T_1\cup T_2)\cup T'$ is not normally embedded, we can choose $\lambda=\lambda_2$ and have $p=2$.
Moreover, as $T_1\cup T'$ is normally embedded, $V(T_1)\subset A$ thus $\beta \le \beta_1=\mu(T_1)$.
	
	Since $\{T_i\}$ is a minimal pizza, $\mu(T_2)=\beta_2 \le \beta_1$.
From now on we assume that $\lambda_2=\lambda$.
We can further assume that any arc $\gamma' \subset T$ such that $tord(\gamma',\gamma)>\beta$ belongs to $A$.
If $\beta_2>\beta$ then $\lambda_2=\lambda \in A$ and we obtain the same contradiction as in the proof of Lemma \ref{Lem: maximal abnormal zone alpha le beta}. Thus we have $\beta_2\le \beta\le \beta_1$. It remains to prove that $\beta = \beta_1$. 	
	
	If $\beta_1 > \beta$ then, in particular, $\beta_1 > \beta_2$, since $\beta_2\le \beta$. Lemma \ref{Lem: beta1 less beta2 and beta1 equals beta2} implies that we can find arcs $\theta \subset T_2$ and $\theta'\subset T'$ such that $\beta < itord(\theta,\gamma) <\beta_1$ and $tord(\theta,\theta')>itord(\theta,\theta')$. Replacing $\lambda=\lambda_2$ with $\theta$ and $\lambda'$ with $\theta'$, we obtain a minimal pizza $\{T_1,T_2\}$ such that $\beta_2>\beta$, a contradiction with $\beta_2\le \beta$ for any minimal pizza $\{T_1,T_2\}$ of $T$.
\end{proof}

\begin{Cor}\label{Cor: max abn zone is perfect}
	Let $X$ be a non-singular H\"older triangle, and $A\subset V(X)$ a maximal abnormal $\beta$-zone. Then $A$ is a closed perfect $\beta$-zone.
\end{Cor}

\begin{proof}
	This an immediate consequence of Lemmas \ref{Lem: T1 cup T' is NE and T1 cup T2 cup T' is not} and \ref{Lem: maximal abnormal zone first and second slices}.
\end{proof}

\begin{remark}\label{REM: cases of the main theorem}
	\normalfont The links of bubbles shown in Figs.~\ref{fig:bubble snake and non-snake bubbles} and ~\ref{fig:nonsnakebubble1} are  examples of the possibilities for the minimal pizza decomposition in Lemmas \ref{Lem: beta1 less beta2 and beta1 equals beta2} and \ref{Lem: maximal abnormal zone first and second slices}:
	
	1. In Fig.~\ref{fig:bubble snake and non-snake bubbles}a, the triangle $T=T(a,\theta)$ has exactly two pizza slices with $\lambda_0=\theta$, $\lambda_1$ being any generic arc of $T$, and $\lambda_2=a$. Moreover, $\beta_1=\beta_2=\beta$ and $Q_2$ is not a point. Also, the maximum of $\mu_2$ is $\mu_2(q_2)$.
	
	2. In Fig.~\ref{fig:bubble snake and non-snake bubbles}b, the triangle $T=T(a,\theta)$ has exactly three pizza slices with the same exponent $\beta$, where $\lambda_0=\theta$, $\lambda_1$ is any generic arc of $T(b,\theta)$, $\lambda_2=b$ and $\lambda_3=a$. Moreover, $Q_2$ is not a point, maximum of $\mu_2$ is $\mu_2(q_2)$, and $Q_3$ is a point with $q_3>\beta$.
	
	3. In Fig.~\ref{fig:bubble snake and non-snake bubbles}c, the triangle $T=T(a,\theta)$ has exactly two pizza slices with $\lambda_{0}=\theta$, $\lambda_1$ being any arc in $T$ having exponent $\beta$ with $\theta$, and $\lambda_2=a$. Moreover, $\beta_2=\eta < \beta=\beta_1$ and $Q_2$ is a point.
	
	4. In both Fig.~\ref{fig:bubble snake and non-snake bubbles}d and Fig.~\ref{fig:bubble snake and non-snake bubbles}e, the triangle $T=T(a,\theta)$ has exactly two pizza slices with $\lambda_{0}=\theta$, $\lambda_1$ being any arc in $T$ having exponent $\beta$ with $\theta$, and $\lambda_2=a$. Moreover, $\beta_2=\eta < \beta=\beta_1$, $q_2>\beta$ and $\mu_2(q_1)$ may be either the maximum or the minimum of $\mu_2$. If $\max \mu_2 = \mu_2(q_1)$ then $\max \mu_2\le\beta$ and the slope of $\mu_2$ is negative in the case of Fig.~\ref{fig:bubble snake and non-snake bubbles}d and positive in the case of Fig.~\ref{fig:bubble snake and non-snake bubbles}e. Otherwise, $\max \mu_2 < \beta$. In both cases, if $\mu_2(\lambda_1)<\beta$ then the bubble contains the bubble in Fig.~\ref{fig:bubble snake and non-snake bubbles}d with $\eta=\mu(\lambda_1)$. If $\max \mu_2 = \mu_2(q_2)$ then the slope of $\mu_2$ is positive in Fig.~\ref{fig:bubble snake and non-snake bubbles}d and negative in Fig.~\ref{fig:bubble snake and non-snake bubbles}e.  	
	
	5. In Fig.~\ref{fig:nonsnakebubble1}, the minimal pizza on $T=T(\gamma_{1},\theta)$ such that $\lambda_{0}=\theta$, has exactly four pizza slices, each of them with exponent $\beta$.
\end{remark}

\begin{Lem}\label{Lem: maximal abnormal zone and pancake decomp}
	Let $X$ be a non-singular H\"older triangle and $\{X_k=T(\theta_{k-1},\theta_k)\}_{k=1}^{p}$ a minimal pancake decomposition of $X$ with $\beta_k=\mu(X_k)$. If $A\subset V(X)$ is a maximal abnormal $\beta$-zone then:
	\begin{enumerate}
		\item the zone $A$ has non-empty intersection with at least two of the zones $V(X_k)$;
		\item if $V(X_k)\cap A\ne \emptyset$ then $\beta_k \le \beta$;
	\end{enumerate}
\end{Lem}
\begin{proof}
	$(1)$ Suppose, by contradiction, that $A$ intersects only a single zone $V(X_k)$. Then $A\subset V(X_k)$ and $\mu(X_k)\le \beta$. Given an arc $\gamma \in A$ there exist arcs $\lambda, \lambda'\in V(X)$ such that $T=T(\lambda,\gamma)$ and $T'=T(\gamma,\lambda')$ are normally embedded H\"older triangles with a common boundary arc $\gamma$ and $tord(\lambda,\lambda')>itord(\lambda,\lambda')$. Let $\alpha=itord(\lambda,\lambda')$. Lemma \ref{Lem:beta-bubble two NE pieces} implies that $\mu(T)=\mu(T')=\alpha$.
	
	Since $X_k$ is normally embedded, one of the arcs $\lambda$ and $\lambda'$, say $\lambda$, is not contained in $X_k$. Assume that $\theta_k \subset T(\lambda, \gamma)$. As $\{X_k\}$ is a minimal pancake decomposition, we can assume that $\lambda$ and $\lambda'$ are in adjacent pancakes, $\lambda'\in V(X_k)$ and $\lambda\in V(X_{k+1})$. Then, $\theta_k$ is abnormal, since $T(\lambda',\theta_k)$ and $T(\theta_k,\lambda)$ are normally embedded. However, by Lemma \ref{Lem: T1 cup T' is NE and T1 cup T2 cup T' is not}, there exist arcs of $A$ in $V(X_{k+1})$, a contradiction with $A\subset V(X_k)$.
	
	$(2)$ Suppose, by contradiction, that $V(X_k)\cap A\ne \emptyset$ and $\mu(X_k) > \beta$. Corollary \ref{Cor: max abn zone is perfect} implies that $V(X_k)\subset A$. In particular, $\theta_k$ is abnormal. Thus, there are arcs $\lambda, \lambda'\in V(X)$ such that $T=T(\lambda,\theta_k)$ and $T'=T(\theta_k,\lambda')$ are normally embedded $\alpha$-H\"older triangles with a common boundary arc $\theta_k$ and $tord(\lambda,\lambda')>itord(\lambda,\lambda')$. As $\{X_k\}$ is a minimal pancake decomposition, we may assume that $\lambda$ and $\lambda'$ are in adjacent pancakes. Hence, $\alpha \ge \mu(X_k) >\beta$ and consequently, $\lambda,\lambda'\in A$, a contradiction with Lemma \ref{Lem: maximal abnormal zone first and second slices}.
\end{proof}

\begin{Lem}\label{Lem: maximal abnormal zone MAIN PART}
	Let $A$, $T$ and $T'$ be as in Lemma \ref{Lem: maximal abnormal zone alpha le beta}, and let $\{T_1,T_2\}$ be a minimal pizza on $T$ associated with $f\colon (T,0)\rightarrow (\R,0)$, given by $f(x)=d(x,T')$, such that $\lambda_0=\gamma$ and $\lambda_2=\lambda$ (see Lemma \ref{Lem: maximal abnormal zone first and second slices}). Then
	\begin{enumerate}
		\item If $\mu(T_2)=\beta$ in Lemma \ref{Lem: maximal abnormal zone first and second slices} then $\gamma$ is contained in a $\beta$-bubble snake and $A\subset V(Y)$ where $Y$ is a $\beta$-snake.
		\item If $\mu(T_2)<\beta$ then $\gamma$ is not contained in any snake.
	\end{enumerate}
\end{Lem}
\begin{proof}
	$(1)$ If $\beta_2=\mu(T_2)=\beta$ then, by Lemmas \ref{Lem: beta1 less beta2 and beta1 equals beta2} and \ref{Lem: maximal abnormal zone first and second slices}, we have $\beta_1=\beta_2=\beta$. We claim that $T\cup T'$ is a $\beta$-bubble snake.
Indeed, since $tord(\lambda_1,T')=\beta$ and $tord(\lambda_2,T')>\beta$, we have $\min\mu_2 = \mu_2(q_1)=\beta$ and $\max\mu_2 = \mu_2(q_2)$. Proposition \ref{Prop:width function properties elementary triangle} implies that $q_{\gamma'}=tord(\gamma',T')=\beta$ for every $\gamma'\in G(T)$. Then, every arc in $G(T)$ is abnormal and similarly we can prove that every arc in $G(T')$ is also abnormal. Finally, since $\gamma$ is in a closed perfect abnormal $\beta$-zone by Corollary \ref{Cor: max abn zone is perfect},
it follows that $G(T\cup T')=Abn(T\cup T')$.
	
 Now we are going to prove that, when $\beta_1=\beta_2=\beta$, there is a $\beta$-snake $Y$ such that $A\subset V(Y)$.
 We already proved that $T\cup T'$ is a $\beta$-bubble snake. If $\lambda\notin A$ and $\lambda'\notin A$ then we can take $Y=T\cup T'$. Suppose that $\lambda\in A$ and $\lambda'\notin A$. Since $\lambda\in A$, there are normally embedded $\alpha_1$-H\"older triangles $T'_1=T(\theta_1,\lambda)$ and $T''_1=T(\lambda,\theta'_1)$ such that $T'_1\cap T=T'_1\cap T''_1 =\lambda$ and $tord(\theta_1,\theta'_1)>itord(\theta_1,\theta'_1)$.
Then $\alpha_1\le\beta$ according to Lemma \ref{Lem: maximal abnormal zone alpha le beta}.
 As $T\cup T'$ is not normally embedded, we have $\theta'_1\subset T\cup T'$.
 Thus, $\alpha_1=tord(\lambda,\theta'_1)\ge itord(\lambda,\lambda')=\beta$. Hence $\alpha_1=\beta$ and $T'_1\cup T''_1$ is a $\beta$-bubble snake. If $\theta_1 \notin A$, then $Y=(T'_1\cup T''_1)\cup (T\cup T')$ is a $\beta$-snake and $A\subset V(Y)$. If $\theta_1\in A$ we apply the same argument to find normally embedded $\beta$-H\"older triangles $T'_2=T(\theta_2,\theta_1)$ and $T''_2=T(\theta_1,\theta'_2)$ such that $T'_2\cap T''_2 =\theta_1$ and $tord(\theta_2,\theta'_2)>itord(\theta_2,\theta'_2)$. If $\theta_2\notin A$ then $Y=(T'_2\cup T''_2)\cup (T'_1\cup T''_1)\cup (T\cup T')$ is a $\beta$-snake and $A\subset V(Y)$. If $\theta_2\in A$ we continue applying the same argument. This procedure must stop after finitely many steps, otherwise there would be infinitely many pancakes in a minimal pancake decomposition of $X$, a contradiction. Thus, after finitely many steps, we find an integer $p$ such that $\theta_p \notin A$, thus $Y=(\bigcup_{i=1}^pY_i)\cup (T\cup T')$ is a $\beta$-snake, where $Y_i=T'_i\cup T''_i$, and $A\subset Y$.
	
	$(2)$ It is enough to prove that if $\beta_2=\mu(T_2)<\beta=\beta_1$ then $T\cup T'$ is a non-snake $\beta$-bubble. Let $Y=T\cup T'$. Suppose, by contradiction, that $Y$ is a $\beta_2$-bubble snake. Consider $\alpha'\in \F$ such that $\beta_2 < \alpha' <\beta_1$. By Lemma \ref{Lem: beta1 less beta2 and beta1 equals beta2}, there is an arc $\theta \subset T_2$ such that $tord(\theta, T')>\beta_2$, a contradiction with Proposition \ref{Prop:bubble snake and its generic arcs}.
\end{proof}

\begin{Teo}\label{Teo: Main HT decomposition}
	Let $X$ be a surface germ. Then $V(X)$ is the union of finitely many maximal normal (possibly singular) zones and finitely many maximal abnormal zones. Moreover, each maximal abnormal zone is closed perfect, and if its order is $\beta$ then it is either a circular snake (see Remark \ref{REM: circular snake}), or the set of generic arcs in a $\beta$-snake $Y\subset X$, or else it is $\beta$-complete and, for any small $\epsilon > 0$, contained in $V(T_\eta)$, where $\eta=\beta - \epsilon$ and $T_\eta$ is a non-snake $\eta$-bubble (see Figures ~\ref{fig:bubble snake and non-snake bubbles}c, ~\ref{fig:bubble snake and non-snake bubbles}d, ~\ref{fig:bubble snake and non-snake bubbles}e)).
\end{Teo}

\begin{proof} It is enough to prove the statement for a surface germ $X$ with connected link. If the link of $X$ is a circle and all arcs in $X$ are abnormal then $X$ is a circular snake. Otherwise, $X$ must contain some normal (possibly, Lipschitz singular) arcs.
Existence of pancake decomposition implies that $tord(\gamma,Abn(X))<\infty$ for each normal arc $\gamma$ of $X$.
If $\alpha>tord(\gamma,Abn(X))$ then an $\alpha$-horn neighborhood $H_{a,\alpha}(\gamma)$ of $\gamma$ in $X$ (see Definition \ref{Def:horn-neighbourhood} and Remark \ref{Rmk:horn-neighborhood}) consists of normal arcs.
Removing $H_{a,\alpha}(\gamma)$ from $X$ does not change $Abn(X)$.
Indeed, let $\lambda\in Abn(X)$ be an arc such that $\lambda=T_1\cap T_2$ where $T_1=T(\lambda_1,\lambda)$ and $T_2=T(\lambda,\lambda_2)$ are normally embedded $\beta$-H\"older triangles such that $tord(\lambda_1,\lambda_2)>\beta$ and $\gamma\subset T_1$. In particular, $\beta<\alpha$.
If $T(\gamma,\lambda)\cup T_2$ is not normally embedded then $\lambda\in Abn(X\setminus H_{a,\alpha}(\gamma))$.
Otherwise, we would have $tord(\gamma,\lambda_1)=\beta$, thus $\gamma$ would be an abnormal arc, a contradiction.

Removing from $X$, if necessary, open horn neighborhoods of Lipschitz singular arcs and finitely many normal arcs, we get finitely many disjoint non-singular H\"older triangles $X_i\subset X$ such that $Abn(X)\subset \cup Abn(X_i)$.
Thus it is enough to prove the statement for a non-singular H\"older triangle $X$.

	By definition, abnormal zones do not intersect with normal zones. Moreover (see Definition \ref{Def: maximal abnormal and normal zones}) there are no adjacent maximal abnormal zones and adjacent maximal normal zones.
	
	Let $\{X_k=T(\theta_{k-1},\theta_k)\}$ be a minimal pancake decomposition of $X$. To prove that there are finitely many abnormal zones in $V(X)$ it is enough to prove that each zone $V(X_k)$ intersects finitely many maximal abnormal zones in $V(X)$. Suppose, by contradiction, that there are infinitely many maximal abnormal zones $A_1,A_2,\ldots$ in $V(X)$ such that $V(X_k)\cap A_i\ne \emptyset$ for all $i=1,2,\ldots$. Lemma \ref{Lem: maximal abnormal zone and pancake decomp} implies that $\mu(X_k) \le \mu(A_i)$ for all $i=1,2,\ldots$. One of the boundary arcs of $X_k$ must belong to $A_i$ for some $i$, otherwise, since $\mu(X_k) \le \mu(A_i)$ for all $i=1,2,\ldots$, we would have $A_i\subset V(X_k)$ for all $i$, a contradiction with Lemma \ref{Lem: maximal abnormal zone and pancake decomp}. Assume that $\theta_k \in A_i$ and consider $A_j$ for $j\ne i$. Thus, $\theta_{k-1}\in A_j$, otherwise we would have $A_j\subset V(X_k)$, since $A_i\cap A_j=\emptyset$, a contradiction with Lemma \ref{Lem: maximal abnormal zone and pancake decomp}. Then, $\theta_{k-1}\in A_j$, $\theta_k \in A_i$ and $A_l\subset V(X_k)$ for all $l=1,2,\ldots$ with $l\ne i$ and $l\ne j$, a contradiction with Lemma \ref{Lem: maximal abnormal zone and pancake decomp}.
Since there exist finitely many abnormal zones it follows that there are finitely many maximal normal zones in $V(X)$.
	
	Finally, let $A$ be a maximal abnormal $\beta$-zone in $V(X)$. Corollary \ref{Cor: max abn zone is perfect} implies that $X$ is a closed perfect zone. Consider an arc $\gamma\in A$ and arcs $\lambda,\lambda'\in V(X)$ such that $T=T(\lambda,\gamma)$ and $T'=T(\gamma,\lambda')$ are normally embedded H\"older triangles with $tord(\lambda,\lambda')>itord(\lambda,\lambda')$. Let $f\colon (T,0)\rightarrow (\R,0)$ be the function given by $f(x)=d(x,T')$, and $\{T_i\}$ a minimal pizza on $T$ associated with $f$. By Lemma \ref{Lem: maximal abnormal zone first and second slices}, we can assume that $p=2$ and $\beta_2\le \beta_1=\beta$. If $\beta_2= \beta_1=\beta$ then, by Lemma \ref{Lem: maximal abnormal zone MAIN PART}, $A$ is contained in a $\beta$-snake. If $\beta_2<\beta_{1}$ then, also by Lemma \ref{Lem: maximal abnormal zone MAIN PART}, $A$ is contained in the non-snake bubble $Y=T\cup T'$ and, by Lemma \ref{Lem: beta1 less beta2 and beta1 equals beta2}, for any $\epsilon>0$ such that $\beta_2 < \eta=\beta - \epsilon <\beta_1=\beta$, $A\subset V(T_\eta)$ where $T_\eta \subset Y$ is a non-snake $\eta$-bubble.	
\end{proof}

\section{Combinatorics of snakes}\label{Section: Names of snakes}
In this section we assign a word to a snake. It is a combinatorial invariant of the snake reflecting the order, with respect to a fixed orientation, in which nodal zones belonging to each of its nodes appear.

\subsection{Words and partitions}\label{Subsec: words and partitions}

\begin{Def}\label{DEF: word}
	\normalfont Consider an alphabet $A=\{x_1,\ldots,x_n\}$. A \textit{word} $W$ of length $m=|W|$ over $A$ is a finite sequence of $m$ letters in $A$, i.e., $W=[w_1 \cdots w_m]$ with $w_i\in A$ for $1\le i \le m$. One also considers the \textit{empty word} $\varepsilon=[\;]$ of length $0$. Given a word $W=[w_1\cdots w_m]$, the letter $w_i$ is called the $i$-\textit{th entry} of $W$.
If $w_i=x$ for some $x\in A$, it is called a \textit{node entry} of $W$ if it is the first occurrence of $x$ in $W$.
Alternatively, $w_i$ is a node entry of $W$ if $w_j\ne w_i$ for all $j<i$.
\end{Def}

\begin{Def}\label{DEF: open, semi-open subwords etc}
	\normalfont
	Given a word $W=[w_1\cdots w_m]$, a \textit{subword} of $W$ is either an empty word or a word $[w_j\cdots w_k]$ formed by consecutive entries of $W$ in positions $j, \ldots, k$, for some $j\le k$. We also consider \textit{open} subwords $(w_j\cdots w_k)$ formed by the entries of $W$ in positions $j+1, \ldots, k-1$, for some $j<k$, and \textit{semi-open} subwords $(w_j\cdots w_k]$ and $[w_j\cdots w_k)$ formed by the entries of $W$ in positions $j+1, \ldots, k$ and $j, \ldots, k-1$, respectively.

\end{Def}

\begin{Def}\label{Def: partition assoc with a snake name}
	\normalfont Let $W=[w_1\cdots w_m]$ be a word of length $m$ containing $n$ distinct letters $x_1,\ldots,x_n$.
We associate with $W$ a partition $P(W)=\{I_1,\ldots,I_n\}$ of the set $\{1,\ldots,m\}$ where $i\in I_j$ if $w_i=x_j$.
\end{Def}

\begin{remark}\label{REM: convention for word and equivalent words}
	\normalfont Note that $P(W)$ does not depend on the alphabet, only on positions where the same letters appear. For convenience we often assign $a$ (or $x_1$) to the first letter of the word $W$, $b$ (or $x_2$) to the first letter of $W$ other than $a$, and so on. Two words $W$ and $W'$ are equivalent if $P(W)=P(W')$. In particular, equivalent words have the same length and the same number of distinct letters. For example, the words $X=abcdacbd$, $Y=bcdabdca$ and $Z=xyzwxzyw$ are equivalent, since
$P(X)=P(Y)=P(Z)=\{ \{1,5\},\{2,7\},\{3,6\},\{4,8\}\}$.
\end{remark}

\begin{Def}\label{Def: primitive and binary words}
	\normalfont A word $W$ is \textit{primitive} if it contains no repeated letters, i.e., if each part of $P(W)$ contains a single entry. We say that $W=[w_1\cdots w_m]$ is \textit{semi-primitive} if $w_1=w_m$ and the subword $[w_1\cdots w_m)$ of $W$ is primitive, i.e., if each part of $P(W)$ except $\{w_1, w_m\}$ contains a single entry. A word $W$ is \textit{binary} if each of its letters appears in $W$ exactly twice, i.e., if each part of $P(W)$ contains exactly two entries.
\end{Def}

\subsection{Snake names}\label{Subsec: snake names}

\begin{Def}\label{Def: rules for snakes names}
	\normalfont Given a non-empty word $W=[w_1\cdots w_m]$, we say that $W$ is a \textit{snake name} if the following conditions hold:	
	\begin{enumerate}
		\item[$(i)$] Each of the letters of $W$ appears in $W$ at least twice;
		\item[$(ii)$] For any $k\in\{2,\ldots,m-1\}$, there is a semi-primitive subword $[w_j\cdots w_l]$ of $W$ such that $j<k<l$.
	\end{enumerate}
\end{Def}

\begin{remark}
	\normalfont Note that every word equivalent to a snake name $W$ is also a snake name.
\end{remark}
	
\begin{remark}\label{Rem: order of letters in words}
	\normalfont The word $[aa]$ (or any equivalent word) is the only snake name of length two. No snake name of length greater than two contains the same letter in consecutive positions. There are no snake names of length three, and the words $[abab]$ and $[ababa]$ are the only snake names, up to equivalence, of length four and five, respectively.
\end{remark}

\begin{Exam}\label{Exam: snake names}
	\normalfont The word $W=[abcdacbd]$ is a snake name, while the word $W'=[abacdcbd]$ is not, since the entry $w_3=a$ of $W'$ does not satisfy condition $(ii)$ of Definition \ref{Def: rules for snakes names}. There may be more than one subword in a snake name satisfying condition  $(ii)$ of Definition \ref{Def: rules for snakes names} for a fixed position $k$. For example both subwords $[abcda]$ and $[cdac]$ of $W$ satisfy condition $(ii)$ for its entry $w_4=d$.
\end{Exam}

\begin{Def}\label{Def: word associated with a snake}
	\normalfont  Let $T=T(\gamma_1,\gamma_2)$ be a $\beta$-snake with $n$ nodes $\N_1,\ldots, \N_n$ and $m$ nodal zones $N_1,\ldots N_m$. From now on we assume that the link of $T$ is oriented from $\gamma_1$ to $\gamma_2$, and the nodal zones of $T$ are enumerated in the order in which they appear when we move along the link of $T$ from $\gamma_{1}$ to $\gamma_{2}$. We enumerate the nodes of $T$ similarly, starting with the node $\N_1$ containing $\gamma_1$, skipping the nodes for which the numbers were already assigned.
In particular, $\gamma_1 \in N_1 \subset \N_1$ and $\gamma_2 \in N_m$, but $N_m$ does not necessarily belong to $\N_n$.

Consider an alphabet $A=\{x_1,\ldots,x_n\}$ where each letter $x_j$ is assigned to the node $\N_j$ of $T$. A word $W=[w_1\cdots w_m]$ over $A$ is associated with the snake $T=T(\gamma_1,\gamma_2)$ (notation $W=W(T)$) if, while moving along the link of $T$ from $\gamma_1$ to $\gamma_2$, the $i$-th entry $w_i$ of $W$ is the letter $x_j$ assigned to the node $\N_j$ to which the nodal zone $N_i$ belongs.
\end{Def}

\begin{Prop}\label{Prop: word assoc with a snake is a snake name}
	Let $T=T(\gamma_1,\gamma_2)$ be a snake, other than a spiral snake, and let $W=W(T)$ be the word associated with $T$. Then $W$ is a snake name satisfying conditions $(i)$ and $(ii)$ of Definition \ref{Def: rules for snakes names}.
\end{Prop}
\begin{proof}
	Condition $(i)$ of Definition \ref{Def: rules for snakes names} holds since each node of $T$ contains at least two nodal zones (see Proposition \ref{Prop: each node contains at least two nodal zones}).
	
	For a bubble snake condition $(ii)$ of Definition \ref{Def: rules for snakes names} is empty, thus we may assume that $T$ is not a bubble snake. Consider $w_k$, with $1<k<m$, and the nodal zone $N_k$ associated with $w_k$. Since $N_k$ is an interior nodal zone, every arc in $N_k$ is abnormal. Let $\gamma \in N_k$, and let $\lambda_1 \subset T(\gamma_1,\gamma)$ and $\lambda_2 \subset T(\gamma,\gamma_2)$ be two arcs such that $T(\lambda_1,\gamma)$ and $T(\gamma,\lambda_2)$ are normally embedded H\"older triangles with $tord(\lambda_1,\lambda_2)>itord(\lambda_1,\lambda_2)$ (see Remark \ref{Rem:triangles of abnormal arc}). Propositions \ref{Prop: each node contains at least two nodal zones} and \ref{Prop:bubble inside snake} imply that $\lambda_1$ and $\lambda_2$ belong either to distinct nodal zones in the same node or to distinct segments.
	
	We can assume, replacing the arcs $\lambda_1$ and $\lambda_2$ if necessary, that $\lambda_1$ and $\lambda_2$ belong to different nodal zones in the same node. Indeed, suppose that $\lambda_1$ and $\lambda_2$ belong to segments $S_1$ and $S_2$, respectively. Let $\N$ and $\N'$ be the nodes containing the nodal zones adjacent to $S_1$ and $S_2$ (see Proposition \ref{Prop: segments with contact implies adjacent nodal zones with contact}). We can assume that $T(\lambda_1,\gamma)$ and $T(\gamma,\lambda_2)$ do not contain arcs in nodal zones of the same node. Otherwise, $\lambda_1$ and $\lambda_2$ can be replaced by those arcs. Then, $T(\lambda_1,\gamma)$ contains arcs in a nodal zone in one of these nodes, say $N\subset\N$, and $T(\gamma,\lambda_2)$ contains arcs in a nodal zone $N'\subset\N'$. This implies that $T(\gamma,\lambda_2)$ does not contain arcs in $N$. If $T(\gamma,\lambda_2)$ contains arcs of some other nodal zone $N''$ in $\N$ other $N$ then $\lambda_1$ and $\lambda_2$ can be replaced by arcs in $N$ and $N''$, respectively. Thus, assume that $T(\gamma,\lambda_2)$ does not contain arcs in $\N$ and consider arcs $\lambda'_1\in N$ and $\lambda'_2\in N'$. Since H\"older triangles $T(\lambda'_1,\gamma)$ and $T(\gamma,\lambda'_2)$ are normally embedded, we can replace $\lambda_i$ by $\lambda'_i$ for $i=1,2$. In fact, $T(\lambda'_1,\gamma)$ is normally embedded because $T(\lambda'_1,\gamma)\subset T(\lambda_1,\gamma)$ where $T(\lambda_1,\gamma)$ is normally embedded. If $T(\gamma,\lambda'_2)$ is not normally embedded, we get a contradiction with $T(\gamma,\lambda_2)$ containing no arcs in $\N$ (see Proposition \ref{Prop: each node contains at least two nodal zones} item $(1)$).
	
Assume now that $\lambda_1$ and $\lambda_2$ belong to nodal zones $N_j$ and $N_l$, respectively, in the same node $\N$, where $j<l$. Since $T(\lambda_1,\gamma)$ and $T(\gamma,\lambda_2)$ are normally embedded, the nodal zone $N_k$ does not belong to $\N$, and consequently, $j<k<l$. Furthermore, we may assume that $V(T(\lambda_1,\lambda_2))$ does not contain distinct nodal zones in the same node other than $\N$.

	Let $w_j$ and $w_l$ be the entries associated with $N_j$ and $N_l$, respectively. By our assumption for $V(T(\lambda_1,\lambda_2))$, the only letters common to the primitive subwords $[w_j\cdots w_k]$ and $[w_k\cdots w_l]$ are $w_k$ and $w_j=w_l$. Hence, the subword $[w_j\cdots w_l]$ is semi-primitive, and condition $(ii)$ of Definition \ref{Def: rules for snakes names} is satisfied.
\end{proof}

\begin{Prop}\label{Prop: subwords in snake names are NE HT}
	Let $T$ be a snake, and $W=[w_1\cdots w_m]$ the word associated with $T$. Let $T'\subset T$ be a H\"older triangle with the boundary arcs in  the nodal zones $N_j$ and $N_k$ of $T$, where $j<k$. Then $T'$ is normally embedded if, and only if, the subword $[w_j\cdots w_k]$ of $W$ is primitive.
\end{Prop}
\begin{proof}
	If two of the nodal zones $N_j,\ldots , N_k$ belong to the same node then, by Proposition \ref{Prop:weak LNE}, $T'$ is not normally embedded.
	
	Conversely, if $T'$ is not normally embedded then there are arcs $\lambda,\lambda' \subset T'$ such that $tord(\lambda,\lambda')>itord(\lambda,\lambda')=\beta$. By the same argument as in the proof of Proposition \ref{Prop: word assoc with a snake is a snake name}, we can assume that $\lambda\in N_{j'}$ and $\lambda'\in N_{k'}$ where $j\le j' <k'\le k$.
Hence $N_{j'}$ and $N_{k'}$ belong to the same node, $w_{j'}=w_{k'}$ and the subword $[w_j\cdots w_k]$ is not primitive.
\end{proof}

\begin{Cor}\label{Cor: subHT is a bubble snake iff subword is semi-primitive}
	Let $T$, $W$ and $T'$ be as in Proposition \ref{Prop: subwords in snake names are NE HT}. Then $T'$ is a bubble snake if, and only if, the subword $[w_j\cdots w_k]$ of $W$ is semi-primitive.
\end{Cor}
\begin{proof}
	If $T'$ is a bubble snake then $N_j$ and $N_k$ belong to the same node. If $[w_j\cdots w_k]$ is not semi-primitive then at least one of the words $[w_j\cdots w_k)$ and $(w_j\cdots w_k]$ is not primitive. If $[w_j\cdots w_k)$ is not primitive (the case when $(w_j\cdots w_k]$ is not primitive is similar) then there are entries $w_{j'}$ and $w_{k'}$ with $j\le j'<k'<k$ such that $w_{j'}=w_{k'}$. Consequently, there are nodal zones $N_{j'}$ and $N_{k'}$ of $T$ such that $tord(N_{j'},N_{k'})>\beta$. As $j\le j'<k'<k$, we have $N_{j'}\cap V(T')\ne \emptyset$ and $N_{k'}\subset G(T')$, a contradiction with Proposition \ref{Prop:bubble snake and its generic arcs}.
	
	Conversely, if $[w_j\cdots w_k]$ is semi-primitive then $N_j$ and $N_k$ are the only nodal zones of $T$ having nonempty intersection with $V(T')$ which belong to the same node. Proposition \ref{Prop: segments with contact implies adjacent nodal zones with contact} implies that $T'$ is a bubble snake.
\end{proof}

\begin{Def}\label{Def: node entry and the sets Pi's}
 \normalfont Let $W=[w_1\cdots w_m]$ be a snake name. If $w_j$ is not a node entry, for some $j=2,\ldots,m$, we define $r(j)$ so that $w_{r(j)}$ is a node entry and $w_{r(j)}=w_j$. If $w_j$ is a node entry then $r(j)=j$.
\end{Def}

\begin{Def}\label{Def: linear HT}
	\normalfont Given arcs $\gamma,\gamma'\subset \R^p$ we define the set $\Delta(\gamma,\gamma')$ as the union of straight line segments, $[\gamma(t),\gamma'(t)]$, connecting $\gamma(t)$ and $\gamma'(t)$ for any $t\ge 0$.
\end{Def}

\begin{Def}\label{Def: arcs and HT assoc with a snake name}
	\normalfont Let $W=[w_1,\ldots,w_m]$ be a snake name of length $m>2$. Consider the space $\R^{2m-1}$ with the standard basis $\mathbf{e}_1,\ldots, \mathbf{e}_{2m-1}$. Let $\alpha,\beta\in \F$, with $1\le \beta < \alpha$, and let $\delta_1,\ldots, \delta_m$ and $\sigma_1,\ldots, \sigma_{m-1}$ be arcs in $\R^{2m-1}$ (parameterized by the first coordinate, which is equivalent to the distance to the origin) such that:
	\begin{enumerate}
		\item $\delta_1(t)=t\mathbf{e}_1$;
		\item for $1<j\le m$, if $w_j$ is a node entry then $\delta_j(t) = t\mathbf{e}_1 + t^\beta \mathbf{e}_j$. Otherwise, $\delta_j(t) = \delta_{r(j)}(t) + t^\alpha \mathbf{e}_j$;
		\item for any $j=1,\ldots,m-1$, we define $\sigma_j(t)=t\mathbf{e}_1 + t^\beta \mathbf{e}_{m+j}$.
	\end{enumerate}
	Consider the $\beta$-H\"older triangles $T_j=\Delta(\delta_j,\sigma_j)\cup \Delta(\sigma_j,\delta_{j+1})$ for $j=1,\ldots,m-1$, and $T'_j=\Delta(\sigma_{j},\delta_{j+1})\cup \Delta(\delta_{j+1},\sigma_{j+1})$ for $j=1,\ldots, m-2$ (see Definition \ref{Def: linear HT}). Let $T_{W}=\bigcup_{j=1}^{m-1}T_j$. The H\"older triangle $T_W =T(\delta_1,\delta_m)$ is the \textit{$\beta$-H\"older triangle associated with the snake name} $W$. Assuming that the link of $T_W$ is oriented from $\delta_1$ to $\delta_m$, the arcs $\delta_j$ and $\sigma_k$ appear in $T_W$ in the following order $\delta_1,\sigma_1,\delta_2,\ldots,\delta_{m-1},\sigma_{m-1},\delta_m$.
\end{Def}

The following Lemma is a consequence of Definition \ref{Def: arcs and HT assoc with a snake name}.
\begin{Lem}\label{Lem: tords of deltas and sigmas}
	The arcs $\delta_1,\ldots, \delta_m, \sigma_1,\ldots, \sigma_{m-1}$ of Definition \ref{Def: arcs and HT assoc with a snake name} satisfy the following:
	\begin{enumerate}
		\item[(i)] $tord(\delta_i,\delta_{j})=\left\{
		\begin{array}{cl}
		\alpha & \text{if} \;  w_i=w_{j} \\
		\beta & otherwise
		\end{array}\right.$ for all $i\ne j$,
		\item[(ii)] $tord(\sigma_i,\delta_{j})=\beta$ for all $i$ and $j$,
		\item[(iii)] $tord(\sigma_i,\sigma_{j})=\beta$ for all $i$ and $j$ with $i\ne j$.
	\end{enumerate}
\end{Lem}

\begin{Lem}\label{Lem: Tj is NE}
	Each $T_j$ in Definition \ref{Def: arcs and HT assoc with a snake name} is a normally embedded $\beta$-H\"older triangle.
\end{Lem}
\begin{proof}
	Note that, for each $j=1,\ldots, m-1$,
$$T_j=\bigcup_{t\ge 0} \big([\delta_j(t),\sigma_j(t)]\cup [\sigma_j(t),\delta_{j+1}(t)]\big)$$
where $[\delta_j(t),\sigma_j(t)]$ and $[\sigma_j(t),\delta_{j+1}(t)]$ are straight line segments with a common endpoint. As $n>1$, any consecutive letters of $W$ are distinct, thus item (i) of Lemma \ref{Lem: tords of deltas and sigmas} implies that $tord(\delta_{j},\delta_{j+1})=\beta$. Also, item (ii) of Lemma \ref{Lem: tords of deltas and sigmas} implies that $tord(\delta_{j},\sigma_{j})=\beta$ and $tord(\sigma_{j},\delta_{j+1})=\beta$. Then, the family of angles $\phi(t)$ formed by the straight line segments $[\delta_j(t),\sigma_j(t)]$ and $[\sigma_j(t),\delta_{j+1}(t)]$ is bounded from below by a positive constant.
	
	This implies that $T_j$ is normally embedded. Indeed, given two arcs $\gamma\subset \Delta(\delta_j,\sigma_j)$ and $\gamma'\subset \Delta(\sigma_j,\delta_{j+1})$, such that $\gamma(t)\in [\delta_j(t),\sigma_j(t)]$ and $\gamma'(t)\in [\sigma_j(t),\delta_{j+1}(t)]$, we have
	$|\gamma'(t)-\gamma(t)|>C\,\max(|\sigma_j(t)-\gamma(t)|,\,|\gamma'(t)-\sigma_j(t)|)$ for some constant $C>0$, thus $$itord(\gamma,\gamma')=\min(tord(\gamma,\sigma_j),\,tord(\sigma_j,\gamma'))=tord(\gamma,\gamma').$$
\end{proof}

\begin{Lem}\label{Lem: sigma HTs are NE and interior deltas are Lips non-singular}
	Each $T'_j$ in Definition \ref{Def: arcs and HT assoc with a snake name} is a normally embedded $\beta$-H\"older triangle.
\end{Lem}
\begin{proof}
	Note that $T'_j = \bigcup_{0\le t} ([\sigma_{j}(t),\delta_{j+1}(t)]\cup [\delta_{j+1}(t),\sigma_{j+1}(t)])$, with $[\sigma_{j}(t),\delta_{j+1}(t)]$ and $[\delta_{j+1}(t),\sigma_{j+1}(t)]$ being straight line segments with a common endpoint. The same argument as in the proof of Lemma \ref{Lem: Tj is NE} implies that $T'_j$ is a normally embedded $\beta$-H\"older triangle.
\end{proof}

\begin{Cor}\label{Cor: HT assoc with a snake name is non-singular}
	Let $W$ be a snake name of length $m>2$, and let $T_W$ be the H\"older triangle associated with $W$ in Definition \ref{Def: arcs and HT assoc with a snake name}. Then $T_W$ is a non-singular H\"older triangle.
\end{Cor}
\begin{proof}
	Since, by Lemma \ref{Lem: Tj is NE}, each $T_j$ is non-singular, it is enough to prove that $\delta_{j}$ is Lipschitz non-singular for each $j=2,\ldots,m-1$. But $\delta_{j}\in I(T'_{j-1})$, where $T'_{j-1}$ is normally embedded by Lemma \ref{Lem: sigma HTs are NE and interior deltas are Lips non-singular}.
\end{proof}

\begin{Lem}\label{Lem: HT assoc with a primitive subsnake name is NE}
	Let $W=[w_1\cdots w_m]$ and $T_W$ be as in Corollary \ref{Cor: HT assoc with a snake name is non-singular}. If $W'=[w_j\cdots w_l]$ is a primitive subword of $W$ then $T(\delta_j,\delta_l)\subset T_W$ is a normally embedded $\beta$-H\"older triangle.
\end{Lem}
\begin{proof}
	Consider constants $c_j,c_{j+1},\ldots, c_l$ and $s_j,s_{j+1},\ldots,s_{l-1}$ such that \begin{equation}\label{Inequality of constants on Lemma sec 7}
	c_j<s_j<c_{j+1}<\dots<c_{l-1}<s_{l-1}<c_l,
	\end{equation}
	and $c_i=0$ if $r(i)=1$ for some $i$ (see Definition \ref{Def: node entry and the sets Pi's}). Consider the ordered sequence of basis vectors \begin{equation*}
	\E=\{\mathbf{e}_{r(j)}, \mathbf{e}_{m+j}, \mathbf{e}_{r(j+1)}, \mathbf{e}_{m+j+1},\ldots,\mathbf{e}_{m+l-1}, \mathbf{e}_{r(l)}\},
	\end{equation*} where each vector $\mathbf{e}_{r(i)}$ is associated with the arc $\delta_i$ and each vector $\mathbf{e}_{m+i}$ is associated with the arc $\sigma_i$ (see Definition \ref{Def: arcs and HT assoc with a snake name}).
	
	We define the linear mapping $\pi \colon \R^{2m-1}\rightarrow \R^2$ given by $\pi(\mathbf{e}_1)=(1,0)$, $\pi(\mathbf{e}_{r(i)}) = (0,c_i)$, $\pi(\mathbf{e}_{m+i})=(0,s_i)$ and $\pi(\mathbf{e}_{i})=(0,0)$ if $\mathbf{e}_{i}\notin \E\cup \{\mathbf{e}_1\}$. We claim that $\pi$ maps $T(\delta_j,\delta_l)$ one-to-one to the $\beta$-H\"older triangle $T(\pi(\delta_j),\pi(\delta_l))$. Indeed, for each $i=j,\ldots,l$, we have $$\pi(\delta_i(t))  =  \pi(\delta_{r(i)}(t) + t^\alpha\mathbf{e}_i) =  t\pi(\mathbf{e}_1) + t^\beta \pi(\mathbf{e}_{r(i)}) =  (t,c_it^\beta).$$

	Similarly, $\pi(\sigma_i) = \pi(t\mathbf{e}_1 + t^\beta\mathbf{e}_{m+i})=(t,s_it^\beta)$ for each $i=j,\ldots,l-1$. Inequality (\ref{Inequality of constants on Lemma sec 7}) implies that the arcs $\pi(\delta_j), \pi(\sigma_j), \pi(\delta_{j+1}), \cdots  , \pi(\sigma_{l-1}), \pi(\delta_l)$ are ordered in $\R^2$ in the same way as $\delta_j,\sigma_j,\delta_{j+1},\ldots,\sigma_{l-1},\delta_l$ are ordered in $T(\delta_j,\delta_l)$ (see Definition \ref{Def: arcs and HT assoc with a snake name}). Then, as each H\"older triangle $\Delta(\delta_i,\sigma_i)$ and $\Delta(\sigma_i,\delta_{i+1})$ is a union of straight line segments and $\pi$ is a linear mapping, it follows that $\pi\colon T(\delta_j,\delta_l) \rightarrow T(\pi(\delta_j),\pi(\delta_l))$ is one-to-one.
	
	One can easily check that
$$tord(\pi(\delta_i),\pi(\sigma_{k}))=tord(\pi(\delta_i),\pi(\delta_{p}))=tord(\pi(\sigma_i),\pi(\sigma_{p}))=\beta$$ for all $i,k,p$ with $i\ne p$.	
	We want to prove that given two arcs $\gamma,\gamma'\subset T(\delta_j,\delta_l)$ we have $tord(\pi(\gamma),\pi(\gamma'))\ge tord(\gamma,\gamma')$. First, note that $\pi$ is Lipschitz, since it is linear. Thus, there is $K>0$ such that $|\pi(x)-\pi(y)|\le K|x-y|$ for every $x,y \in \R^{2m-1}$. Given an arc $\gamma \subset T(\delta_j,\delta_l)$, we may assume that $\gamma \subset T(\delta_i,\sigma_i)$ (if $\gamma \subset T(\sigma_i,\delta_{i+1})$ the argument is the same). Reparameterizing $\gamma$, if necessary, we can assume that $\gamma(t) \in [\delta_i(t),\sigma_i(t)]$ for any $t\ge 0$. Then, as $\delta_i$ and $\sigma_i$ are both parameterized by the first coordinate, $\gamma$ is also parameterized by the first coordinate $t$. So, since $\pi$ maps the first coordinate $t$ of $\delta_i$ and $\sigma_i$ to the first coordinate $t$ of $\pi(\delta_i)$ and $\pi(\sigma_i)$, it follows that $\pi(\gamma)$ is also parameterized by the first coordinate $t$. Hence, given two arcs $\gamma,\gamma'\subset T(\delta_j,\delta_l)$ we have $tord(\pi(\gamma),\pi(\gamma'))\ge tord(\gamma,\gamma')$, since $|\pi(\gamma(t))-\pi(\gamma'(t))|\le K |\gamma(t)-\gamma'(t)|$.

	Now we can finally prove that $T(\delta_j,\delta_l)$ is normally embedded. Suppose, by contradiction, that there are arcs $\gamma,\gamma'\subset T(\delta_j,\delta_l)$ such that $tord(\gamma,\gamma')>itord(\gamma,\gamma')$. Lemmas \ref{Lem: Tj is NE} and \ref{Lem: sigma HTs are NE and interior deltas are Lips non-singular} imply that $\gamma$ and $\gamma'$ cannot be both contained in $T_i$ or $T'_k$ for every $i=1,\ldots, m-1$ and $k=1,\ldots, m-2$ (in particular, $itord(\gamma,\gamma')=\beta$). Then, as the arcs $\pi(\delta_j),\pi(\sigma_j),\pi(\delta_{j+1}),\cdots , \pi(\sigma_{l-1}),\pi(\delta_l)$ are ordered as described above and $tord(\pi(\delta_i),\pi(\sigma_{k}))=tord(\pi(\delta_i),\pi(\delta_{p}))=tord(\pi(\sigma_i),\pi(\sigma_{p}))=\beta$ for all $i,k,p$ with $i\ne p$, we have $tord(\pi(\gamma),\pi(\gamma'))=\beta$. However, we should have $\beta=itord(\gamma,\gamma')<tord(\gamma,\gamma')\le tord(\pi(\gamma),\pi(\gamma'))=\beta$, a contradiction.
\end{proof}

\begin{Cor}\label{Cor: interior deltas of T are abnormal}
	Let $W$ and $T_W$ be as in Lemma \ref{Lem: HT assoc with a primitive subsnake name is NE}. Then $G(T_W)\subset Abn(T_W)$.
\end{Cor}
\begin{proof}
	Note that each arc $\delta_{k}$ of $T_W$ is abnormal, for $k=2,\ldots, m-1$. Indeed, since $W$ is a snake name and $1<k<m$, there is a semi-primitive subword $[w_j\cdots w_l]$ of $W$ with $j<k<l$. In particular, $[w_j\cdots w_k]$ and $[w_k\cdots w_l]$ are also primitive. Thus, Lemma \ref{Lem: HT assoc with a primitive subsnake name is NE} implies that the H\"older triangles $T(\delta_j,\delta_{k})$ and $T(\delta_k,\delta_l)$ are normally embedded. As $w_j=w_l$ we have $tord(\delta_j,\delta_l)=\alpha>\beta = itord(\delta_j,\delta_l)$. Hence, $\delta_{k}$ is abnormal.
	
	Now, consider an arc $\gamma\in G(T_W)$. Let $\gamma \subset T_{k-1}$ and assume that $k<m$. As $1<k<m$, we have $\delta_k$ abnormal. Let $\delta_j$ and $\delta_l$ be arcs such that the H\"older triangles $T(\delta_j,\delta_{k})$ and $T(\delta_k,\delta_l)$ are normally embedded and $tord(\delta_j,\delta_l)=\alpha>\beta = itord(\delta_j,\delta_l)$. If $k-1>1$ then, as $[w_{j}\cdots w_k]$ and $[w_{k-1}\cdots w_l]$ are also primitive words, Lemma \ref{Lem: HT assoc with a primitive subsnake name is NE} implies that $T(\delta_j,\gamma)$ and $T(\gamma,\delta_l)$ are normally embedded, since $T(\delta_j,\gamma)\subset T(\delta_j,\delta_k)$ and $T(\gamma,\delta_l)\subset T(\delta_{k-1},\delta_l)$. Thus, $\gamma$ is abnormal. If $k-1=1$ then $j=1$. Hence, as $\mu(T(\delta_1,\delta_k))=\beta$, Lemma \ref{Lem: maximal abnormal zone MAIN PART} implies that $\delta_k$ is contained in $\beta$-snake where $\delta_1$ is a boundary arc. So, as $itord(\gamma,\delta_1)=\beta$, by Remark \ref{Rem: abnormal arcs of a snake}, $\gamma$ is abnormal.
	
	If $k=m$ the argument to prove that $\gamma$ is abnormal is similar (regarding $\delta_{k-1}$ instead of $\delta_k$) and will be omitted.
\end{proof}

\begin{Teo}\label{Teo: snake name realization}
	Given a snake name $W$, there exists a snake $T$ such that $W=W(T)$ (see Definition \ref{Def: word associated with a snake}).
\end{Teo}
\begin{proof}
	Let $W=[w_1\ldots w_m]$ be a snake name with $n$ distinct letters. If $m=2$ then $W$ is the word associated with a bubble snake. Thus, assume that $m>2$. Let $T=T_W$ be the $\beta$-H\"older triangle associated with $W$ (see Definition \ref{Def: arcs and HT assoc with a snake name}). We claim that $T$ is a $\beta$-snake such that $W=W(T)$.
	
	Corollary \ref{Cor: HT assoc with a snake name is non-singular} implies that $T$ is a non-singular $\beta$-H\"older triangle. So, to show that $T$ is a $\beta$-snake it remains to prove that $G(T)=Abn(T)$. The inclusion $Abn(T)\subset G(T)$ is obvious, and the inverse inclusion is given by Corollary \ref{Cor: interior deltas of T are abnormal}.

	Finally, as the link of $T$ is oriented from $\delta_1$ to $\delta_m$ (see Definition \ref{Def: arcs and HT assoc with a snake name}), the $i$-th nodal zone of $T$ is $N_i=\{\gamma \in V(T) : itord(\gamma,\delta_i)>\beta\}$, for $i=1,\ldots,m$, and the $k$-th node of $T$ is $\N_k=\bigcup_{i  \in I_k}N_i$ for each $k=1,\ldots,n$ (here $I_k$ is as in Definition \ref{Def: partition assoc with a snake name}).
In particular, $T$ is a $\beta$-snake with $m$ nodal zones and $n$ nodes, such that $W=W(T)$.
\end{proof}

\begin{remark}\label{Rem: case n=1 for snake assoc with a snake name}
	\normalfont If we would consider $T_W$ as in Definition \ref{Def: arcs and HT assoc with a snake name} for $m=2$, so that $n=1$, we would obtain a H\"older triangle outer bi-Lipschitz equivalent to the H\"older triangle $T$ in Example \ref{Exam: Lipschitz singular arc}. Then $T_W$ would contain a Lipschitz singular arc and would not be a snake.
\end{remark}

\begin{remark}\label{Rem: comments about the snake TW}
	\normalfont The triangle $T_W$ in Definition \ref{Def: arcs and HT assoc with a snake name} is the simplest kind of a $\beta$-snake associated with the snake name $W$.
All segments of $T_W$ have multiplicity one, and the spectrum of each of its nodes consists of a single exponent $\alpha$.
Moreover, if we consider a pancake decomposition $\{X_k\}$ of $T$ defined in Proposition \ref{Prop:pancake decom from nodal zones},
then a minimal pizza on any pancake $X_k$, for the distance function from $X_k$ to any other pancake, has at most two pizza slices $T_i$, such that either $Q_i=\{\beta\}$ is a point and $\mu_i=\beta$ or $Q_i=[\beta,\alpha]$ and $\mu_i(q)=q$ for all $q\in Q_i$.
Note that construction in Definition \ref{Def: arcs and HT assoc with a snake name} can be slightly modified to obtain a snake with the given snake name $W$ and prescribed cluster partitions of the sets $\mathcal{S}(\N,\N')$ of its segments
(see Remark \ref{Rmk:clusters} for conditions satisfied by such partitions).
\end{remark}

\begin{remark}\label{REM: snake name ignores}
	\normalfont The snake name ignores many geometric properties of a snake, such as pizza decompositions for the distance functions on pancakes associated with its segments, and the spectra of its nodes.
\end{remark}

\subsection{Weakly bi-Lipschitz maps and weak Lipschitz equivalence}\label{Subsection: weak equivalence}
In this Subsection we consider combinatorial and geometric significance of the cluster partitions of the sets $\mathcal{S}(\N,\N')$ in Definition \ref{Def: clusters}.

\begin{Def}\label{Def:weak equivalence}
\normalfont Let $h:X\to X'$ be a homeomorphism of two $\beta$-H\"older triangles $X$ and $X'$, bi-Lipschitz with respect to the inner metrics of $X$ and $X'$. We say that $h$ is \textit{weakly outer bi-Lipschitz} when $tord(h(\gamma),h(\gamma'))>\beta$ for any two arcs $\gamma$ and $\gamma'$ of $X$ if, and only if, $tord(\gamma,\gamma')>\beta$.
If such a homeomorphism exists, we say that $X$ and $X'$ are \textit{weakly outer Lipschitz equivalent}.
\end{Def}

\begin{Teo}\label{Teo:weak equivalence}
Two $\beta$-snakes $X$ and $X'$ are weakly outer Lipschitz equivalent if, and only if, they can be oriented so that
\begin{enumerate}
\item[(i)] Their snake names are equivalent, the nodes $\N_1,\ldots,\N_n$ of $X$ are in one-to-one correspondence with the nodes
$\N'_1,\ldots,\N'_n$ of $X'$, and the nodal zones $N_1,\ldots,N_m$ of $X$ are in one-to-one correspondence with the nodal zones $N'_1,\ldots,N'_m$ of $X'$;
\item[(ii)] For any two nodes $\N_j$ and $\N_k$ of $X$, and the corresponding nodes $\N'_j$ and $\N'_k$ of $X'$,
each cluster of the cluster partition of the set $\mathcal{S}(\N'_j,\N'_k)$ (see Definition \ref{Def: clusters}) consists
of the segments of $X'$ corresponding to the segments of $X$ contained in a cluster of the cluster partition of the set $\mathcal{S}(\N_j,\N_k)$.
\end{enumerate}
\end{Teo}

\begin{proof}
It follows from the definition \ref{Def:weak equivalence} that a weakly outer bi-Lipschitz homeomorphism $h:X\to X'$ defines equivalence of the snake names $W=W(X)$ and $W'=W(X')$, and identifies cluster partitions of the sets $\mathcal{S}(N_j,N_k)$ and $\mathcal{S}(N'_j,N'_k)$ for any $j$ and $k$.
Thus we have to prove that conditions (i) and (ii) of Theorem \ref{Teo:weak equivalence} imply weak outer Lipschitz equivalence of the snakes $X$ and $X'$.

Let us assume first that $X$ and $X'$ are not bubble or spiral snakes, so any segment of each of them has two adjacent nodal zones
in two distinct nodes.
Since the snake names $W$ and $W'$ are equivalent, each nodal zone $N_j$ of $X$ corresponds to
the $j$-th entry $w_j$ of $W$ and each nodal zone $N'_j$ of $X'$ corresponds to the $j$-th entry $w'_j$ of $X'$.
Also, nodal zones $N_j$ and $N_k$ of $X$ (resp., $N'_j$ and $N'_k$ of $X'$) belong to the same node if, and only if, $w_j=w_k$ (resp., $w'_j=w'_k$).
Selecting an arc $\gamma_j$ in each nodal zone $N_j$ of $X$ (a boundary arc if $N_j$ is a boundary nodal zone of $X$) and an arc $\gamma'_j$ in each nodal zone $N'_j$ of $X'$ (a boundary arc if $N'_j$ is a boundary nodal zone of $X'$) we obtain, according to Proposition
\ref{Prop:pancake decom from nodal zones}, pancake decompositions of $X$ and $X'$, such that each pancake $X_j=T(\gamma_j,\gamma_{j+1})$ of $X$ (resp., pancake $X'_j=T(\gamma'_j,\gamma'_{j+1})$ of $X'$) is a $\beta$-H\"older triangle corresponding to a segment of $X$ with adjacent nodal zones $N_j$ and $N_{j+1}$ (resp., to a segment of $X'$ with adjacent nodal zones $N'_j$ and $N'_{j+1}$).

We construct a weakly outer bi-Lipschitz homeomorphism $h:X\to X'$ as follows.

First, we define $h$ on each arc $\gamma_j$ as the map $\gamma_j\to\gamma'_j$ consistent with the parameterisations of both arcs by the distance to the origin. Next, for each nodes $\N$ and $\N'$ of $X$, if the set $\mathcal{S}=\mathcal{S}(\N,\N')$ is not empty, we choose one pancake $X_j=T(\gamma_j,\gamma_{j+1})$ corresponding to a segment from each cluster of the cluster partition of $\mathcal{S}$, and define a bi-Lipschitz homeomorphism $h_j:X_j\to X'_j$ consistent with the previously defined mappings for the arcs $\gamma_j$ and $\gamma_{j+1}$.
Finally, for any cluster of $\mathcal{S}$ containing a segment with the homeomorphism $h=h_j$ defined on the corresponding pancake $X_j$,
if there is another segment in that cluster, we define $h$ on the pancake $X_k$ corresponding to that segment as follows.
Since pancakes $X_j$ and $X_k$ correspond to segments in the same cluster, pancakes $X'_j$ and $X'_k$ also correspond to segments in the same cluster.
It follows from Proposition \ref{Prop:two triangles} that there
is a bi-Lipschitz homeomorphism $h_{kj}:X_k\to X_j$ such that $tord(\gamma,h_{kj}(\gamma))>\beta$ for each arc $\gamma\subset X_k$,
and a bi-Lipschitz homeomorphism $h'_{jk}:X'_j\to X'_k$ such that $tord(\gamma',h'_{jk}(\gamma'))>\beta$ for each arc $\gamma'\subset X'_j$.
Then $h:X_k\to X'_k$ is defined as the composition of $h_{kj},\;h_j$ and $h'_{jk}$.
This defines an outer bi-Lipschitz homeomorphism $h:X\to X'$.

If $X$ and $X'$ are either bubble snakes or spiral snakes, so their segments are not normally embedded, the above construction
should be slightly modified by adding extra arcs $\lambda_j$ in each segment of $X$ and $\lambda'_j$ in each segment of $X'$ so
that $tord(\lambda_j,\lambda_k)>\beta$ and $tord(\lambda'_j,\lambda'_k)>\beta$ for all $j$ and $k$.
\end{proof}

\begin{remark}\label{Rem:combinatorial cluster partitions} \normalfont
The sets of segments $\mathcal{S}(\N,\N')$ in Definition \ref{Def: clusters}
can be recovered from the snake name $W=W(X)$ of a snake $X$ as follows. Let $\N$ and $\N'$ be two nodes of $X$ associated with the letters $x$ and $x'$ of $W$.
Then the set $\mathcal{S}(\N,\N')$ can be identified with the set $\mathcal{S}(x,x')$ of pairs of consecutive entries $(w_j,w_{j+1})$ of $W$ such that either $(w_j,w_{j+1})=(x,x')$ or $(w_j,w_{j+1})=(x',x)$.
Accordingly, a cluster partition of the set $\mathcal{S}(\N,\N')$ in Definition \ref{Def: clusters} can be identified with a partition of $\mathcal{S}(x,x')$. Remark \ref{Rmk:clusters} implies that, if $X$ is a spiral snake, then $\N=\N'$ and the partition of $\mathcal{S}(\N,\N)$ consists of a single cluster. Also, if $w_{j-1}=w_{j+1}$ in $W(X)$ then the pairs $(w_{j-1},w_j)$ and $(w_j,w_{j+1})$ cannot belong to the same cluster of partition.
\end{remark}

\subsection{Binary snakes and their names}\label{Subsec: Binary snakes and their names}

In this subsection we consider binary snakes (see Definition \ref{Def: binary snake}). They play important role in the combinatorial classification of snakes since any snake name can be reduced to a binary one (see Definition \ref{Def: reducing snakes to binary snakes}).

\begin{Def}\label{Def: binary snake}
	\normalfont A \textit{binary snake name} is a snake name $W$ which is also a binary word (see Definition \ref{Def: primitive and binary words}). A snake $T$ is \textit{binary} if $W(T)$ is a binary snake name. Alternatively, a snake $T$ is binary if each of its nodes contains exactly two nodal zones.
\end{Def}

\begin{Def}\label{Def: reducing snakes to binary snakes}
	\normalfont Let $W$ be a snake name and $x$ a letter of $W$. If $x$ appears $p>2$ times in $W$ and $W=X_0 x X_1 x \cdots x X_{p-1} x X_p$, we replace $x$ by $p-1$ distinct new letters $x_1,\ldots,x_{p-1}$, and define the \textit{binary reduction} of $W$ with respect to $x$ as the word
	\begin{equation}\label{equa: binary reduction}
	W_x=X_0 x_1 X_1 x_1 x_2X_2x_2x_3 \cdots x_{p-2} x_{p-1} X_{p-1} x_{p-1} X_p.
	\end{equation}
	Note that the first and last entries of $x$ are replaced by a single letter each, while every other entry of $x$ is replaced by two letters.
\end{Def}

\begin{Prop}\label{Prop: reduction to binary snakes if well-defined}
	The word $W_x$ in Definition \ref{Def: reducing snakes to binary snakes} is a snake name.
\end{Prop}
\begin{proof}
	Note that $W_x$ satisfies condition (i) of Definition \ref{Def: rules for snakes names} because each of the new letters $x_i$ in $W_x$ appears exactly twice, and each other letter appears at least twice, since $W$ is a snake name. It remains to prove that $W_x$ satisfies condition (ii) of Definition \ref{Def: rules for snakes names}.
	
	Let $w$ be an entry of $W$ other than $x$ such that there is a semi-primitive subword $[w_j\cdots w_l]$ of $W$ containing $w$, where $w_j=w_l\ne w$. If $w_j=w_l=x$ then $w$ belongs to one of the subwords $X_k$ of $W$ and $x_k X_k x_k$ is a semi-primitive subword of $W_x$ containing $w$. Otherwise $[w_j\cdots w_l]$ contains at most one entry of $x$, and replacing that entry with one or two new letters results in a semi-primitive subword of $W_x$ containing $w$.
	
	If $w=x$ then $[w_j\cdots w_l]$ does not contain other entries of $x$, and replacing $x$ with one or two new letters results
	in a semi-primitive subword of $W_x$ containing the new entries.
\end{proof}

\begin{remark}
	\normalfont The binary reduction could be geometrically interpreted as splitting a node with more than two nodal zones (see Fig.~\ref{fig:snake to generic binary snake}).
\end{remark}

\begin{figure}
	\centering
	\includegraphics[width=4.8in]{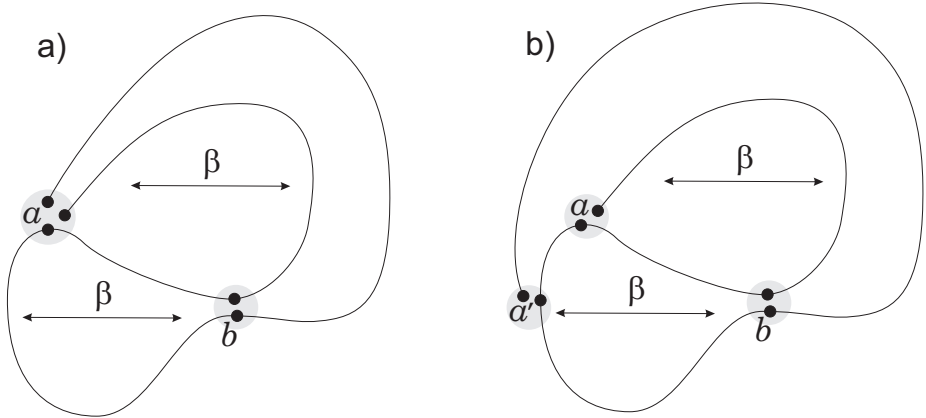}
	\caption{Reducing a non-binary snake (a) to a binary snake (b)}\label{fig:snake to generic binary snake}
\end{figure}

\begin{remark}
	\normalfont Any non-binary snake name $W$ could be reduced to a binary snake name by applying binary reduction to each letter that appears in $W$ more than twice.
If there are several such letters, the resulting binary snake name does not depend on the order of the letters to which binary reduction is applied.
\end{remark}

\subsection{Recursion for the number of binary snake names}\label{Subsec: Recursion for the number of binary snakes}

\begin{Prop}\label{Prop: aXa is semi-primitive}
	If $W=aXaZ$ is a binary snake name then $aXa$ is semi-primitive.
\end{Prop}
\begin{proof}
	Since $W$ is a snake name, there is a semi-primitive subword $[w_j\cdots w_l]$ of $W$ such that $j<2<l$. Thus $j=1$ and $w_j=w_l=a$. As $W$ is binary, $aXa$ is the only option for such a subword.
\end{proof}

\begin{Def}\label{Def: left and right appearances of a letter}
	\normalfont Given a word $W$ and a letter $x$ of $W$ that appears exactly twice, we write $x^{-}$ and $x^{+}$ to denote the first and second entries of $x$ in $W$, respectively. If $W=X_1x^-X_2x^+X_3$ we write $W-\{x\}$ to denote the word $X_1X_2X_3$ representing \textit{deletion} of the letter $x$ from $W$.
\end{Def}

\begin{Lem}\label{Lem: deleting first letter}
	Let $W = abZ$ be a binary snake name, and $W' = W - \{a\}$. Then, $W'$ is a snake name if and only if $[b^-\cdots b^+]$ is a semi-primitive subword of $W$.
\end{Lem}
\begin{proof} Given a letter $x$ of $W$, let $W'(x)$ be the subword of $W'$ obtained by deleting the letter $a$ from the subword $[x^-\cdots x^+]$ of $W$.

If $W'$ is a snake name then Proposition \ref{Prop: aXa is semi-primitive} applied to $W'$ (note that $b$ is the first letter of $W'$) implies that $W'(b)$ is a semi-primitive subword of $W'$. As $W'(b)$ does not contain $a$, the subword $[b^-\cdots b^+]$ of $W$,
obtained by inserting the second entry of $a$ into $W'(b)$, is also semi-primitive.
	
Conversely, suppose that $[b^-\cdots b^+]$ is a semi-primitive subword of $W$. Since $W'$ is a binary word, it satisfies condition (i) of Definition \ref{Def: rules for snakes names}, and we have only to check that condition (ii) is satisfied.
Since $[b^-\cdots b^+]$ is a semi-primitive subword of $W$, the subword $W'(b)$ of $W'$ is also semi-primitive.
Thus any entry $w\ne b$ of $W'$ contained in $W'(b)$ satisfies condition (ii) of Definition \ref{Def: rules for snakes names}.
 Let $w$ be an entry of $W'$, other than the last one, not contained in the subword $W'(a)$. Since $W$ is a snake name, there exists a semi-primitive subword $[x^-\cdots x^+]$ of $W$ containing the corresponding entry $w$ of $W$.
 Then $w\ne x\ne a$, and $W'(x)$ is a semi-primitive subword of $W'$ containing $w$.
 Since any entry of $W'$ either belongs to $W'(b)$ or does not belong to $W'(a)$, this implies that all entries of $W'$, except the first and last ones, satisfy condition (ii) of Definition \ref{Def: rules for snakes names}. Thus $W'$ is a snake name.
\end{proof}

\begin{Lem}\label{Lem: when a letter x could be removed}
	Let $W$ be a snake name where a letter $x$ appears exactly twice.
	If the subword $[x^-\cdots x^+]$ of $W$ is not semi-primitive then $W-\{x\}$ is a snake name.
\end{Lem}
\begin{proof}
	Since $W$ is a snake name, $W-\{x\}$ satisfies condition (i) of Definition \ref{Def: rules for snakes names}. Let $w\ne x$ be an entry of $W$ such that there is a semi-primitive subword $[w_j\cdots w_l]$ of $W$ containing $w$, with $w_j=w_l\ne w$. Since $[x^-\cdots x^+]$ is not semi-primitive, $w_j\ne x$, and deleting $x$ from $W$ results in a semi-primitive subword $[w_j\cdots w_l]$ of $W-\{x\}$ containing $w$. This implies that all entries of $W-\{x\}$, except the first and last ones, satisfy condition (ii) of Definition \ref{Def: rules for snakes names}. Then $W-\{x\}$ is a snake name.
\end{proof}

\begin{Prop}\label{Prop: if not the first, the second letter can be removed}
	Let $W=abZ$ be a binary snake name. If $W - \{a\}$ is not a snake name then $W-\{b\}$ is a snake name.
\end{Prop}
\begin{proof}
	If $W - \{a\}$ is not a snake name then, by Lemma \ref{Lem: deleting first letter}, $[b^-\cdots b^+]$ is not semi-primitive. Then, Lemma \ref{Lem: when a letter x could be removed} implies that $W-\{b\}$ is a snake name.
\end{proof}

\begin{remark}
	\normalfont We can (similarly) prove a symmetric version of Proposition \ref{Prop: if not the first, the second letter can be removed}, i.e., if $W=Xyz$ is a binary snake name and $W-\{z\}$ is not a snake name then $W-\{y\}$ is a snake name.
\end{remark}

\begin{Def}\label{Def: parameters of a binary snake name}
	\normalfont Given a binary snake name $W=aXaZ$ of length $2m>2$, we define its \textit{parameters} as the numbers $j$ and $k$ where $j$ is the position of $a^{+}$ and $w_k$ is the first entry of $W$ such that $[w_2\cdots w_k]$ is not a primitive subword. For $m>1$, we define $\W_m(j,k)$ as the set of all binary snake names of length $2m$ with parameters $j$ and $k$.
\end{Def}

\begin{remark}\label{Rem: inequality for the parameters of a BSN}
	\normalfont Note that parameter $k$ is not defined for the bubble snake name $[aa]$. The word $[abab]\in\W_2(3,4)$ is the only binary snake name of length $4$. For $m\ge 3$, the set $\W_m(j,k)$ is nonempty only when $3\le j < k$ and $5\le k \le m +2$. In particular, $[abacbc]\in\W_3(3,5)$ and $[abcabc]\in\W_3(4,5)$ are the only binary snake names of length $6$.
\end{remark}

\begin{Def}\label{Def: new snake names of types A and B}
	\normalfont
	Given a binary snake name $W=[w_1\cdots w_{2m}]\in\W_m(j,k)$, we can obtain new binary words of length $2m+2$ inserting a new letter at two positions in $W$ as follows:
	
	\begin{enumerate}
		\item[$(A)$] For $l=2,\ldots,j$, insert the first copy of a new letter $a$ to $W$ in front of $w_1$, and a second copy between $w_{l-1}$ and $w_l$.
		\item[$(B)$] For $l=k+1,\ldots 2m$, insert the first copy of a new letter $b$ to $W$ between $w_1$ and $w_2$, and a second copy between $w_{l-1}$ and $w_l$.
	\end{enumerate}
\end{Def}
	
\begin{Exam}
	\normalfont The binary snake names $[abacbc]\in\W_3(3,5)$ and $[abcabc]\in\W_3(4,5)$ can be obtained from the binary snake name $[abab]\in\W_2(3,4)$ by applying operation $(A)$ with $l=2$ and $l=3$, respectively, and renaming the letters. Applying operations $(A)$ with $l=2,3$ and $(B)$ with $l=6$ to $W=[abacbc]$ we obtain, renaming the letters, the words $[abacbdcd]\in\W_4(3,5)$, $[abcabdcd]\in\W_4(4,5)$ and $[abcadcbd]\in\W_4(4,6)$.
Applying operations $(A)$ with $l=2,3,4$ and $(B)$ with $l=6$ to $W=[abcabc]$ we obtain, renaming the letters, the words $[abacdbcd]\in\W_4(3,6)$, $[abcadbcd]\in\W_4(4,6)$, $[abcdabcd]\in\W_4(5,6)$ and $[abcdacbd]\in\W_4(5,6)$.
Note that all these words are binary snake names, and that all $7$ binary snake names of length $8$ are thus obtained
(see Propositions \ref{new words are snake names} and \ref{Prop: uniqueness from operations A and B} and Theorem \ref{Teo: recursion formula for binary snakes} below).
\end{Exam}

\begin{Prop}\label{new words are snake names}
	If $W=[w_1\cdots w_{2m}]\in\W_m(j,k)$ is a binary snake name then the words obtained from $W$ by applying any operations $(A)$ and $(B)$ in Definition \ref{Def: new snake names of types A and B} are also binary snake names.
\end{Prop}

\begin{proof}
	Let $W_A$ be the word obtained by applying operation $(A)$ in Definition \ref{Def: new snake names of types A and B} to $W$ for some $l\in\{2,\ldots,j\}$. Since $W_A$ is a binary word, condition $(i)$ of Definition \ref{Def: rules for snakes names} is satisfied.
As the first entry $a^-$ of the letter $a$ is the first letter of $W_A$, we have to check condition $(ii)$ of Definition
\ref{Def: rules for snakes names} for the second entry $a^+$ of $a$, and for any entry $w\ne a$ of $W_A$ other than its last entry.
Since $W\in\W_m(j,k)$, we have $w_1=w_j$ and $[w_1\cdots w_j]$ is a semi-primitive subword of $W$, by Proposition \ref{Prop: aXa is semi-primitive}.
Since $l\le j$, the corresponding subword $[w_1\cdots a^+\cdots w_j]$ of $W_A$ is also semi-primitive.
Since $W$ is a snake name, any entry $w\ne a$ of $W_A$, other than its last entry, corresponds to an entry of $W$ contained in some semi-primitive subword $[w_p\cdots w\cdots w_q]$ of $W$, where $w_p=w_q\ne w$. The corresponding subword $[w_p\cdots w\cdots w_q]$ of $W_A$ is also semi-primitive (it is either the same as in $W$ or contains one extra entry $a^+$).
Thus condition $(ii)$ of Definition \ref{Def: rules for snakes names} is satisfied for any entry $w\ne a$ of $W_A$. Then $W_A$ is a snake name.

Let now $W_B$ be the word obtained by applying operation $(B)$ in Definition \ref{Def: new snake names of types A and B} to $W$ for some $l\in\{k+1,\ldots,2m\}$. Since $W_B$ is a binary word, condition $(i)$ of Definition \ref{Def: rules for snakes names} is satisfied.
The first entry $b^-$ of the letter $b$ is contained in the semi-primitive subword $[w_1 b^-\cdots w_j]$ of $W_B$, and its second entry
$b^+$, inserted between the entries $w_{l-1}$ and $w_l$ of $W$, belongs to the semi-primitive subword of $W_B$ corresponding to a
semi-primitive subword $[w_p\cdots w_{l-1}\cdots w_q]$ of $W$ containing $w_{l-1}$. Note that, as $l>k>j$, we have $w_p=w_q\ne w_1$, thus the subword $[w_p\cdots b^+\cdots w_q]$ of $W_B$ cannot contain $b^-$ and remains semi-primitive. The same argument as for $W_A$ shows that condition $(ii)$ of Definition \ref{Def: rules for snakes names} is satisfied for any entry $w\ne b$ of $W_B$. Then $W_B$ is a snake name.
\end{proof}

\begin{remark}\label{Rem: uniqueness of operations A and B}
	\normalfont
Note that a word $W_B$, obtained by applying operation $(B)$ in Definition \ref{Def: new snake names of types A and B} to a binary snake name $W$, would be a binary snake name even if $l>j$ instead of $l>k$ was allowed.
However, condition $l>k$ in Definition \ref{Def: new snake names of types A and B} implies that the subword $[b^-\cdots b^+]$ of $W_B$ is not semi-primitive, thus $W_B$ cannot be obtained applying the operation $(A)$ to any binary snake name.
Similarly, the word $W_A$ cannot be obtained applying the operation $(B)$ to any binary snake name.
\end{remark}

\begin{remark}\label{Rem: deletion versus operations A and B}
	\normalfont If $W_A$ (resp., $W_B$) is obtained from a binary snake name $W$ by applying operation $(A)$ (resp., $(B)\,$) then the first (resp., second) letter of $W_A$ (resp., $W_B$) can be deleted, resulting in the original word $W$. Note that ``deletion'' operations are unique, while ``insertion'' operations are not.
\end{remark}

\begin{Prop}\label{Prop: uniqueness from operations A and B}
	Any binary snake name of length $2m+2$ could be obtained from a binary snake name of length $2m$ by applying exactly one of the operations $(A)$ and $(B)$ as in Definition \ref{Def: new snake names of types A and B}.
\end{Prop}
\begin{proof}
	Let $W=abZ$ be a binary snake name of length $2m+2$. If $W-\{a\}$ is a snake name then $W$ can be obtained from $W-\{a\}$ by applying operation $(A)$ to add back the deleted letter $a$. If $W-\{a\}$ is not a snake name then, by Proposition \ref{Prop: if not the first, the second letter can be removed}, $W-\{b\}$ is a snake name and, similarly, $W$ can be obtained from $W-\{b\}$ by applying operation $(B)$.
	
	Finally, if $W$ was obtained from a word of length $2m$ by applying operation $(A)$ (resp., $(B)\,$) then $W$ cannot be obtained from any word of length $2m$ by applying operation $(B)$ (resp., $(A)\,$) (see Remark \ref{Rem: uniqueness of operations A and B}).
\end{proof}

\begin{Teo}\label{Teo: recursion formula for binary snakes}
Let $M_m$ be the number of all binary snake names of length $2m$, and let $M_{m}(j,k)=|\W_m(j,k)|$ be the number of binary snake names of length $2m>2$ with parameters $j$ and $k$ (see Definition \ref{Def: parameters of a binary snake name}).
Then $M_1=1,\;M_2=M_2(3,4)=1$ and, for $m\ge 2$,
\begin{equation}\label{add up}
M_{m+1}(j,k) = M_{m, A}(j,k) + M_{m,B}(j,k),
\end{equation}
where
\begin{equation}\label{MA}
M_{m,A}(j,k) = \sum_{l=k-1}^{m+2}M_{m}(k-2,l)
\end{equation}
and
\begin{equation}\label{MB}
M_{m,B}(j,k) = (2m-k+1)M_{m}(j-1,k-1).
\end{equation}
Consequently,
\begin{equation}\label{M}
M_{m+1} = \sum_{3\le j < k,\;5\le k \le m+3}M_{m+1}(j,k).
\end{equation}
\end{Teo}
\begin{proof} Since the bubble snake name $[aa]$ is the only binary snake name of length $2$, and the word $[abab]$ is the only binary snake name of length 4, we have $M_1=1,\;M_2=M_2(3,4)=1$.
	For $m\ge 2$, Proposition \ref{Prop: uniqueness from operations A and B} implies that it is enough to count separately the binary snake names of length $2m+2$ obtained by applying operations $(A)$ and $(B)$ from the binary snake names of length $2m$.
	
	Note that $M_{m,A}(j,k)$ denotes the number of binary snake names of length $2m+2$ with parameters $j$ and $k$ obtained from binary snake names of length $2m$ by applying operation $(A)$. Each such binary snake name $W'$ of length $2m$ must have parameters $j'=k-2$ and $k' \in \{k-1, \ldots, m+2\}$. This implies (\ref{MA}).
	
	Similarly, $M_{m,B}$ denotes the number of binary snake names of length $2m+2$ with parameters $j$ and $k$ obtained from binary snake names of length $2m$ by applying operation $(B)$. Each such snake name $W'$ of length $2m$ must have parameters $j'=j-1$ and $k'=k-1$. For each of them we have $2m - k' = 2m - (k-1) = 2m - k +1$ possibilities to place the second entry of the new letter. This implies (\ref{MB}).
	
	Adding up these two numbers, we obtain the formula (\ref{add up}).
Remark \ref{Rem: inequality for the parameters of a BSN} implies (\ref{M}).
\end{proof}

\subsection{Binary snake names and standard Young tableaux}
In this subsection we assign a standard Young tableau (SYT) of shape $(m-1,\,m-1)$ to a binary snake name of length $2m$.

\begin{Def}\label{Def: Young Diagram}
	\normalfont A \textit{Young diagram}, or \textit{shape}, $\lambda=(\lambda_1,\lambda_2,\ldots)$ of size $n$, where $\lambda_1\ge\lambda_2\ge\ldots\ge 0$ and $\lambda_1+\lambda_2+\ldots=n$ (see, e.g., \cite{fulton1997young} pp.~1-2) is a collection of cells arranged in left-justified rows of lengths $\lambda_j$.
A \textit{filling} of $\lambda$ means placing positive integers in each of its cells.
A \textit{standard Young tableau} (SYT) of shape $\lambda$ is a filling of $\lambda$ with the numbers from $1$ to $n$, each of them occurring exactly once, so that the numbers in each row and each column of $\lambda$ are strictly increasing.
\end{Def}

\begin{Def}\label{Def: Young Tableau associated with a binary snake}
	\normalfont Let $W=[w_1\cdots w_{2m}]$ be a binary snake name. We assign to $W$ the following filling $\T(W)$ of shape $\lambda=(m-1,\,m-1)$: for $i=2,\ldots, 2m-1$, we place the number $i-1$ into the first empty cell of the first row of $\lambda$ if $w_i\ne w_j$ for all $j<i$, and into the first empty cell of the second row of $\lambda$ otherwise. Alternatively, $i-1$ is inserted into the first row of $\lambda$ if $w_i$ is a node entry of $W$, and into the second row otherwise.
\end{Def}

\begin{Prop}\label{Prop: Young Tableau associated with a binary snake}
	The filling $\T(W)$ assigned to a binary snake name $W=[w_1\cdots w_{2m}]$ in Definition \ref{Def: Young Tableau associated with a binary snake} is a standard Young tableau.
\end{Prop}

\begin{proof}
	By Definition \ref{Def: Young Tableau associated with a binary snake}, each number $1,\ldots,2m-2$ appears in $\T(W)$ exactly once, and the numbers in each row are strictly increasing. To check that the numbers are increasing in columns, suppose that $W\in\W_{j,k}$ for some $j$ and $k$, and that the $\ell$-th cell of the second row of $\T(W)$ contains the number $i-1$.
This means that $w_i$ is the second entry of some letter of $W$, and that exactly $\ell$ distinct letters appear twice in the subword $[w_1\cdots w_i]$ of $W$.
Note that at least one letter of $W$ must appear only once in the subword $[w_1\cdots w_i]$
(Proposition \ref{Prop: aXa is semi-primitive} implies that $j\le i$, thus the first letter of $W$ appears twice in $[w_1\cdots w_i]$).
Otherwise $i=2\ell$ would be even, $i+1<2m$, and there will be no semi-primitive subword $[x^-\cdots w_{i+1}\cdots x^+]$ of $W$
containing $w_{i+1}$, in contradiction to $W$ being a snake name.
This implies that the subword $[w_2\cdots w_{i-1}]$ contains at least $\ell$ node entries of $W$.
Thus the number in the $\ell$-th cell of the first row of $\T(W)$ is strictly less than $i-1$.
This completes the proof.
\end{proof}

\begin{remark}
	\normalfont Note that Proposition \ref{Prop: Young Tableau associated with a binary snake} does not necessarily hold for binary words which are not snake names. For example, it is not true for the binary words $W=[aabb]$ and $W=[ababcdcd]$.
\end{remark}

\begin{remark}
	\normalfont The empty SYT of shape $(0,0)$ is assigned to the bubble snake

\smallskip
\noindent name $[aa]$, and the single SYT $\scriptstyle\begin{bmatrix} 1\\2\end{bmatrix}$ of shape $(1,1)$ is assigned to the binary snake

\smallskip
\noindent  name $[abab]$. Two SYTs $\scriptstyle\begin{bmatrix} 1&3\\2&4\end{bmatrix}$ and $\scriptstyle\begin{bmatrix} 1&2\\3&4\end{bmatrix}$ of shape $(2,2)$ are assigned to the

\smallskip
\noindent binary snake names $[abacbc]$ and $[abcabc]$, respectively. Consider next the SYT

\smallskip
\noindent $\lambda=\scriptstyle\begin{bmatrix} 1&2&4\\3&5&6\end{bmatrix}$ of shape $(3,3)$. The words $W=[abcadbcd]$ and $W'=[abcadcbd]$

\smallskip
\noindent are distinct binary snake names such that $\T(W)=\T(W')=\lambda$.
Thus the same SYT may be assigned to several binary snake names.
\end{remark}

\begin{Def}\label{Def: binary snake associated with SYT}
	\normalfont Let $\T$ be a standard Young tableau of shape $(m-1,m-1)$.
We define a binary word $W=W(\T)=[w_1\cdots w_{2m}]$ with $m$ distinct letters $x_1,\ldots , x_m$ as follows.
If $m=1$ and $\T$ is empty then $W(\T)=[x_1 x_1]$.
If $m>1$, we set $w_1=x_1$, $w_{2m}=x_m$ and, for $1<i<2m$,
	$w_i=x_{k+1}$ (resp., $w_i=x_k$) if the $k$-th cell of the first row (resp., second row) of $\T$ contains the number $i-1$.
\end{Def}

\begin{remark}\label{Remark: binary snake associated with SYT}\normalfont
If $1<i<2m$ and the $k$-th cell of the first row of $\T$ contains $i-1$, it follows from
Definition \ref{Def: binary snake associated with SYT} that the subword $[w_1\cdots w_i]$
of $W(\T)$ contains exactly $k+1$ first entries of the letters $x_1,\ldots,x_{k+1}$, with $w_i=x_{k+1}$,
and at most $k-1$ second entries of letters $x_j$ for some $j<k$.
In particular, there are at least two more first entries than second entries of the letters in $[w_1\cdots w_i]$.

If the $k$-th cell of the second row of $\T$ contains $i-1$, then the subword $[w_1\cdots w_i]$ of $W(\T)$ contains
$\ell\ge k+1$ first entries of the letters $x_1,\ldots,x_\ell$ and exactly $k$ second entries of the letters $x_1,\ldots,x_k$, with $w_i=x_k$.
In particular, there are more first entries than second entries of the letters in $[w_1\cdots w_i]$.

This implies that the first entries of all letters $x_j$ appear in $W(\T)$ in increasing order of their indices $j$.
Similarly, the second entries of all letters $x_j$ appear in $W(\T)$ in increasing order of their indices $j$.
\end{remark}

\begin{Def}\label{Def: inversion on binary snake name}
	\normalfont An \textit{inversion} in a binary word $W$ is a pair of distinct letters $x$ and $y$ contained in $W$ such that the subword $[x^-\cdots x^+]$ of $W$ contains both entries of $y$. We say that a binary word $W$ is \textit{inversion free} if it has no inversions.
\end{Def}

\begin{Lem}\label{Lem: word associated with a SYT is inversion free}
	If $\T$ is a standard Young tableau of shape $(m-1,m-1)$ then $W(\T)$ in Definition \ref{Def: binary snake associated with SYT} is an inversion free binary word.
\end{Lem}
\begin{proof}
For $m=1$ the statement is true since $\T$ is empty and $W(\T)=[x_1 x_1]$, thus we may assume that $m>1$.

Let us show first that $W(\T)$ is binary.
The letter $x_1$ is the first letter of $W(\T)$, and
$w_i=x_1$ for $i>1$ only if the first cell of the second row of $\T$ contains $i-1$.
Thus $x_1$ appears in $W(\T)$ exactly twice.
Similarly, $x_m$ is the last letter of $W(\T)$, and $w_i=x_m$ for $i<2m$ only if the last cell of the first row of $\T$ contains $i-1$. Thus $x_m$ appears in $W(\T)$ exactly twice.
If $1<k<m$ then $w_i=w_j=x_k$ for $i<j$ only when the cell $(k-1)$ of the first row contains $i-1$ and the cell $k$ of the second row contains $j-1$. Thus $x_k$ appears in $W(\T)$ exactly twice. This proves that $W(\T)$ is a binary word.

To prove that $W(\T)$ is inversion free, consider the entries in $W(\T)$ of two letters $x_k$ and $x_\ell$ for $k<\ell$.
If $1<k<m$ then the two entries of $x_k$ are $w_i$ and $w_j$ where $i-1$ is in the cell $k-1$ of the first row of $\T$
and $j-1$ is in the cell $k$ of its second row, while the two entries of $x_\ell$ are $w_{i'}$ and $w_{j'}$ where $i'-1$ is in the cell $\ell-1$ of the first row of $\T$ and $j'-1$ is in the cell $\ell$ of its second row.
Since $\T$ is a standard Young tableau, we have $i<i'$ and $j<j'$, thus $x_k$ and $x_\ell$ is not an inversion.

The proofs for the cases $k=1$ and $\ell=m$ are similar.
\end{proof}

\begin{Prop}\label{Prop: binary snake name associated with SYT}
	The word $W(\T)$ in Definition \ref{Def: binary snake associated with SYT} is an inversion free binary snake name.
\end{Prop}
\begin{proof}
	By Lemma \ref{Lem: word associated with a SYT is inversion free}, $W(\T)$ is an inversion free binary word.
 In particular, condition $(i)$ of Definition \ref{Def: rules for snakes names} is satisfied. We are going to prove that
 condition $(ii)$ of Definition \ref{Def: rules for snakes names} is also satisfied.
 For $m=1$ the statement is true since $\T$ is empty and $W(\T)=[x_1 x_1]$, thus we may assume that $m>1$.

Note first that any subword $[x^-\cdots x^+]$ of an inversion free binary word is semi-primitive.
Let $w_i$ be an entry of $W(\T)$ where $1<i<2m$ which is the first entry of some of its letters.
Remark \ref{Remark: binary snake associated with SYT} implies that the subword $[w_1\cdots w_{i-1}]$ of $W(\T)$
contains only one entry of some letter $x$. Since $W(\T)$ is inversion free, $[x^-\cdots x^+]$ is its semi-primitive subword  containing $w_i$. The proof for the case when $w_i$ is the second entry of some letter is similar.
\end{proof}

\begin{Lem}\label{Lem: T(W(T))=T}
	Let $\T$ be a standard Young tableau of shape $(m-1,m-1)$,
and let $W=W(\T)$ be the word of length $2m$ associated with $\T$ in
Definition \ref{Def: binary snake associated with SYT},
which is an inversion free binary snake name by Proposition \ref{Prop: binary snake name associated with SYT}.
If $\T(W)$ is the standard Young tableau associated with $W$ in
Definition \ref{Def: Young Tableau associated with a binary snake} then $\T(W) = \T$.
\end{Lem}
\begin{proof}
If $w_i$ is an entry of $W$ such that $i-1$ is in the $k$-th cell of the first row of $\T$, then $i>1$ and,
by Remark \ref{Remark: binary snake associated with SYT}, the subword $[w_1\cdots w_i]$ of $W$ contains exactly $k+1$
first entries of the letters $x_1,\ldots,x_{k+1}$ of $W$. By Definition \ref{Def: Young Tableau associated with a binary snake},
the $k$-th cell of the first row of $\T(W)$ contains the same number $i-1$ as the $k$-th cell of the first row of $\T$.

If $w_i$ is an entry of $W$ such that $i-1$ is in the $k$-th cell of the second row of $\T$,
by Remark \ref{Remark: binary snake associated with SYT}, the subword $[w_1\cdots w_i]$ of $W$ contains exactly $k$
second entries of the letters $x_1,\ldots,x_k$ of $W$. By Definition \ref{Def: Young Tableau associated with a binary snake},
the $k$-th cell of the second row of $\T(W)$ contains the same number $i-1$ as the $k$-th cell of the second row of $\T$.
\end{proof}

\begin{Lem}\label{Lem: W(T(W))=W}
Let $W$ be an inversion free binary snake name of length $2m$ containing $m$ letters $x_1,\ldots,x_m$, so that their first entries in $W$ appear in the same order as their indices. Let $\T=\T(W)$ be the standard Young tableau of shape $(m-1,m-1)$ associated with $W$ in Definition \ref{Def: Young Tableau associated with a binary snake}.
If $W(\T)$ is the word associated with $\T$ in Definition \ref{Def: binary snake associated with SYT} then $W(\T)=W$.
\end{Lem}
\begin{proof}
Since $W$ and $W(\T)$ are inversion free words, second entries of all letters $x_j$ in each of them appear in the same order as their first entries, and in the same order as their indices.
In particular, the first entry of $W(\T)$ is $x_1$, same as the first entry of $W$,
and the last entry of $W(\T)$ is $x_m$, same as the last entry of $W$.

Let $w_i=x_k^-$ and $w_j=x_k^+$ be two entries of the letter $x_k$ in $W$, where $1<k<m$.
Since $w_i$ is the $k$-th first entry of a letter in $W$, $i-1$ is in the cell $k-1$ of the first row of $\T$.
Similarly, since $w_j$ is the $k$-th second entry of a letter in $W$, $j-1$ is in the cell $k$ of the second row of $\T$.
Definition \ref{Def: binary snake associated with SYT} implies that $x_k$ appears in $W(\T)$ also as its $i$-th and $j$-th entries.
The proofs for the second entry of $x_1$ and the first entry of $x_m$ are similar.
Thus all entries of these two words are the same.
\end{proof}

\begin{Teo}\label{Prop: bijection SYT and inv. free binary snake names}
	There is a bijection between the set of standard Young tableaux of shape $(m-1,m-1)$ and the set of equivalence classes of inversion free binary snake names of length $2m$, for each $m\ge 1$.
\end{Teo}

\begin{proof}
Definition \ref{Def: Young Tableau associated with a binary snake} defines the map $f$
from the set of equivalence classes of inversion free binary snake names of length $2m$ to the set of standard Young tableaux of shape $(m-1,m-1)$, and Definition \ref{Def: binary snake associated with SYT} defines a map in the opposite direction.
It follows from Lemmas \ref{Lem: T(W(T))=T} and \ref{Lem: W(T(W))=W} that these two maps are inverses of each other,
thus they are bijective.
\end{proof}

\begin{Cor}\label{Cor: binary snake names and Catalan number} \emph{(See \cite{Stanley2} p.~226 Exercise 6.19 ww, p.~230 Exercise 6.20.)}
	The number of equivalence classes of inversion free binary snake names of length $2m+2$ is the $m$-th Catalan number
$$C_m=\frac{1}{m+1}\genfrac(){0pt}{0}{2m}{m}.$$
\end{Cor}


\begin{thebibliography}{99}

\bibitem{birbrair1999local} L.~Birbrair. {\it Local bi-Lipschitz classification of 2-dimensional semialgebraic sets.} Houston J.~Math., 25 (1999), no.~3, 453--472

\bibitem{birbrair2014lipschitz} L.~Birbrair, A.~Fernandes, A.~Gabrielov, V.~Grandjean. {\it Lipschitz contact equivalence of function germs in {$\mathbb{R}^2$}.} Annali SNS Pisa, 17 (2017), 81--92, DOI 10.2422/2036-2145.201503{$\_$}014

\bibitem{birbrair2020lipschitz} L.~Birbrair, A.~Fernandes, Z.~Jelonek. {\it On the extension of bi-Lipschitz mappings.} Preprint arXiv:2001.00753 (2020)

\bibitem{birbrair2018lipschitz} L.~Birbrair, R.~Mendes. {\it Lipschitz contact equivalence and real analytic functions.} Preprint arXiv:1801.05842 (2018)

\bibitem{birbrair2018arc} L.~Birbrair, R.~Mendes. {\it Arc criterion of normal embedding.} In: Singularities and Foliations. Geometry, Topology and Applications. NBMS 2015, BMMS 2015. Springer Proceedings in Mathematics \& Statistics, v.~222. Springer, (2018)
	
	\bibitem{birbrair2000normal} L.~Birbrair, T.~Mostowski. {\it Normal embeddings of semialgebraic sets.} Michigan Math.~J., 47 (2000), 125--132

	\bibitem{fulton1997young} W.~Fulton. {\it Young tableaux: with applications to representation theory and geometry.} Cambridge University Press (1997)

	\bibitem{kurdyka1992subanalytic} K.~Kurdyka. {\it On a subanalytic stratification satisfying a Whitney property with exponent 1.} Real algebraic geometry, Springer, (1992), 316--322
	
	\bibitem{kurdyka1997distance} K.~Kurdyka, P.~Orro. {\it Distance g\'eod\'esique sur un sous-analytique.} Revista Math.~Univ.~Comput.~ Madrid, 10 (1997)
	
	\bibitem{Mostowski} T.~Mostowski. {\it Lipschitz equisingularity} Dissertationes Math. (Rozprawy Mat.), 243 (1985) 46 pp.
	
\bibitem{Parusinski} A.~Parusinski. {\it Lipschitz stratifications of subanalytic sets}, Ann.~Sci.~Ecole.~Norm.~Sup., (4) 27 (1994), 661-–696

	\bibitem{Pham} F.~Pham and B.~Teissier. {\it Fractions Lipschitziennes d’une alg\`ebre analytique complexe et saturation de Zariski.} Pr\'epublications Ecole Polytechnique No. M17.0669. Paris, 1969.  Available at http://hal.archives-ouvertes.fr/hal-00384928/fr/

	\bibitem{Stanley2} R.~Stanley. {\em Enumerative combinatorics. Vol.~2,} Cambridge University Press, Cambridge, 1999.

	\bibitem{valette2007link} G.~Valette. {\it The link of the germ of a semi-algebraic metric space.} Proc.~Amer.~Math.~Soc., 135 (2007), no.~10, 3083--3090
	
\end{thebibliography}
\end{document}